\documentclass{amsart}
\usepackage[utf8]{inputenc}
\usepackage[english]{babel}
\usepackage[left=2cm,right=2cm,top=2cm,bottom=3cm]{geometry}
\usepackage{comment}
\usepackage{graphicx}
\usepackage{dsfont}
\usepackage{amsmath, amsfonts, amssymb, mathrsfs,amsthm,mathtools}

\usepackage{silence} 
\usepackage[pdfencoding=unicode, psdextra]{hyperref}

\usepackage[dvipsnames]{xcolor}
\usepackage{float}
\usepackage{enumerate}
\usepackage{ulem} 
\usepackage[shortlabels]{enumitem}
\setlist[enumerate]{itemsep=0pt}

\usepackage{ulem}

\usepackage[hang,small]{caption}
\usepackage[]{subcaption}

\newtheorem{theorem}{Theorem}[section]
\newtheorem{proposition}[theorem]{Proposition}
\newtheorem{lemma}[theorem]{Lemma}
\newtheorem{corollary}[theorem]{Corollary}
\newtheorem{definition}[theorem]{Definition}

\theoremstyle{remark}
\newtheorem{remark}[theorem]{Remark}
\newtheorem{example}[theorem]{Example}

\newcommand{\R}{\mathbb R}
\newcommand{\N}{\mathbb N}

\newcommand{\E}{\mathbb E}

\newcommand{\cF}{\mathcal{F}}
\newcommand{\cG}{\mathcal{G}}

\newcommand{\cL}{\mathcal{L}}
\newcommand{\cH}{\mathcal{H}}
\newcommand{\cP}{\mathcal{P}}
\newcommand{\cQ}{\mathcal{Q}}
\newcommand{\cR}{\mathcal{R}}

\newcommand{\cU}{\mathcal{U}}

\newcommand{\fS}{\mathfrak{S}}
\newcommand{\one}{\mathds{1}}

\DeclareMathOperator*{\supp}{supp}

\DeclareMathOperator*{\trace}{trace}

\def\xC{{\rm C}}

\def\xLip{ {\rm Lip} }
\def\xBL{ {\rm BL} }

\def\xL{{\rm L}}

\def\xvar{{\rm Var}}
\def\xdiam{{\rm diam}}

\def\xSym{{\rm Sym}}
\def\xM{{\rm M}}

\renewcommand{\phi}{\varphi}
\renewcommand{\epsilon}{\varepsilon}
\newcommand{\Gr}{{\rm G}_{d,n}}

\title{A varifold-type estimation for data sampled on a rectifiable set}
\author{Charly Boricaud}
\author{Blanche Buet}
\address{Universit\'e Paris-Saclay, Inria DataShape, Cnrs, Laboratoire de math\'matiques d'Orsay (Orsay, France)}
\email{charly.boricaud@universite-paris-saclay.fr; charly.boricaud AT gmail.com}
\email{blanche.buet@universite-paris-saclay.fr}
\date{\today}
\keywords{Varifold, geometric inference, estimation, discrete surfaces.}
\thanks{B. Buet acknowledges support from the French National Research Agency (ANR) under grant ANR-21-CE40-0013-01 (project GeMfaceT)}
\subjclass[2010]{49Q15; 62G05}

\begin{document}

\begin{abstract}
We investigate the inference of varifold structures in a statistical framework: assuming that we have access to i.i.d. samples in $\R^n$ obtained from an underlying $d$--dimensional shape $S$ endowed with a possibly non uniform density $\theta$, we propose and analyse an estimator of the varifold structure associated to $S$. The shape $S$ is assumed to be piecewise $\xC^{1,a}$ in a sense that allows for a singular set whose small enlargements are of small $d$--dimensional measure. The estimators are kernel--based both for inferring the density and the tangent spaces and the convergence result holds for the bounded Lipschitz distance between varifolds, in expectation and in a noiseless model. The mean convergence rate involves the dimension $d$ of $S$, its regularity through $a \in (0, 1]$ and the regularity of the density $\theta$.
\end{abstract}

\maketitle

\section{Introduction}
\label{secIntro}

The so called "Manifold Hypothesis" postulates that a wide range of possibly high-dimensional data sets actually lie on a low-dimensional manifold. One can then think of point cloud data in $\R^n$ as one or several instances of a sampling $(X_1, \ldots, X_N)$ of a $d$--dimensional submanifold $S \subset \R^n$. A widely investigated question is then to {\it infer} (recover) information about $S$, topological and geometric features for instance, from the sample.
More precisely, we assume that there exists an underlying "regular" object $S$ given through a probability measure $\mu$ carried by $S$ and we furthermore assume that our data are obtained by sampling $\mu$ with $N$ points: $(X_1, \ldots, X_N) \sim \mu$ is an i.i.d. sample and our data are an instance of the empirical measure 
$$
 \mu_N = \frac{1}{N} \sum_{i = 1}^N \delta_{X_i} \: .
 $$
The specific and important case where $S$ is a $d$--dimensional submanifold of $\R^n$ and $\mu$ is the volume form of $S$ (possibly weighted by some density) has been investigated with varying degrees of formality, different assumptions on the regularity of the manifold and for various choices of metric on shapes (Hausdorff distance, Wasserstein distances, Bounded Lipschitz distance e.g.) \cite{GarciaTrillosSlepcev2015,aamariLevrard18,LevradAamari,tinarrage,Divol21,Divol22,TangYang22,stephanovitchAamariLevrard2023wassersteingansminimaxoptimal}.

Let us recall that the weak convergence of $\mu_N$ towards $\mu$ holds with probability $1$ for very general $S$ and $\mu$ (far beyond our Euclidean scope): it is known as the Glivenko-Cantelli theorem.
Moreover, in the fundamental work \cite{dud}, such a weak convergence is quantified in terms of Bounded Lipschitz distance $\beta (\mu, \mu_N)$ and Prokhorov distance.
In \cite{GarciaTrillosSlepcev2015,trillos2018error}, the authors estimate the convergence of $\mu_N$ in terms of the $\infty$--optimal transportation distance which is stronger than the aforementioned distance, and thus under the stronger assumptions that $S$ is a $d$--submanifold with curvature bounds.
In such a context, one can devise and analyse estimators of geometric quantities such as tangent space and second fundamental form, as well as estimators of $S$ itself. A very complete analysis is carried out in \cite{LevradAamari}: the authors establish general minimax bounds for the aforementioned estimators, with rates of convergence involving the size of the sample $N$, the dimension of the manifold $d$ and its order of regularity. They also evidence the need of a global assumption: in addition to ${\rm C}^k$--regularity, one has to work with submanifolds $S$ sharing a uniform lower bound on their reach to obtain convergence of the geometric estimators in strong enough topology (pointwise convergence in $S$ concerning tangent spaces or second fundamental form and Hausdorff distance for estimating $S$ itself). While the ${\rm C}^k$--manifold ($k \geq 2$) setting has been well--investigated, handling lower regularity frameworks is a natural next issue. Lower regularity may be thought in terms of lowering local parametric regularity, relaxing global structural assumptions such as reach bounds and even questionning the global manifold assumption allowing immersions, more general singular sets as well as multi--dimensional models as in the recent work \cite{AamariBerenfeld}.

In the present work, we consider the problem at hand from a measure perspective: our main purpose is to infer the so-called {\it varifold} structure $V_S = \cH^d_{| S} \otimes \delta_{T_x S}$ associated with $S$ (see Definition~\ref{dfnRectifVarifold}). The varifold $V_S$ can be defined for a smooth $d$--submanifold $S$ but also for less regular objects such as $d$--rectifiable sets, and it encodes the order $1$ structure of $S$: for a $d$--submanifold $S$, $V_S$ is a measure in $\R^n \times \Gr$ whose support is $\{ (x, T_x S) \: : \: x \in S \}$. Let us describe the concrete differences that come along with such a change of standpoint. First, while the manifold setting naturally allows to consider the tangent space as a continuous function in $S$, we only have a weaker counterpart in the varifold setting and we consequently do not expect that an associated estimator converge pointwise in $S$ but rather as measures or almost everywhere in $S$. Similarly, we do not wish to estimate $S$ for the Hausdorff distance but rather to estimate the $d$--dimensional Hausdorff measure $\cH^d_{| S}$ carried by $S$ and then the varifold structure $V_S$ both for the so called {\it Bounded Lipschitz distance} $\beta$, which is related to weak convergence of measures.
Coming along with the measure perspective is the issue of the possible non uniformity of $\mu$. Indeed,
depending on the nature of the collected data, it is not adapted to assume that $\mu$ is well distributed in $S$: such lack of uniformity can be modelled assuming that $\mu = \theta \cH^d_{| S}$ for some positive density function $\theta$ uniformly lower and upper bounded in $S$: $0 < \theta_{min} \leq \theta \leq \theta_{max} < \infty$.
Such an assumption is common in the literature, however, the usual issue that is addressed is the estimation of $\theta$ and $\mu$, whereas we are interested in estimating $\nu = \cH^d_{| S} = \frac{1}{\theta} \mu$ (similarly to what is considered in \cite{Divol22} in a manifold framework), which is a closely related yet slightly different question.
In such a case, it is important to decouple the geometric information contained in $S$ from the whole information encoded by $\mu$ and we naturally adopt a two--step approach, we first estimate $\theta$ and then we design an estimator for $\cH^d_{| S}$. Once the reconstruction of $\cH^d_{| S}$ has been carried out, we can tackle the varifold estimation. Again, we first analyse the a.e. pointwise convergence of a tangent space estimator under rectifiability assumption and we then move to the issue of estimating the whole varifold structure $V_S = \cH^d_{| S} \otimes \delta_{T_x S}$. Let us describe more precisely both the density and the tangent estimation leading to the reconstruction of the varifold structure.

\subsection{Reconstruction of the varifold structure}
\label{secIntroRecVarifold}

\subsubsection*{Density and measure estimation.}  
Note that if we were considering the case of a probability measure $\mu = \theta \cL^n$ which is absolutely continuous with respect to the Lebesgue measure $\cL^n$, we could rely on Lebesgue differentiation theorem with respect to balls: for a.e. $x$, \[
\theta(x) = \lim_{\delta \to 0_+} \frac{\mu(B(x,\delta))}{\cL^n(B(x,\delta))} = \frac{1}{\cL^n(B(0,1))} \lim_{\delta \to 0_+} \frac{\mu(B(x,\delta))}{ \delta^d} \: ,
\]
in order to estimate $\theta$. There is no equivalent property applicable to any Radon measure $\mu$ absolutely continuous with respect to $\cH^d$, however, assuming $\mu = \theta \cH^d_{| S}$ for a {\it $d$--rectifiable} set $S$ is enough to similarly obtain for $\cH^d$--a.e. $x \in S$,
\[
\theta(x) = \frac{1}{\cL^d(B(0,1))} \lim_{\delta \to 0_+} \frac{\mu(B(x,\delta))}{\delta^d} \quad \text{and also}\quad \theta(x) = \frac{1}{C_\eta} \lim_{\delta \to 0_+} \frac{\int_{B(x,\delta)} \eta \left( \frac{|y|}{\delta} \right) \: d \mu (y) }{\delta^d} = \frac{1}{C_\eta} \lim_{\delta \to 0_+} \frac{ \mu \ast \eta_\delta(x) }{\delta^d}
\]
for a smooth radial profile $\eta$.
Such an observation leads to consider a usual kernel-based approximate density (see \eqref{eqThetaN})
$\theta_\delta(x) = ( C_\eta \delta^{d})^{-1} \mu \ast \eta_\delta$ associated with the natural estimator $\theta_{\delta,N}  = ( C_\eta \delta^{d})^{-1} \mu_N \ast \eta_\delta $. We recall in Proposition~\ref{propDensityPointwiseCv} the control of the fluctuation $\E \left[ | \theta_{\delta,N}(x) -  \theta_\delta(x) | \right]$ (see \cite{Berenfeld}). Regarding the convergence of $\theta_\delta$ towards $\theta$, it holds almost everywhere in $S$ as soon as $S$ is $d$--rectifiable, however, its quantification requires to strengthen the regularity assumptions and leads to Proposition~\ref{propFinalCV}.
Once $\theta$ has been estimated, we propose an estimator $\nu_{\delta,N} \simeq \frac{1}{\theta_{\delta,N}} \mu_N$ that converges to $\frac{1}{\theta} \mu = \cH^d_{| S}$ in terms of the Bounded Lipschitz distance $\beta$ under rectifiability assumptions (see Theorem~\ref{thmBetaLocNuN}, Corollary~\ref{corNuDeltaN}) though again not in a minimax sense that also requires to strengthen the regularity class at hand as done in Theorem~\ref{thmCvHolderSplit}.

\subsubsection*{Tangent space and varifold estimation.} The very practical problem of estimating tangent spaces is a long-standing question that has already been substantially addressed. However, there are few theoretical guarantees when the set $S$ is only assumed to be rectifiable.
Relying on the Principal Component Analysis approach,
we directly consider the $\phi$--weighted (and properly renormalized) covariance matrix $\Sigma_r(x,\mu)$ in order to approximate $T_x S$ with scale parameter $r> 0$:
\[
\Sigma_r(x,\mu) = \frac{1}{C_\phi r^{d}}\int_{B(x,r)} \phi \left( \frac{y-x}{r} \right) \frac{y-x}{r} \otimes \frac{y-x}{r} \: d \mu(y) \: .
\]
We point out that in \cite{tinarrage}, such a tangent plane estimator $\Sigma_r$ has proven to be efficient in the reconstruction of the varifold structure associated with an immersed manifold $S$, obtaining explicit convergence rate for the $p$--Wasserstein distance between the exact and the reconstructed varifolds. Our purpose however differs both because of the weaker regularity we assume and because we do not consider the deterministic reconstruction issue (as done in \cite{tinarrage}) but the mean convergence reconstruction of the varifold structure. We first analyse the pointwise reconstruction of the approximate tangent space in Propositions~\ref{propCVSigmar} and \ref{propCVSigmaN},
the reconstruction of the whole varifold then requires to put density and tangent estimation together as done in Proposition~\ref{prop:CV_WrdeltaN}. Once again, one has to strengthen the regularity assumptions to obtain ``minimax convergence'' (we do not mean minimax optimal but uniform in some regularity class as explained below) for both tangent space and varifold structure estimators as done in Theorem~\ref{thmCvHolderSplit}.
It is a very important stage to reach since higher order estimators will rely on the quality of the approximation of the varifold structure. More specifically, we believe that inferring the varifold structure would allow to infer curvature information thanks to the approximations that have been proposed in \cite{BuetLeonardiMasnou} and \cite{BuetLeonardiMasnou2}.

\subsubsection*{Piecewise Hölder regularity class and main result.} As mentioned above, the convergence results stated in Sections~\ref{secInferenceFramework}--\ref{secMeasureEstimator}--\ref{secVarifoldLikeEstimator} do not fit a minimax convergence setting. Let us be more explicit: let 
\[
 \cQ = \{ \mu = \theta \cH^d_{| S} \: : \: S, \: \theta \text{ satisfy  \hyperref[hypH1]{$(H_1)$}--\hyperref[hypH2]{$(H_2)$}--\hyperref[hypH3]{$(H_3)$}} \}
\]
be a regularity class of measures depending on $d, C_0, \theta_{min/max}$. In short, $\mu \in \cQ$ if it is $d$--Ahlfors regular with constant $C_0$, $d$--rectifiable and the density $0 < \theta_{min} \leq \theta \leq \theta_{max} < \infty$ is upper and lower bounded. Let $\hat{\vartheta}_N = \hat{\vartheta}(X_1, \ldots, X_N)$ be an estimator of some quantity $\vartheta (\mu)$ depending of $\mu$.
A convergence result for $\hat{\vartheta}$ fits the minimax setting in the regularity class $\cQ$ if the convergence is uniform with respect to the class $\cQ$:
\[
 \sup_{\mu \in \cQ} \E \left[ \ell ( \hat{\vartheta}_N ,  \vartheta (\mu) ) \right] \xrightarrow[N \to \infty]{} 0 \: ,
\]
for some loss $\ell$ between $\hat{\vartheta}_N$ and $\vartheta$. A natural question is then to find the best estimators possible and investigate both (asymptotic $N \to \infty$) lower and upper bounds for the so-called minimax risk that consider all possible estimators $\hat{\vartheta}_N$:
\[
 \cR_N (\cQ) = \inf_{\hat{\vartheta}_N} \sup_{\mu \in \cQ} \E \left[ \ell ( \hat{\vartheta}_N ,  \vartheta (\mu) ) \right] \: .
\]
An estimator $\hat{\vartheta}_N$ is {\it minimax optimal} over $\cQ$ if its asymptotical convergence rate is $\cR_N (\cQ)$ up to a constant.
In Sections~\ref{secInferenceFramework}--\ref{secMeasureEstimator}--\ref{secVarifoldLikeEstimator}, the convergence results we obtain do not fit the minimax setting since they are not uniform over the regularity class $\cQ$. More precisely, we are only able to prove results of the form: for a given $\mu \in \cQ$ and adapting some parameter (denoted by $\delta_N$ thereafter) we have $\E \left[ \ell ( \hat{\vartheta}_N ,  \vartheta (\mu) ) \right] \xrightarrow[N \to \infty]{} 0$. In Sections~\ref{secUniformPWHolder} and~\ref{secSplit}, we consequently strengthen the assumptions on $\mu$ and consider instead of $\cQ$, the regularity class:
\[
\cP = \{ \mu = \theta \cH^d_{| S} \: : \: S, \: \theta \text{ satisfy  \hyperref[hypH1]{$(H_1)$} to \eqref{hypH7}} \} 
\]
that now depends on $d, \widetilde{C_0}, \theta_{min/max}$ and $C = \max(C_{\theta, sg}, C_{S,sg}), R = \min(R_{\theta, sg}, R_{S,sg})$ as introduced in Definitions~\ref{dfnPwHolderS} and \ref{dfnPwHolderTheta}. Loosely speaking, we consider in $\cP$ measures of the form $\mu = \theta \cH^d_{| S}$ for which $\theta$ and $S$ are uniformly piecewise respectively $\xC^{0,b}$ and $\xC^{1,a}$ up to a singular set $\fS$ of zero $\mu$--measure, and we assume moreover that $\fS$ has co-dimension at least $1$ in $S$ in the sense that $\fS$ is a union of $l$--Ahlfors regular sets for $0 \leq l \leq d-1$. We therefore in particular allow for $S$ to be a smooth manifold with boundary or to be piecewise smooth up to auto-crossings and junctions along a set $\fS$ that is reasonably small in the aforementioned sense. We introduce and investigate the convergence of estimators in the regularity class $\cP$ that we refer to as the {\it piecewise Hölder regularity class}. The main result we obtain is Theorem~\ref{thmCvHolderSplit} that establishes a convergence result for an estimator $\widehat{V}_N$ of the varifold structure $W_S = \cH^d_{| S} \otimes \delta_{\Pi_{T_x S}}$ (see \eqref{eqWS}) that is uniform with respect to the $\mu \in \cP$ (with $S = \supp \mu$) when the loss $\ell = \beta$ is the so-called {\it Bounded Lipschitz distance} (Definition~\ref{dfnBeta}):
\[
 \sup_{\mu \in \cP} \E \left[ \beta ( \widehat{V}_N , W_S ) \right] \leq M \: N^{-\frac{\min(a,b)}{d + 2 \min(a,b)}}
\]
where the constant $M$ only depends on $d, n, \widetilde{C_0}, \theta_{min/max},C = \max(C_{\theta, sg}, C_{S,sg}), R = \min(R_{\theta, sg}, R_{S,sg})$ that are the parameters of the regularity class $\cP$.
In other words, we obtain an explicit upper bound of the minimax risk $\cR_N (\cP)$ concerning the inference of the varifold structure with respect to the Bounded Lipschitz distance $\beta$.
According to us, a very important point achieved with such a result is the following: it proves that it is possible to obtain minimax convergence results for reconstructing the order $1$ information of $S = \supp \mu$ despite the presence of a singular set, hence beyond the manifold hypothesis. A crucial point is to consider a loss, in our case $\ell = \beta$, associated with a topology weaker than a pointwise loss.

\medskip Let us now comment on some achievements, limitations and perspectives of our study.

\subsection{Comments and perspectives}
\label{SecCommentsPerspectives}

\subsubsection*{Reach assumption.} In addition to local smoothness requirements, it has been evidenced in the literature the necessity of global regularity assumptions in order to obtain minimax inference results in at least $\xC^2$ manifolds models, we refer to \cite{LevradAamari} as well as \cite{Berenfeld} (Thm 2.6 in Section 2.5) in the case of the pointwise density inference where the necessity of lower bounding the {\it reach} of the manifold is established.
Let us recall that the reach of a compact set $K\subset \R^n$ is defined as
\[ 
 {\rm reach}(K)=\sup\{r\geq 0 \:|\: \forall x \in \R^n, d(x,K)\leq r \implies \exists ! y \in K , \, d(x,K)= |x-y| \} \: .
\]
Assuming that ${\rm reach}(S) \geq \kappa > 0$ for instance ensures that $S$ does not have auto-intersection in a quantative way: $S$ has a $\kappa$--neighborhood in which no piece of $S$ that would be far in terms of geodesic distance can go through.
Let us comment on the fact that our results are however obtained without explicitly requiring any lower bound on the {\it reach} of $S$.
First of all, as already mentioned above, the results of Sections~\ref{secInferenceFramework},~\ref{secMeasureEstimator} and~\ref{secVarifoldLikeEstimator} do not meet the usual minimax framework: indeed,
we obtain convergence results of the following type, for instance considering the density estimator $\theta_{\delta_N,N}$: one can chose $\delta_N \to 0$ such that for each $\mu \in \cQ$ and for $\mu$--a.e. $x$,
\begin{equation} \label{eqRiskUpperBound}
 \E \left[ | \theta_{\delta_N,N}(x) - \theta(x) | \right] \xrightarrow[N \to +\infty]{} 0 \: ,
\end{equation} 
but the convergence holds without any uniform bounds. %
This first observation explain the absence of lower bounded reach assumption in those three Sections~\ref{secInferenceFramework},~\ref{secMeasureEstimator} and~\ref{secVarifoldLikeEstimator}. However, in Sections~\ref{secUniformPWHolder} and \ref{secSplit}, we obtain uniform convergence results with respect to the piecewise Hölder regularity class $\cP$ as stated in Theorem~\ref{thmCvHolderSplit} and Proposition~\ref{propFinalCV}.
Our understanding is that such a result illustrates that the necessity of lower bounding the reach holds when establishing pointwise convergence results but not for ``weaker'' convergence like the Bounded Lipschitz or the $\xL^1$ convergence (see also \eqref{eqThetaNL1D}), for which it is sufficient to control the measure of the bad set $\fS$ in some quantitative way. We believe that it is an interesting starting point in finding a consistent inference model allowing for low-regularity and singularities in $S$, beyond the $\xC^2$--manifold model.
We also obtain a pointwise result assuming that the singular set $\fS$ is empty (see \eqref{eqThetaNHolder} in Proposition~\ref{propFinalCV}) in which we recover known minimax pointwise rate for the density $\theta$ of order $N^{-\frac{\min(a,b)}{d + 2 \min(a,b)}}$, see \cite{Berenfeld}. Note that in this particular case, though we do not explicitly assume any lower bound on the reach, Definition~\ref{dfnPwHolderS} however requires that $S$ is a $d$--manifold (only in this particular case where $\fS = \emptyset$) and furthermore a $\xC^{1,a}$ graph in any ball of radius less than $R$ with uniform Hölder constant. On one hand, such assumptions prevent that parts of $S$ could get arbitrarily close in the ambient space while distant in the instrinsic metric of $S$ (induced by the ambient one) similarly to what a reach bound would ensure. On the other hand, we also emphasize that we assume that the density $\theta$ is Hölder with respect to the ambient metric (see Definition~\ref{dfnPwHolderTheta}), which is different from the model investigated in \cite{Berenfeld} where the Hölder regularity of the density is assumed in terms of geodesic distance in $S$, which makes sense since $S$ is a regular manifold in their work.

\subsubsection*{Minimax lower bound and optimal rates.}
To the best of our knowledge, there is no existing result to infer the varifold structure and more generally, there are few inference results in such a low regularity framework. In particular, lower bounds for the minimax rates are yet to established. Nonetheless, we first point out that our analysis provides the same rate for the estimation of $\nu = \cH^d_{| S}$ as for the estimation of the varifold $W_S$ (see Theorem~\ref{thmCvHolderSplit}), yet, $W_S = \cH^d \otimes \delta_{\Pi_{T_x S}}$ involves the order $1$ structure: we expect that $\nu$ can be inferred with a better rate than $W_S$. Note that in the case where $S$ is at least a $\xC^2$ manifold with lower bounded reach,  
\cite{Divol22} establishes minimax rates concerning the inference of $\nu$ in terms of Wasserstein distance, assuming Besov regularity $B_{p,q}^b (S)$ for the density and similar lower and upper bound $\theta_{min/max}$. In particular, for $d \geq 3$, they establish a minimax rate of order $N^{-\frac{1 +b}{d + 2 b}}$ that is, at least loosely speaking, of order $\sim \delta_N N^{-\frac{\min(a,b)}{d + 2 \min(a,b)}}$ and we observe a gain of the factor $\delta_N$ when compared to the rate we obtain in Theorem~\ref{thmCvHolderSplit}. From a technical point of view, a key point in their proof to achieve such a rate is to perform the estimation of $\theta$ in Sobolev ${\rm H}^{-1}(S)$--norm rather than pointwisely. Unfortunately, the counterpart of this approach is not clear in our regularity framework.
As for the estimation of the order $1$ structure, we are not aware of results beyond (uniform) pointwise estimator of the tangent space. In \cite{LevradAamari}, a $\xC^k$--manifold regularity model is investigated and concerning the pointwise estimation of tangent spaces, a lower bound for the minimax risk of order $N^\frac{k-1}{d}$ is established (in the noise free model), and an almost optimal (up to a logarithmic factor) estimator is given. Formally replacing $k$ with $1 + a$ in our regularity class, we can compare our convergence rate $N^\frac{\min(a,b)}{d + 2 \min(a,b)}$ with $N^\frac{k-1}{d} \sim N^\frac{a}{d}$ that indicates some room for improvement if the known $\xC^k$--manifolds minimax rates consistently extend for $1 < k < 2$, which is yet another point that emphasizes the importance of searching for minimax lower bounds for low regularity models. We note that recent results establish minimax rates of order $N^\frac{a + \gamma}{d + 2 a}$ ($1 \leq \gamma \leq a +1$) for the estimation of $\nu$ (though not $\mu$ up to our understanding) with respect to a distance $d_{\mathcal{H}^\gamma}$ similar to the bounded Lipschitz distance but adding a Hölder condition on the first order derivative of test functions. Such minimax rates are valid in a manifold framework and it would be worth understanding wether such manifold regularity can be relaxed to establish minimax rates in a piecewise manifold regularity framework similar to the one we investigate.

\subsubsection*{Impact of noise and higher order structure inference.} We restricted our setting to an idealistic noise free estimation to carry out the analysis of a varifold estimator. The natural and important next step is to include noise, for instance through an additive noise model as suggested in \cite{aamariLevrard18} (for the estimation of $S$ through a simplicial complex and in terms of Hausdorff distance). It would also be important to understand the respective impacts of tangential and normal component of such an additive noise: loosely speaking, we expect that the tangential perturbations are connected to the uniformity of the sampling while the normal perturbations directly affect the geometry of $S$. To this end, it is for instance possible to explore a ``tubular noise'' model, as already done in \cite{LevradAamari} for the $\xC^k$--model.
An important motivation to carry out the analysis of such kernel-based estimators in a low regularity framework roots in \cite{BuetLeonardiMasnou,BuetLeonardiMasnou2} where the authors introduce approximations of mean curvature and more generally second fundamental form based on varifolds theory. They show the convergence of such approximations and some partial stability involving localized Bounded Lipschitz distances under low regularity assumptions compatible with the piecewise Hölder regularity class $\cP$ introduced 
in Section~\ref{secUniformPWHolder}. However, in \cite{BuetLeonardiMasnou,BuetLeonardiMasnou2}, the deterministic setting is a limitation to obtain more tractable convergence rates and we hope that transferring the issue into a statistical inference setting would lead to more explicit bounds. We expect that in the case of the mean curvature vector estimation, tangential noise would produce tangential error, and the mean curvature vector being normal to $S$, it is then possible to deal with such error as long as we have a good estimation of the tangent space: concretely projecting onto the normal space, as implemented in \cite{BuetLeonardiMasnou}. In this particular case at least, it is relevant to consider noise model in which it is possible to split tangential and normal components. 

\subsubsection*{Unknown parameters of the estimators.} We finally underline that though uniform in the regularity class $\cP$:
\begin{enumerate}[$-$]
 \item The parameter $\tau$ used to define the estimators (see \eqref{eqPhiChi}, Proposition~\ref{nudel} and Remark~\ref{eqtauchoice}) depends on $d$, $C_0$ and $\eta$. Though $\eta$ is known, and $d$ might be known in some cases, $C_0$ has to be estimated.
 \item The definitions of our estimators also rely on the choice of the parameter $\delta_N$ of order $N^{\frac{1}{d+2\min(a,b)}}$, and $a$ and $b$ are not known either in general.
\end{enumerate}

\subsection{Organisation of the paper}

The paper can be divided into $3$ parts: Sections~\ref{secBasic} and~\ref{secInferenceFramework} are preliminary sections. Sections~\ref{secMeasureEstimator} and~\ref{secVarifoldLikeEstimator} investigate the varifold inference issue in the regularity class $\cQ$ \eqref{eqClassQ} in which convergence comes without uniform bounds, see Corollary~\ref{coroCvWrdeltaNW}. Sections~\ref{secUniformPWHolder} and~\ref{secSplit} then address the issue in the class $\cP \subset \cQ$ \eqref{eqClassP}: strengthening the regularity setting allows to obtain uniform convergence as stated in the main result Theorem~\ref{thmCvHolderSplit}.

\section*{Notations}
We fix $n \in \N$, $n \geq 1$ and $d \in \R$, $0 < d \leq n$.
\begin{itemize}
\item In Sections~\ref{secBasic} to \ref{secMeasureEstimator}, $d$ is real, unless otherwise specified. In Sections~\ref{secVarifoldLikeEstimator} to \ref{secSplit}, $d$ is an integer. Generally speaking, rectifiability and $d$--varifold structure require $d$ to be an integer, while $d$--Ahlfors regularity is well-defined for $d$ real.
\item We use a generic positive constant $M$ whose default dependency is given at the beginning of each sections (in Remark~\ref{remkCstMSecTgt},~\ref{remkCstMSecHolder} and~\ref{remkCstMSecSplitt}).
\item Given  $x\in \R^{n}$ and $r>0$, we set $B(x,r) = \left\lbrace y\in \R^{n} \, | \, |y-x| < r \right\rbrace$.

\item For $\delta >0,$ the open $\delta$-neighbourhood ($\delta$-thickening) of $A \subset \R^n$ is
\begin{equation}
\label{eqThick}
A^\delta = \{ x\in \R^n : d(x,A)<\delta \}=\bigcup\limits_{x\in A} B(x,\delta) \: .
\end{equation} 

\item $\cL^n$ is the $n$--dimensional Lebesgue measure and $\omega_n = \cL^n (B_1(0))$.

\item $\cH^d$ is the $d$--dimensional Hausdorff measure in $\R^{n}$.

\item ${\rm M}_n(\R)$ is the space of square matrices with real entries and of size $n$. We consider $\| \cdot \|$ the matrix (operator) norm associated with the euclidean norm in $\R^n$.

\item $\xSym(n) \subset {\rm M}_n(\R)$ is the subspace of symmetric matrices and $\xSym_+(n) \subset \xSym(n)$ is the set of positive semi-definite matrices. 


\item ${\rm P}_{d,n} = \left\lbrace A \in \xSym(n) \: : \: A^2 = A \text{ and } \trace A = d \right\rbrace$ is the compact subset of rank--$d$ orthogonal projectors.

\item $\Gr$ denotes the Grassmannian of $d$-dimensional vector subspaces of $\R^n$:
\[
\Gr = \{ d\text{--dimensional vector subspace of } \R^n \} \: .
\]
A $d$-dimensional subspace $T$ will be often identified with the orthogonal projection onto $T$, denoted as $\Pi_{T} \in {\rm M}_n(\R)$. $\Gr$ is equipped with the metric $d(T,P) = \Vert \Pi_T - \Pi_P \Vert$ so that
\[
\begin{array}{rccl}
i : & \Gr & \rightarrow & {\rm P}_{d,n} \\
    &  T &  \mapsto    & \Pi_T
\end{array} \quad \text{is a bijective isometry.}
\]

\item $\xC_c(X)$ is the space of continuous real--valued and compactly supported functions defined on a topological space $X$ 

\item $\xLip (X)$ is the space of real--valued Lipschitz functions $f$ defined on a metric space $(X,\delta)$ with Lipschitz constant $\xLip(f)$.

\item Given $k \in \N$, $k \geq 1$, $\xC_c^k (\Omega)$ is the space of real--valued functions of class $\xC^k$ with compact support in the open set $\Omega \subset \R^n$.


\item The constant $C_0 \geq 1$ (respectively $\widetilde{C_0} \geq 1$) is a Ahlfors regularity constant for $\mu$ (respectively $\nu$), see Definition~\ref{dfnAhlfors}.

\end{itemize}

\section{Preliminary notions around rectifiability and varifolds}
\label{secBasic}

We recall some elementary facts concerning Radon measures, Ahlfors regularity, rectifiability and varifolds. We restrict ourselves to what is necessary to the understanding of the paper and we refer to \cite{mattila}, \cite{evg}, \cite{ambrosio2000fbv} and \cite{simon} for more details.

\subsection{Radon measures and Ahlfors regularity} 
\label{BNM}

We restrict ourselves to the case of Radon measures in $\R^n$, $\R^n \times \Gr$ or $\R^n \times \xSym_+(n)$, whence $X =\R^n$, $X = \R^n \times \Gr$ or $X = \R^n \times \xSym_+(n)$ hereafter.

\begin{definition}[Radon measure]
We say that $\mu$ is a Radon measure on $X$ if $\mu$ is a locally finite Borel measure.
\end{definition}
We know from Riesz representation theorem that Radon measures can alternatively be defined as continuous linear form on $\xC_c(X)$, leading to the so called {\it weak star} convergence of Radon measures.
\begin{definition}[Weak star convergence of Radon measures]
Let $\mu,(\mu_k)_{k\in \mathbb{N}}$ be Radon measures on $X$. We say that $(\mu_k)_{k\in\mathbb{N}}$ weak star converges to $\mu$, and we write $\mu_k \stackrel{\ast}{\rightharpoonup}\mu$ if
\[
\forall f \in \xC_c(X), \quad \int_X f \: d \mu_k \xrightarrow[k \to \infty]{} \int_X f \: d \mu \: .
\]
\end{definition}
We also introduced the following distance among Radon measures, that sometimes is referred to as flat distance in the context of varifolds.
\begin{definition}[Bounded Lipschitz distance]
\label{dfnBeta}
Let $\lambda_1, \lambda_2$ be two Radon measures in $X$, then
$$\beta(\lambda_1,\lambda_2)=\sup\left\{ \left|\int_{X} f \: d \lambda_1 - \int_{X} f \: d \lambda_2 \right| : f\in \xC_c(X,\mathbb{R}), \|f \|_\infty\leq 1, \xLip(f) \leq 1 \right\} $$
defines a distance in the space of Radon measures.
\end{definition}
We will also consider a localized version that allows to compare measures in a given open set $D$:
\begin{definition}[Localized Bounded Lipschitz distance] \label{dfnBetaLoc}
Let $\lambda_1$, $\lambda_2$ be Radon measures in $X$ and let $D \subset X$ be an open set, we consider
\[
\beta_D(\lambda_1,\lambda_2)=\sup\left\{ \left|\int_{X} f \: d \lambda_1 - \int_{X} f \: d \lambda_2 \right| : f\in {\xC}_c(X ,\mathbb{R}), \|f \|_\infty\leq 1, {\xLip}(f) \leq 1, \supp f \subset D \right\} .
\]
In the case where $X = \R^n \times \xSym_+(n)$, we only localize with respect to the spatial part, that is considering $\beta_{D \times \xSym_+(n)}$ that we abusively denote by $\beta_D$.
\end{definition}
Such a localized version is symmetric and satisfies the triangular inequality, yet note that $\beta_D (\lambda_1, \lambda_2) = 0$ only implies that $\lambda_1 = \lambda_2$ in the open set $D$ thus defining a distance in the space of Radon measures in $D$.

The following notion of $d$--Ahlfors regularity for a measure $\mu$ requires the measure $\mu(B(x,r))$ of balls centered at $x \in \supp \mu$ to be comparable with the $d$--volume $\omega_d r^d$ of such a ball.

\begin{definition}[Ahlfors regularity] \label{dfnAhlfors}
Let $\mu$ be a Radon measure in $\mathbb{R}^n$, for $d>0$, we say that $\mu$ is $d$--Ahlfors regular (or simply $d$--Ahlfors) if
$\exists C_0 \geq 1$ such that $\forall x \in \supp ( \mu )$, $\forall r \in (0, \xdiam ( \supp \mu)]$,
\begin{equation}
\frac{1}{C_0}r^d\leq \mu( B(x,r)) \leq C_0r^d \: .
\label{eqAhlfors}
\end{equation}
Such a constant $C_0 \geq 1$ will be referred as a Ahlfors regularity constant for $\mu$. We also say that a closed set $E \subset \R^n$ is $d$--Ahlfors regular if $\cH^d_{| E}$ is $d$--Ahlfors regular.
\end{definition}

\begin{remark}[Finite Ahlfors regular measure]
\label{remkAhlfors}
Let $\mu$ be a $d$--Ahlfors regular measure in the sense of Definition~\ref{dfnAhlfors} above and furthermore assume that $\mu$ is finite (we will assume that $\mu$ is a probability measure later in the paper). Then,
\begin{enumerate}
 \item the measure $\mu$ is compactly supported. Indeed, assume by contradiction that $\xdiam( \supp \mu) = \infty$, then for all $r > 0$ and for some $x \in \supp \mu$,
 \[
  C_0^{-1} r^d \leq \mu(B(x,r)) \leq \mu(\R^n) < \infty \: ,
 \]
which is impossible letting $r \to + \infty$.
 \item It is equivalent (up to adapting $C_0$) to require \eqref{eqAhlfors} only for small radii $r \in (0, r_0]$ with $0 < r_0 < R := \xdiam ( \supp \mu)$. Indeed, let $x \in \supp \mu$ then $R < \infty$ thanks to $(i)$ and thus, for $r_0 < r \leq R$, one has
\[
0 < C_0^{-1} \frac{r_0^d}{R^d} \leq \frac{ \mu( B(x,r_0)) }{R^d} \leq \frac{ \mu( B(x,r)) }{r^d} \leq \frac{ \mu( B(x,R)) }{r_0^d} < +\infty \: .
\]
\end{enumerate}
\end{remark}

\begin{example}[$d$--Ahlfors measures.]
\label{AhlParam}
Examples of $d$--Ahlfors regular measures include Lebesgue measure on a $d$--subspace, $d$--Hausdorff measure on a smooth closed $d$--submanifold. More generally, $d$--Ahlfors regularity is preserved through bi--Lipschitz mappings so that Lipschitz graphs are Ahlfors regular. Yet, there are also less regular examples:
for instance, the $d$--Hausdorff measure on the Koch snowflake is $d$--Ahlfors regular for $d = \log_3 4$. Even worse, the ``$4$--corners Cantor
set'' (see \cite{David1993AnalysisOA} chapter 1) obtained by dividing a unit square into $4 \times 4$ equally sized squares, deleting all but the $4$ corners squares and iterating the process is known to be $1$--Ahlfors regular though purely--$1$--unrectifiable. Ahlfors regularity conveys information on the dimensionality of an object but is not restrictive in terms of smoothness. Another interesting example (rather counter-example) is given by the following highly oscillating curve
$\mathcal{C}=\left\lbrace \left(x,\sin\left(\frac{1}{x}\right)\right) \: : \:  x\in (0,1]\right\rbrace$ which is $\xC^1$ while $\cH^1_{|\mathcal{C}}$ is not $1$--Ahlfors regular. Note that in this example $\mathcal{C}$ is not closed and its closure $\overline{\mathcal{C}} = \mathcal{C} \cup \{0\} \times [0,1]$ is not a $\xC^1$ curve. On the other hand, the graph $\left\lbrace \left(x, f(x) \right) \: : \:  x\in K \right\rbrace$ of a $\xC^1$ function $f : K \subset \R^d \rightarrow \R^{n-d}$ over a compact set $K$ is $d$--Ahlfors regular (since $f$ is Lipschitz in $K$).
We refer to \cite{DavidSemmes1997BrokenDreams} for additional examples and properties connected to $d$--Ahlfors regularity.
\end{example}

The following proposition (Proposition~\ref{propPackingPartition}) gives a bound on the number of balls with common diameter $\delta$ needed to cover the support of a $d$--Ahlfors measure. Such estimates will be crucial in the proof of Proposition~\ref{propBetaLoc}. For the sake of clarity, we give the proof of Proposition~\ref{propPackingPartition} and we refer to \cite{mattila} (Section 5) for further details.

\begin{proposition} \label{propPackingPartition}
Let $\mu$ be a Radon measure in $\mathbb{R}^n$, $S= \supp (\mu)$. Assume that there exists $ d >0, C_0\geq 1$
such that $\forall x \in S,$ $\forall 0<r< \xdiam(S)$,
\begin{equation}
\label{alm}
  \mu(B(x,r))\geq C_0^{-1}r^d \: . 
\end{equation}
Then, for all bounded set $B\subset \mathbb{R}^n$ and for all $\delta > 0$, the minimal number $m(B \cap S, \delta)$ of sets of diameter smaller than $\delta$ needed to form a partition of $B\cap S$ satisfies
\begin{equation} \label{eqmPartition}
m(B\cap S,\delta)\leq 4^dC_0\delta^{-d}\mu \left( B^{\frac{\delta}{4}}\right)\: .
\end{equation}
\end{proposition}

\begin{remark}
We note that \eqref{alm} is in particular satisfied by $d$--Ahlfors measures
\end{remark}

\begin{proof} 
Given a bounded set $A \subset \R^n$ and $\delta > 0$, $m(A,\delta)$ is the minimal number of sets of diameter smaller than $\delta$ needed to form a partition of $A$ and we additionally introduce:
\begin{enumerate}[$\bullet$]
 \item the packing number $P(A,\delta)$ as the greatest number of disjoint balls of radius $\delta$ with center in $A$,
 \item the smallest number of balls with radius $\delta$ needed to cover $A$, denoted by $V(A,\delta)$.
\end{enumerate}
{\bf Step $1$}: We will prove that 
\begin{equation*} \label{eqPackingPartitionNb}
 m(A,\delta) \leq V \left(A , \frac{\delta}{2} \right) \leq P \left(A , \frac{\delta}{4} \right) \: .
\end{equation*}
Indeed, denoting $k = P(A,\delta)$, let $B(x_1,\delta),\dots ,B(x_k,\delta)$ be disjoint balls with centers $x_1,\dots,x_k \in A$.
Suppose that $V(A,2\delta)>k$ then $\bigcup_{i=1}^k B(x_i,2\delta)$ does not cover $A$ and thus there exists $x\in A \setminus \bigcup_{i=1}^k B(x_i,2\delta)$.
Consequently $B(x,\delta), B(x_1,\delta),\dots,B(x_k,\delta)$ are $k+1$ disjoint balls of radius $\delta$ with center in $A$, that is impossible by definition of $k$. We conclude that $V(A,2\delta)\leq P(A,\delta)$.

\noindent Let now $\displaystyle k^\prime = V \left(A , \frac{\delta}{2} \right)$ and $B_1,\dots,B_{k^\prime}$ be balls of radius $\displaystyle \frac{\delta}{2}$ such that $\displaystyle A\subset \bigcup \limits_{j=1}^{k^\prime} B_j$ .
We can take $A_1=A\cap B_1$ then $A_2=A\cap B_2\setminus A_1$ up to $A_{k^\prime}=(A\cap B_{k^\prime})\setminus{\bigcup\limits_{j=1}^{k^\prime-1}}A_j$
Then $A_1,\dots,A_{k^\prime}$ form a partition of $A$ with sets of diameter smaller than $\delta$ and $m(A,\delta) \leq k^\prime$.

\smallskip
\noindent {\bf Step $2$}: Let $B \subset \R^n$ be a bounded set, then $\displaystyle P(B \cap S, \delta) \leq C_0 \mu(B^\delta)\delta^{-d}$.
Indeed, let $k=P(B\cap S,\delta)$ and $B_1,\dots,B_k$ be disjoint balls of radius $\delta$ and centered in $B\cap S$. Then for $i=1,\dots,k$, \eqref{alm} gives $\mu(B_i)\geq C_0^{-1}\delta^d$ and therefore
\[
k \: C_0^{-1}\delta^d\leq \sum_{i=1}^k \mu(B_i)= \mu \left( \bigsqcup\limits_{i=1}^k B_i \right) \leq  \mu(B^\delta)
\quad \Rightarrow \quad P(B\cap S , \delta)=k\leq C_0 \mu(B^\delta)\delta^{-d} \: .
\]
Step $1$ and Step $2$ lead to \eqref{eqmPartition}.
\end{proof}

\subsection{Rectifiability}

We first recall some classical definitions and results on rectifiability and we refer to \cite{ambrosio2000fbv} for more details.
\begin{definition}[$d$--rectifiable set]
Let $S \subset \R^n$ be a Borel set, we say that $S$ is countably $d$--rectifiable if there exists a countable family $(f_i)_{i\in \mathbb{N}}$ of Lipschitz maps from $\mathbb{R}^d$ to $\mathbb{R}^n$ such that
$$\mathcal{H}^d\left(S\setminus \bigcup_{i\in \mathbb{N}}f_i(\mathbb{R}^d)\right)=0 \: .
$$
We say that $S$ is $d$--rectifiable if moreover $\cH^d(S)< \infty$ (which will be the case hereafter).
\end{definition}
\noindent Note that the above notion of rectifiability are referred to as (countable) $\cH^d$--rectifiability in \cite{ambrosio2000fbv}, Definition~2.57.
We then define (non-negative) rectifiable measures as those carried by rectifiable sets:
\begin{definition}[Rectifiable measures, \cite{ambrosio2000fbv} Definition~2.59]
Let $\mu$ be a Radon measure in $\mathbb{R}^n$. We say that $\mu$ is $d$--rectifiable if there exist a countably $d$--rectifiable set $S \subset \R^n$ and a Borel function $\theta : S \mapsto \R_+$ such that $ \mu=\theta\mathcal{H}_{|S}^d$.
\end{definition}
In other words, a $d$--rectifiable set $S$ is included in a countable union of $d$--dimensional Lipschitz graphs up to a $\cH^d$--negligible set. Then recalling that Lipschitz functions in $\R^d$ are a.e. differentiable, it is natural (to try) to define tangent planes a.e on rectifiable sets. Yet, it is crucial to have a notion that does not depend on the choice of
a covering of $S$ with Lipschitz graphs:
\begin{definition}[Approximate tangent plane, \cite{ambrosio2000fbv} Definition~2.79]
\label{ATP} 
Let $\mu$ be a Radon measure in $\R^n$ and let $x\in \mathbb{R}^n$. We say that $\mu$ has {\em approximate tangent space} (or plane) $P\in G_{d,n}$ with multiplicity $\theta\in \mathbb{R}_+$ at $x$, if for all $f \in \xC_c(\R^n)$,
$$\lim\limits_{r\to 0_+} r^{-d}\int_{\mathbb{R}^n} f \left(\frac{y-x}{r}\right) \: d\mu(y) = \theta(x) \int_P f(y)d\mathcal{H}^d(y) \: .
$$
\end{definition}

\begin{proposition}[A.e. existence of approximate tangent space for rectifiable measures, \cite{ambrosio2000fbv} Proposition~2.83]
\label{propTgtSpaceToRectifMeasure}
Let $\mu = \theta \cH^d_{| S}$ be a $d$--rectifiable measure in $\R^n$. Then, for $\cH^d$--a.e. $x \in S$, 
\begin{enumerate}[$\bullet$]
 \item $\mu$ admits an approximate tangent space with multiplicity $\theta(x)$, denoted by $T_x S$,
 \item in particular $\displaystyle \theta(x) =
\lim_{r \to 0_+} \frac{\mu(B(x,r))}{\omega_d r^d}$.
\end{enumerate}
\end{proposition}
\noindent Note that Proposition~2.83 in \cite{ambrosio2000fbv} actually characterize the rectifiability of $\mu$ through the existence of approximate tangent space, but it is not necessary in this paper. Note also that such an approximate tangent space to $\mu = \theta \cH^d_{| S}$ actually does not depend on $\theta$ (see \cite{ambrosio2000fbv} Proposition 2.85) and we will use the notation $T_x S$ for the approximate tangent space.

\subsection{Varifolds}
\label{secVarifolds}

We introduce in this section some basic definitions concerning $d$--varifolds in $\R^n$. Our purpose is not to describe the formalism of varifolds in an abstract setting, for instance we do not define the class of rectifiable $d$--varifolds but we only explain how to associate a $d$--varifold $V_S$ with a $d$--rectifiable set $S$. We refer to \cite{simon} for a comprehensive reference. We also point out \cite{Menne2017} for a short introduction through examples. We recall that varifolds have been introduced and investigated to tackle the existence of minimal surfaces and we refer to the seminal paper of \cite{allard}. 

First of all, note that there is a one-to-one correspondence between $\Gr$ and the subset of rank--$d$ orthogonal projectors ${\rm P}_{d,n}$ of $\xM_n(\R)$ via $i : T \in \Gr \mapsto \Pi_T$, where we recall that $\Pi_T \in \xM_n (\R)$ is the orthogonal projector onto the $d$--vector subspace $T$. We endow $\Gr$ with the distance $d$ satisfying for $T_1$, $T_2$ in $\Gr$, $d(T_1, T_2) = \| \Pi_{T_1} - \Pi_{T_2} \|$, in other words, with such a metric in $\Gr$, $i$ is an isometry.

\begin{definition}[$d$--varifold in $\R^n$]\label{dfnVarifold} 
A $d$--varifold in $\R^n$ is a non-negative Radon measure in $\R^n \times \Gr$. 
\end{definition}

\noindent Beyond such a general definition, there are natural $d$--varifolds associated with geometric objects such as $d$--submanifolds and $d$--rectifiable sets:

\begin{definition}[Rectifiable $d$--varifold associated with a $d$--rectifiable set]\label{dfnRectifVarifold}
Let $S \subset \R^n$ be a $d$--rectifiable set. We can define the \emph{rectifiable $d$--varifold} $V_S= \mathcal{H}^d_{| S} \otimes \delta_{T_x S}$ in the sense:
\[
\int_{\R^n \times \Gr} f (x,P) \, dV_ S(x,P) = \int_S f (x, T_x S) \:  d \mathcal{H}^d (x) \quad \forall f \in \xC_c (\R^n \times \Gr) \: ,
\] where $T_x S$ is the approximate tangent space at $x$ which exists $\mathcal{H}^d$--almost everywhere in $S$.
\end{definition}

We recall that using Riesz representation theorem, we can define $d$--varifolds as non-negative linear forms on $\xC_c ( \R^n \times \Gr)$, as done in Definition~\ref{dfnRectifVarifold} above.

\noindent An important feature of varifolds is that they can model discrete geometric objects as well. In our context, we are more specifically in \emph{point cloud varifolds} (see \cite{BuetLeonardiMasnou}) that are associated with sets of points:
\begin{definition}[Point cloud varifold]\label{dfnPointCloudVarifold}
Let $N \in \N^\ast$ and assume that we are given: a finite set of points $\{x_i\}_{i=1 \ldots N} \subset \R^N$, associated with positive weights $(m_i)_{i=1 \ldots N}$ and $d$--dimensional directions $(P_i)_{i=1 \ldots N}$ in $\Gr$.
We can associate with such data the \emph{point cloud $d$--varifold}
\[
V_N = \sum_{i=1}^N m_i \delta_{(x_i,P_i)} \: .
\]
\end{definition}
\noindent We observe that unlike the case of rectifiable $d$--varifolds, the dimension of the geometric object, i.e. $0$ for the set of point, can differ from the dimension $d$ of the $d$--varifold which is fixed by the choice of the Grassmannian $\Gr$.

\subsubsection*{Varifolds and Radon measures in $\R^n \times \xSym_+(n)$}

For the purpose of the paper, we will also consider Radon measures in $\R^n \times \xSym_+(n)$, similarly to what is done in \cite{tinarrage,BuetPennec}, see also \cite{ambrosioSoner}. As already mentioned, the identification between $T \in \Gr$ and $\Pi_T$ the rank $d$ orthogonal projector onto $T$ provides a natural embedding $i : \Gr \hookrightarrow \xSym_+(n)$, $T \mapsto \Pi_T$ that induces the following embedding of $d$--varifolds into Radon measures in $\R^n \times \xSym_+(n)$:
\[
 \mathcal{I} : V \in \{ d\text{--varifold in } \R^n \} \mapsto W = ( {\rm id}, i)_\# V \: ,
\]
or equivalently, by definition of pushforward measure, for $f \in \xC_c (\R^n \times \xSym_+(n))$, 
\[
 \int_{\R^n \times \xSym_+(n)} f \: dW  = \int_{\R^n \times \Gr} f(x, i(T))  \: dV (x,T) =  \int_{\R^n \times \Gr} f(x, \Pi_T)  \: dV (x,T) \: .
\]
In the particular case of the $d$--varifold $V_S = \cH^d_{| S} \otimes \delta_{T_x S}$ associated with the rectifiable set $S$ (see Definition~\ref{dfnRectifVarifold}), we obtain the Radon measure 
\begin{equation}
\label{eqWS}
 \mathcal{I} (V_S) = W_S = \cH^d_{| S} \otimes \delta_{\Pi_{T_x S}} \: .
\end{equation}
Recalling the notation ${{\rm P}_{d,n}} = \{ \text{rank } d\text{--orthogonal projectors} \} \subset \xSym_+(n)$, we observe that for a $d$--varifold $V$, $W = \mathcal{I}(V)$ is supported in $\R^n \times {{\rm P}_{d,n}}$.
We recall that $i : \Gr \rightarrow {{\rm P}_{d,n}}$ is a bijective isometry, we can denote by $s_0 : {{\rm P}_{d,n}} \rightarrow \Gr$ its inverse and in particular, $i$ and $s_0$ are $1$--Lipschitz maps. Then, $\mathcal{I}$ is bijective from $\{ d\text{--varifolds in } \R^n \}$ onto $\{ \text{Radon measures in } \R^n \times \xSym_+(n) \text{ with support in } \R^n \times {{\rm P}_{d,n}} \}$, with inverse $\mathcal{I}^{-1} : W \mapsto V = \one_{{{\rm P}_{d,n}}} \left( ({\rm id}, s_0)_\# W \right)$ i.e., for $g \in \xC_c (\R^n \times \Gr)$,
\begin{equation*}
    \int_{\R^n \times \Gr} g \: d V = \int_{\R^n \times \xSym_+^1(n)} g(x,s_0(A)) \one_{{{\rm P}_{d,n}}} (A) \: d W(x,A) \: .
\end{equation*}
As we detail in the proof of Proposition~\ref{propBiLipI}, it is not difficult to check that $\mathcal{I}$ is $1$--Lipschitz using that $i$ is, whereas it is more subtle to check
that its inverse $\mathcal{I}^{-1}$ is Lipschitz since $s_0$ is only defined on ${{\rm P}_{d,n}}$. More precisely, we have the following bi--Lipschitz correspondence:

\begin{proposition}
\label{propBiLipI}
 The map $\mathcal{I} : V \mapsto ( {\rm id}, i)_\# V$ induces a bi--Lipschitz bijection from $\{ d\text{--varifolds in } \R^n \}$ onto $\{ \text{Radon measures in } \R^n \times \xSym_+(n) \text{ with support in } \R^n \times {{\rm P}_{d,n}} \}$ endowed with the bounded Lipschitz distance $\beta$. More precisely, there exists $L \geq 1$ depending on $n$ such that for any $d$--varifolds $V_1$, $V_2$ in $\R^n$,
 \begin{equation*}
  \beta (\mathcal{I}(V_1), \mathcal{I}(V_2)) \leq \beta (V_1, V_2) \leq L \beta (\mathcal{I}(V_1), \mathcal{I}(V_2)) \: .
 \end{equation*}
\end{proposition}

\begin{proof}
Given two $d$--varifolds $V_1$, $V_2$ and using the notations $W_1 = \mathcal{I}(V_1)$ and $W_2 = \mathcal{I}(V_2)$, it is not difficult to check that
\begin{equation} \label{eqBilip1}
 \beta (W_1, W_2) \leq \beta (V_1, V_2) \: .
\end{equation}
Indeed, given $f \in \xC_c(\R^n \times \xSym_+(n))$ a $1$--Lipschitz function such that $\| f \|_\infty \leq 1$, we have that $g = f \circ ({\rm id}, i) \in \xC_c (\R^n \times \Gr)$ is $1$--Lipschitz and satisfies $\| g \|_\infty \leq 1$ as well whence
\begin{equation*}
 \int_{\R^n \times \xSym_+(n)} f \: ( dW_1 - dW_2)  = \int_{\R^n \times \Gr} g  \: (dV_1 - dV_2) \leq \beta (V_1, V_2) \: ,
\end{equation*}
and one can take the supremum with respect to such maps $f$ to infer \eqref{eqBilip1}.

\noindent The converse inequality holds true as well, up to some dimensional constant but requires to extend the above map $s_0$ in a neighbourhood of ${{\rm P}_{d,n}}$ while controlling the Lipschitz constant of the extension. We introduce the neighbourhood $\mathcal{V} = \left\lbrace A \in \xSym_+(n) \: : \:  {\rm dist} (A , {\rm P}_{d,n}) \leq \frac{1}{3} \right\rbrace$ of ${{\rm P}_{d,n}}$ and we first define $\tilde{s} : \mathcal{V} \rightarrow {{\rm P}_{d,n}}$ such that $\tilde{s}_{| {{\rm P}_{d,n}}} = {\rm id}_{{\rm P}_{d,n}}$ as follows. Let $A \in \mathcal{V}$ so that there exists $P\in {\rm P}_{d,n}$ satisfying $\|A-P\|\leq\frac{1}{3}$. We denote by $\lambda_1(A)\geq \ldots \geq\lambda_n(A)$ the ordered eigenvalues of $A$. 
Applying Weyl inequality for symmetric matrices, that is for all $k = 1 \ldots n$, $|\lambda_k (A) - \lambda_k(P)| \leq \|A- P\|$, we obtain:
\begin{equation*}
   \lambda_{d+1}(A)\leq  \underbrace{\lambda_{d+1}(P)}_{=0} +\underbrace{\|A-P\|}_{\leq \frac{1}{3}} \leq \frac{1}{3} 
   \quad \text{and} \quad
   \underbrace{\lambda_{d}(P)}_{=1}\leq  \lambda_{d}(A) +\underbrace{\|A-P\|}_{\leq \frac{1}{3}} \quad \Rightarrow \quad \lambda_d(A) \geq \frac{2}{3} > \frac{1}{3} \geq \lambda_{d+1}(A) \: .
\end{equation*}
We then define $\tilde{s}(A) \in {\rm P}_{d,n}$ to be the orthogonal projection matrix onto the direct sum of eigenspaces of $A$ from $\lambda_1(A)$ up to $\lambda_d(A)$.
In other words, $\tilde{s}(A)$ is obtained by replacing the eigenvalues of $A$ with $1$ for $\lambda_1,\dots,\lambda_{d}$ and $0$ for the others, while not changing the eigenspaces.
The Lipschitz continuity of $\tilde{s}$ in $\mathcal{V}$ is a direct consequence of Davis--Kahan Theorem (\cite{davisK}, see also Theorem~2 in \cite{Samworth}): let $A,B \in \mathcal{V}$, by Davis--Kahan Theorem,
\[
 \| \tilde{s}(A)- \tilde{s}(B) \|_F \leq 2 \sqrt{2} \frac{\| A-B \|_F}{|\lambda_{d+1}(A) - \lambda_d(A) |} \leq \frac{2}{3} \sqrt{2} \| A-B \|_F \: ,
\]
where $\| \cdot \|_F$ is the Frobenius norm that is equivalent to the operator norm $\| \cdot \|$ then implying the Lipschitz continuity of $\tilde{s}$ in $\mathcal{V}$ with a Lipschitz constant $\tilde{L} \geq 1$ only depending on $n$. Then $s = s_0 \circ \tilde{s}$ is $\tilde{L}$--Lipschitz as well and extends $s_0$ to $\mathcal{V}$.
Given $g \in \xC_c (\R^n \times \Gr)$, we can then write
\begin{equation*}
 \int_{\R^n \times \Gr} g \: (dV_1 - dV_2)  = \int_{\R^n \times \xSym_+(n)} g(x, s(A))  \xi_{\mathcal{V}} (A)  \: (dW_1 - dW_2)  \: .
\end{equation*}
where $\xi_{\mathcal{V}} (A) = 1 - 4 \: {\rm dist}(A, {\rm P}_{d,n})$ if ${\rm dist}(A, {\rm P}_{d,n}) \leq \frac{1}{4}$ and $0$ otherwise so that $\xi_{\mathcal{V}}$ is $4$--Lipschitz continuous and compactly supported in $\left\lbrace A \in \xSym_+(n) \: : \:  {\rm dist} (A , {\rm P}_{d,n}) \leq \frac{1}{4} \right\rbrace$. Therefore, $f = (g \circ ({\rm id} , s)) \xi_\mathcal{V}$ if $\tilde{L} + 4$--Lipschitz and satisfies $\| f \|_\infty \leq 1$. We can eventually take the supremum with respect to such maps $g$ to infer that for some dimensional constant $L \geq 1$,
\begin{equation} \label{eqBilip12}
 \beta (V_1, V_2) \leq L \beta (W_1, W_2) \: .
\end{equation}
Note that we showed that we could express $\mathcal{I}^{-1}$ as $\mathcal{I}^{-1} : W \mapsto V = \xi_{\mathcal{V}} \left( ({\rm id}, s)_\# W \right)$.
\end{proof}
Building upon such an explicit bi--Lipschitz correspondence $\mathcal{I}$ between $d$--varifolds and Radon measures in $\R^n \times \xSym_+(n)$ with support in $\R^n \times {\rm P}_{d,n}$, we estimate $W_S$ hereafter, rather than $V_S$. We will also abusively call $d$--varifolds such Radon measures in $\R^n \times \xSym_+(n)$ with support in $\R^n \times {{\rm P}_{d,n}}$.

\section{Inference framework}
\label{secInferenceFramework}

Consistently with what we explained in the introduction, we fix the following setting from now on and all along the paper: we are given a $d$--dimensional closed set $S \subset \R^n$ satisfying $\cH^d (S) < +\infty$ and we assume that we have access to samples of $S$, but possibly not uniformly distributed in $S$. More formally, for $N \in \N^\ast$, $(X_1, \ldots, X_N)$ are $\R^n$--valued independent random variables identically distributed according to a law $\mu = \theta \cH^d_{| S}$. In this setting, ``uniformly distributed in $S$'' would mean that the density $\theta$ is constant. We consider hereafter more general Borel densities $\theta : S \rightarrow \R_+$, however, a crucial requirement is that $\theta \geq \theta_{min} > 0$ is lower bounded. Loosely speaking, we exclude the case where the sampling process would produce holes in some parts of $S$. We also assume that $\theta \leq \theta_{max} < +\infty$ is upper bounded. We mention that such uniform assumptions could be relaxed regarding the pointwise results in Section~\ref{secPointwiseDensity} (see Remark~\ref{rmkUpperDensityBound}) while the ``non pointwise'' convergence results (involving the Bounded Lipschitz distance) strongly rely on such uniform bounds. 
An important point is that the following assumptions 
\hyperref[hypH1]{$(H_1)$} and \hyperref[hypH2]{$(H_2)$} would automatically imply that $S$ is compact even though we do not require it (see Remark~\ref{rkHyp}) and we directly assume it then.
For the sake of clarity, we will thereafter refer to the following hypotheses:

\begin{description}
 \item[$(H_1)$  \label{hypH1}] The set $S \subset \R^n$ is a compact set satisfying $\cH^d(S) < +\infty$ and $\exists \widetilde{C_0} \geq 1$ such that $\forall x \in S$, $\forall 0<r \leq \xdiam (S)$,
\begin{equation*}
{\widetilde{C_0}}^{-1} r^d\leq \cH^d (S \cap B(x,r)) \leq \widetilde{C_0} r^d \: ,
\end{equation*}
and the probability measure $\mu$ is defined as $\mu = \theta \cH^d_{| S}$ with $\theta : \R^n \rightarrow \R_+ \in \xL^1 (\cH^d_{| S})$ such that $\int_S \theta \: d \cH^d = 1$.
\item[$(H_2)$  \label{hypH2}] We assume that there exist $0 < \theta_{min} \leq \theta_{max} < +\infty$ such that for $\cH^d$--a. e. $x \in S$, $$\theta_{min} \leq \theta(x) \leq \theta_{max} \: .$$
\item[$(H_3)$  \label{hypH3}] We assume $d \in \N$ and $S$ is $d$--rectifiable.
\end{description}

We recall that we introduced in Section~\ref{secIntroRecVarifold} the regularity class 
\begin{equation} \label{eqClassQ}
\cQ = \{ \mu = \theta \cH^d_{| S} \: : \: S, \: \theta \text{ satisfy  \hyperref[hypH1]{$(H_1)$}--\hyperref[hypH2]{$(H_2)$}--\hyperref[hypH3]{$(H_3)$}} \}
\end{equation}
that depends on $d, \widetilde{C_0}, \theta_{min/max}$ or equivalently on $d, C_0, \theta_{min/max}$.

\begin{remark} \label{rkHyp}
Note that:
\begin{enumerate}[$(i)$]
 \item If $S$ and $\theta$ satisfy \hyperref[hypH1]{$(H_1)$} and \hyperref[hypH2]{$(H_2)$} then the measure $\nu := \cH^d_{| S}$ is $d$--Ahlfors regular with regularity constant $\widetilde{C_0}$, $\supp \nu = S$, the measure $\mu$ is $d$--Ahlfors regular with regularity constant $C_0 = \widetilde{C_0} \max (\theta_{min}^{-1}, \theta_{max})$ and $S$ is bounded even if not required.
 \item If $S$ and $\theta$ satisfy \hyperref[hypH1]{$(H_1)$}, it would be equivalent to require that $\mu = \theta \cH^d_{| S}$ is $d$--Ahlfors regular or \hyperref[hypH2]{$(H_2)$}.
 \item \hyperref[hypH3]{$(H_3)$} implies that $\mu$ is $d$--rectifiable.
\end{enumerate}
\end{remark}

\noindent Indeed, if $S$ and $\theta$ satisfy \hyperref[hypH1]{$(H_1)$}, let us check that $S = \supp \nu$ and then the Ahlfors regularity of $\nu$ follows from \hyperref[hypH1]{$(H_1)$}. Let $x \in \R^n \setminus S$, which is an open set, then for $r > 0$ small enough $S \cap B(x,r) = \emptyset$ hence $x \notin \supp \nu$ and thus $\supp \nu \subset S$. Conversely, if $x \in  S$, then for all $r > 0$, $\nu(B(x,r)) \geq \widetilde{C_0}^{-1} r^d > 0$ hence $x \in \supp \nu$ and thus $S \subset \supp \nu$.
The Ahlfors regularity of $\mu$ is then straightforward thanks to \hyperref[hypH2]{$(H_2)$}:
let $x \in S$ and $0 < r \leq \xdiam(S)$, then
\[
  \widetilde{C_0}^{-1} \theta_{min} r^d \leq \theta_{min} \nu(B(x,r))
  \leq \mu(B(x,r)) = \int_{B(x,r)} \theta \: d \nu \leq \theta_{max} \nu(B(x,r)) \leq \theta_{max} \widetilde{C_0} r^d
\]
and $S = \supp \mu$ (as for $\nu$) so that $\mu$ is $d$--Ahlfors regular with regularity constant $C_0 = \widetilde{C_0}\max ( \theta_{min}^{-1} , \theta_{max} )$. As $\mu$ is a finite $d$--Ahlfors measure in $\R^n$, $S = \supp \mu$ is automatically bounded thanks to Remark~\ref{remkAhlfors}.
We thus checked $(i)$.
As for $(ii)$, if $S$ and $\theta$ satisfy \hyperref[hypH1]{$(H_1)$} and $\mu$ is $d$--Ahlfors regular with regularity constant $C_0$, then $\theta$ satisfy \hyperref[hypH2]{$(H_2)$} by differentiation of Radon measures: for a. e. $x \in S$,
\[
 \theta (x) = \lim_{r \to 0_+} \frac{\mu(B(x,r))}{\nu(B(x,r))} \quad \text{and} \quad \forall r > 0, \, (C_0 \widetilde{C_0})^{-1} \leq \frac{\mu(B(x,r))}{\nu(B(x,r))} \leq C_0 \widetilde{C_0} \: .
\]
In Sections~\ref{secUniformPWHolder} and \ref{secSplit}, further assumptions \eqref{hypH4} to \eqref{hypH7} will be added to this minimal setting leading to the regularity class $\cP$ \eqref{eqClassP}.

\subsection{Some preliminary facts concerning the empirical measure}

Let $(\Omega,\mathcal{A},\mathds{P})$ be a probability triplet. We fix hereafter some notations and collect some preliminary computations that will be useful in the sequel (more particularly in the proof of Proposition~\ref{propBetaLoc}).

\begin{definition}[Empirical measure]
\label{empmea}
Let $N \in \mathbb{N}^\ast$ and $(X_1,\dots,X_N)$ be $N$ independent $\R^n$--valued random variables with the same law $\mu$, the associated empirical measure $\mu_N$ is defined as:
\[
 \mu_N=\frac{1}{N}\sum_{i=1}^N\delta_{X_i} \: .
\]
\end{definition}

We recall the following basic properties of the empirical measure. Given $N \in \mathbb{N}^*$ and an i.i.d. sample $(X_1,\dots,X_N)$ with common distribution $\mu$ in $\R^n$, we have for any Borel sets $B \subset \R^n$,
\begin{equation} \label{eqBasicMuN}
 \E \left[ \mu_N(B) \right] = \mu (B) \quad \text{and} \quad \xvar \left( \mu_N(B) \right) = \frac{\mu(B) - \mu(B)^2}{N} \quad \text{and} \quad \E \left[ \mu_N(B)^2 \right] \leq \mu(B)^2 + \frac{1}{N} \mu(B) \: .
\end{equation}
We will also need the following similar though slightly more technical estimates (in the proof of Lemma~\ref{L1}): 

\begin{lemma}
For any Borel set $A, B \subset \R^n$ satisfying $\displaystyle \frac{1}{N} \leq C \min (\mu(A), \mu(B))$ (with $C \geq 1$), we have
\begin{equation} \label{eqBasicMuNBis}
 \E \left[ \mu_N(A) \mu_N(B) \right] \leq 2 C \mu(A) \mu(B) 
 \quad \text{and} \quad
 \E \left[ \mu_N(A)^2 \mu_N(B)^2 \right] \leq 15 C^3 \mu(A)^2 \mu(B)^2 \: .
\end{equation}
Assume in addition that $\mu$ is $d$--Ahlfors regular with constant $C_0 \geq 1$, then, for any Borel set $T \subset \R^n$ and for all $0 < N^{-\frac{1}{d}} \leq R, \widetilde{R}$, we have
\begin{equation} \label{eqBasicMuNTer}
 \E \left[ \int_T \mu_N (B(x,R)) \: d (\mu_N + \mu)(x) \right] \leq 3 C_0 R^d \mu(T) \: ,
\end{equation}
\begin{equation} \label{eqBasicMuN4}
 \E \left[ \int_T \mu_N (B(x,R)) \mu_N (B(x, \widetilde{R})) \: d (\mu_N + \mu)(x) \right] \leq 7 C_0^2 R^d \widetilde{R}^d \mu(T) \: .
\end{equation}
\end{lemma}
\begin{proof}
Let us start with the proof of \eqref{eqBasicMuNBis}.
First note that due to the symmetry of the estimates, we can assume that $\mu(A) \leq \mu(B)$ without loss of generality.
Then,
\[
 \mu_N (B) = \frac{1}{N} \sum_{i=1}^N \delta_{X_i} (B) \quad \text{and} \quad \E \left[ \mu_N(A) \mu_N(B) \right] = \frac{1}{N^2} \sum_{i,j = 1}^N \E \left[ \delta_{X_i} (A) \delta_{X_j} (B) \right] \: .
\]
If $i = j$, then $\E \left[ \delta_{X_i} (A) \delta_{X_j} (B) \right] = \E \left[ \delta_{X_i} (A \cap B) \right] = \mathds{P} ( X_i \in A \cap B) = \mu(A \cap B)$ and otherwise $X_i$ and $X_j$ are independent and thus $\E \left[ \delta_{X_i} (A) \delta_{X_j} (B) \right] = \E \left[ \delta_{X_i} (A) \right] \E \left[ \delta_{X_j} (B) \right] =\mu(A) \mu(B)$
so that
\[
 \E \left[ \mu_N(A) \mu_N(B) \right] \leq \mu(A) \mu(B) + \underbrace{\frac{N}{N^2}}_{\leq C \mu(A)} \underbrace{\mu(A \cap B)}_{\leq \mu(B)} \leq (1 + C) \mu(A) \mu(B) \: .
\]
We are left with the second estimate in \eqref{eqBasicMuNBis} and we perform similar computations:
\[
 \E \left[ \mu_N(A)^2 \mu_N(B)^2 \right] = \frac{1}{N^4} \sum_{i,j,k,l = 1}^N \E \left[ \delta_{X_i} (A) \delta_{X_j} (A) \delta_{X_k} (B) \delta_{X_l} (B) \right] \: .
\]
We can then use independence when the four indices are all different and enumerate the other cases, which leads to
\begin{align*}
 \E \left[ \mu_N(A)^2 \mu_N(B)^2 \right]  \leq & \mu(A)^2 \mu(B)^2 + \frac{1}{N} \mu(A) \mu(B) \left( \mu(A) + \mu(B) + 4 \mu(A \cap B) \right) + \frac{1}{N^2} \left( \mu(A) \mu(B) + 2 \mu(A \cap B)^2 \right) \\
 & + \frac{2}{N^2} \mu(A \cap B) \left( \mu(A) + \mu(B) \right) + \frac{1}{N^3} \mu(A \cap B) \\
 \leq & \mu(A)^2 \mu(B)^2 + \frac{6}{N} \mu(A) \mu(B)^2 + \frac{7}{N^2} \mu(B)^2  + \frac{1}{N^3} \mu(B) \: ,
\end{align*}
where we used $\mu(A \cap B) \leq \mu(A) \leq \mu(B)$ and we conclude the proof of \eqref{eqBasicMuNBis} with $\frac{1}{N} \leq C \mu(A)$.

\noindent We proceed with the proof of \eqref{eqBasicMuNTer}. By linearity and Ahlfors regularity, we have
\begin{equation} \label{eqPrelimMuNMu1}
\mathbb{E}\left[\int_T \mu_N(B(x,R))d\mu(x)\right]=\int_T \mathbb{E}\left[\mu_N(B(x,R))\right]d\mu(x) = \int_T\mu(B(x,R))d\mu(x) \leq C_0 R^d \mu(T) \: .
\end{equation}
Furthermore,
\begin{align*}
\mathbb{E}\left[\int_T \mu_N(B(x,R))d\mu_N(x)\right]  = \frac{1}{N}\sum_{i=1}^N \E \left[ \mu_N(B(X_i,R)) \: \delta_{X_i}(T) \right]  = \frac{1}{N^2} \sum_{i,j=1}^N \E \left[ \mathds{1}_{\{|X_i - X_j | < R \}}  \: \mathds{1}_{\{X_i \in T\}} \right] 
\end{align*} 
and for $i \neq j$, 
\begin{align}
\E \left[ \mathds{1}_{\{|X_i - X_j | < R \}}  \: \mathds{1}_{\{X_i \in T\}} \right] & = \E [ u(X_i, X_j) ] \quad \text{with} \quad u(x,y) =  \mathds{1}_{\{|x - y| < R \}}  \: \mathds{1}_{\{x\in T\}} \nonumber \\
& = \int_{(x,y) \in (\R^n)^2} u(x,y) \: d\mu(x) d\mu(y) \quad \text{since } (X_i,X_j) \sim \mu \otimes \mu \nonumber \\
& = \int_{x \in T} \int_{\{ y \: : \: |x-y| < R  \}} \: d\mu(y) d\mu(x) \nonumber \\
& = \int_T \mu \left( B(x,R) \right) \: d \mu(x) \: ,
\label{eqPrelimMuNMu12}
\end{align}
whereas for $i = j$,
\begin{equation*}
 \E \left[ \mathds{1}_{\{|X_i - X_j | < R \}}  \: \mathds{1}_{\{X_i \in T\}} \right] = E \left[ \mathds{1}_{\{X_i \in T\}} \right] = \mu(T) \: .
\end{equation*}
Therefore, using $N^{-1} \leq R^d$,
\begin{equation} \label{eqPrelimMuNMu2}
 \mathbb{E}\left[\int_T \mu_N(B(x,R))d\mu_N(x)\right]  \leq \int_T \mu \left( B(x,R) \right) \: d \mu(x) + \frac{1}{N} \mu(T) \leq (C_0 +1) R^d \mu(T) \leq 2 C_0 R^d \mu(T) \: ,
\end{equation}
and we infer \eqref{eqBasicMuNTer} from \eqref{eqPrelimMuNMu1} and \eqref{eqPrelimMuNMu2}.\\
We are left with the proof of \eqref{eqBasicMuN4} which is similar. We first note that by symmetry of the statement, we can assume for instance $R \leq \widetilde{R}$. Then by assumption $\frac{1}{N} \leq R^d$ and by Ahlfors regularity,
\begin{equation*}
\E \left[ \mu_N(B(x,R)) \mu_N(B(x, \widetilde{R})) \right] \leq \mu(B(x,R)) \mu(B(x, \widetilde{R})) + \frac{1}{N} \mu(B(x,R)) \leq 2C_0^2 R^d \widetilde{R}^d \: ,
\end{equation*}
so that
\begin{equation} \label{eqPrelimMuNMu3}
\mathbb{E}\left[\int_T \mu_N(B(x,R)) \mu_N(B(x, \widetilde{R})) \:  d\mu(x)\right]  \leq 2C_0^2 R^d \widetilde{R}^d \mu(T) \: .
\end{equation}
Furthermore,
\begin{align*}
\mathbb{E}\left[\int_T \mu_N(B(x,R))  \mu_N(B(x, \widetilde{R})) \: d\mu_N(x)\right]  = \frac{1}{N^3}\sum_{i,j,k=1}^N \E \left[ \mathds{1}_{\{|X_i - X_j | < R \}}  \: \mathds{1}_{\{|X_i - X_k | < R \}} \: \mathds{1}_{\{X_i \in T\}} \right] \: ,
\end{align*} 
and we review the disjoint possibilities hereafter.\\
If $i = j = k$, we have $\displaystyle \E \left[ \mathds{1}_{\{|X_i - X_j | < R \}}  \: \mathds{1}_{\{|X_i - X_k | < R \}} \: \mathds{1}_{\{X_i \in T\}} \right] = \E \left[ \mathds{1}_{\{X_i \in T\}} \right] = \mu(T)$. \\
Otherwise, if $i = k$ (or similarly $i = j$, replacing $R$ with $\widetilde{R}$), we have thanks to \eqref{eqPrelimMuNMu12}:
\begin{equation*}
 \E \left[ \mathds{1}_{\{|X_i - X_j | < R \}}  \: \mathds{1}_{\{|X_i - X_k | < \widetilde{R} \}} \: \mathds{1}_{\{X_i \in T\}} \right] = \E \left[ \mathds{1}_{\{|X_i - X_j | < R \}} \mathds{1}_{\{X_i \in T\}} \right] = \int_T \mu \left( B(x,R) \right) \: d \mu(x) \leq C_0 R^d \mu(T) \: ,
\end{equation*}
whereas if $j = k$, with $R \leq \widetilde{R}$, 
\begin{align*}
 \E \left[ \mathds{1}_{\{|X_i - X_j | < R \}}  \: \mathds{1}_{\{|X_i - X_k | < \widetilde{R} \}} \: \mathds{1}_{\{X_i \in T\}} \right] & = \E \left[ \mathds{1}_{\{|X_i - X_j | < R \}} \mathds{1}_{\{X_i \in T\}} \right]\\
 & = \int_T \mu \left( B(x, R) \right) \: d \mu(x) \leq C_0 R^d \mu(T) \: .
\end{align*}
Finally, if $(i,j,k)$ are distinct,
\begin{align*}
 \E \left[ \mathds{1}_{\{|X_i - X_j | < R \}}  \: \mathds{1}_{\{|X_i - X_k | < \widetilde{R} \}} \: \mathds{1}_{\{X_i \in T\}} \right] & = \E \left[ u(X_i, X_j, X_k) \right] \text{ with } u(x,y,z) = \mathds{1}_{\{|x - y | < R \}}  \mathds{1}_{\{|x - z | < \widetilde{R} \}} \: \mathds{1}_{\{x \in T\}} \\
 & = \int_{(\R^n)^3} u(x,y,z) \: d \mu(x) d \mu(y) d \mu(z) \\
 & = \int_T \mu \left( B(x, R) \right)  \mu \left( B(x, \widetilde{R}) \right) \: d \mu(x) \leq C_0^2 R^d \widetilde{R}^d \mu(T) \: .
\end{align*}
Enumerating and combining the different cases, we thus obtain using $\frac{1}{N} \leq R^d$,
\begin{align*}
\mathbb{E}\left[\int_T \mu_N(B(x,R))  \mu_N(B(x, \widetilde{R})) \: d\mu_N(x)\right] & \leq C_0^2 R^d \widetilde{R}^d \mu(T) + \frac{C_0}{N} (2 R^d + \widetilde{R}^d ) \mu(T) + \frac{1}{N^2} \mu(T) \\
\leq 5C_0^2 R^d \widetilde{R}^d \mu(T) \: ,
\end{align*}
which concludes the proof of \eqref{eqBasicMuN4} thanks to \eqref{eqPrelimMuNMu3}.
\end{proof}

\subsection{Pointwise estimator of density} \label{secPointwiseDensity}

As a first step in the estimation of $S$, we recall in this section a usual kernel estimator of the density $\theta$ and its pointwise mean convergence rate (see Corollary~\ref{coroDensityPointwiseCv}). Note that the weak regularity framework does not affect the convergence rate of the ``concentration'' part (see Proposition~\ref{propDensityPointwiseCv} and Proposition 3.5 in \cite{Berenfeld}) whose proof is straightforward.
We fix a family of radial kernels in the usual way: we fix a Lipschitz even function $\eta : \mathbb{R}\mapsto \mathbb{R}_+$ with support in $(-1,1)$ and we additionally assume that $\eta > 0$ in $\left[ - \frac{1}{2}, \frac{1}{2} \right]$ and
we define
\begin{equation} \label{eqKernelEta}
C_\eta = d \omega_d \int_{r=0}^1 \eta(r)r^{d-1} \: dr \quad \text{and for } \delta > 0, \: x \in \R^n, \, 
\eta_\delta(x) =  \eta \left(\frac{|x|}{\delta} \right) \: .
\end{equation}
We introduce the following notations that will be used throughout the next sections: given a Radon measure $\lambda$ in $\R^n$, $x \in \R^n$ and $\delta > 0$, $\displaystyle \Theta_\delta  (x, \lambda) = \frac{\lambda \ast \eta_\delta}{C_\eta \delta^d}$ and when there is no ambiguity, we will use the following shortened notations in the cases $\lambda = \mu$ and $\lambda = \mu_N$,
\begin{equation} \label{eqThetaN}
 \theta_{\delta} = \Theta_\delta  (\cdot , \mu) = \frac{\mu \ast \eta_\delta}{C_\eta \delta^d} \quad \text{and} \quad  \theta_{\delta, N} = \Theta_\delta  (\cdot, \mu_N) = \frac{\mu_N \ast \eta_\delta}{C_\eta \delta^d} \: .
\end{equation}
We first check in Proposition~\ref{lemDensityPointwiseCv} that we have the convergence of $\theta_\delta$ to $\theta$ a.e. in $S$ when $\delta$ tends to $0$, only assuming that $\mu = \theta \cH^d_{| S}$ is rectifiable.
We then recall in Proposition~\ref{propDensityPointwiseCv}) the mean convergence of $\theta_{\delta,N}$ to $\theta_\delta$ for a given $\delta > 0$. We conclude in Corollary~\ref{coroDensityPointwiseCv} that we can choose a sequence $(\delta_N)_N$ tending to $0$ such that $\theta_{\delta_N, N}$ is a pointwise convergent estimator of the density $\theta$.

\begin{proposition} \label{lemDensityPointwiseCv}
Let $1 \leq d \leq n$ be an integer and let $\mu = \theta \cH^d_{| S}$ be a $d$--rectifiable measure. 
Let $\theta_\delta :\mathbb{R}^n \mapsto \mathbb{R}_+$ be as in \eqref{eqThetaN}.
Then for $\mathcal{H}^d$--a.e. $x \in S$, $\displaystyle \lim_{\delta \to 0_+} \theta_\delta(x) = \theta(x)$.
\end{proposition}

\begin{proof}
We use the rectifiability of $\mu$ implying that $\mathcal{H}^d$--a.e. in $S$, $\mu$ has an approximate tangent plane $T_x S$. Let $x \in S$ be such a point and apply Definition~\ref{ATP} with $\eta_1 = \eta (|\cdot|) \in \xC_c(\R^n)$ to obtain 
\begin{equation*}
\theta_\delta (x) = \frac{1}{C_\eta \delta^d}\int_{\mathbb{R}^n} \eta \left(\frac{|y-x|}{\delta}\right) \: d\mu(y)\xrightarrow[\delta\to 0]{}\theta(x) \frac{1}{C_\eta}\int_{T_x S} \eta(|y|) \: d\mathcal{H}^d(y) = \theta(x) \: ,
\end{equation*}
where the last equality follows from the coarea formula applied on concentric spheres:
\begin{align}
\int_{T_xS}\eta(|y|)d\mathcal{H}^d(y)=&\int_{r=0}^1\int_{ \{ y\in T_x S \: : \: |y| = r \} }\eta(|y|)d\mathcal{H}^{d-1}(y)dr   
= \int_{r=0}^1\eta(r) \underbrace{\mathcal{H}^{d-1}(T_x S \cap \partial B(0,1))}_{= d \omega_d} r^{d-1} \: dr \nonumber  \\
=& d\omega_d\int^1_0\eta (r)r^{d-1}dr=C_{\eta} \: . \label{eqCetaRn}
\end{align}
\end{proof}

\begin{proposition} \label{propDensityPointwiseCv}
Let $0 < d \leq n$ and assume that $S$ and $\theta$ satisfy \hyperref[hypH1]{$(H_1)$} and \hyperref[hypH2]{$(H_2)$} and let $C_0 \geq 1$ be a regularity constant for $\mu = \theta \cH^d_{| S}$ (see Remark~\ref{rkHyp}$(i)$).
Let $\mu_N$ be the empirical measure associated with $\mu$ and $\theta_\delta$, $\theta_{\delta, N} :\mathbb{R}^n \mapsto \mathbb{R}_+$ be as in \eqref{eqThetaN}.
Then there exists a constant $M = M(d,C_0,\eta) > 0$ such that for all $x \in S$ and for all $\delta > 0$, $N \in \N^\ast$,
\begin{equation} \label{densityestim}
\mathbb{E}\left[\left|\theta_{\delta, N}(x) -\theta_\delta (x)\right|\right]\leq \frac{M}{\sqrt{N \delta^d}} \: .
\end{equation}
\end{proposition}

\begin{proof}
We fix $x\in S$ and define for $j \in \{ 1, \ldots, N\}$,
$$
Z_j =\frac{1}{C_\eta \delta^d} \eta\left(\frac{|x-X_j|}{\delta}\right) \quad \text{and} \quad Z = \theta_{\delta,N}(x) = \frac{1}{N} \sum_{j=1}^N Z_j  \:  .
$$
Since $Z_1, \ldots, Z_N$ are i.i.d., we have
\[
\E \left[ |\theta_{\delta,N}(x)-\theta_\delta(x)|^2 \right] = \xvar(Z) = \frac{\xvar(Z_1)}{N} \leq \frac{\E[(Z_1)^2]}{N} \leq \frac{\| \eta \|_\infty^2}{C_\eta^2 N \delta^{2d}} \E [ \one_{\{ |X_j - x| < \delta\} }] \leq \frac{C_0 \| \eta \|_\infty^2}{C_\eta^2 N \delta^{d}}  \: .
\]
We conclude the proof of Proposition~\ref{propDensityPointwiseCv} thanks to Jensen inequality:
\[ \E \left[ |\theta_{\delta,N}(x)-\theta_\delta(x)| \right] \leq \sqrt{ \E \left[ |\theta_{\delta,N}(x)-\theta_\delta(x)|^2 \right] }\leq M(d, C_0, \eta) \frac{1}{\sqrt{ N \delta^{d}}} \: .
\]
\end{proof}
Note that we do not need $\eta$ to be Lipschitz in the proof of Proposition~\ref{propDensityPointwiseCv}, it will be however necessary in Lemma~\ref{lemPropertiesPhiThetaN} and thus in Theorem~\ref{thmBetaLocNuN} as well. Combining Proposition~\ref{lemDensityPointwiseCv} and Proposition~\ref{propDensityPointwiseCv} yields the mean convergence of the kernel estimator of density $\theta_{\delta,N}$ provided that $\delta = \delta_N \to 0$ is well-chosen, as stated in Corollary~\ref{coroDensityPointwiseCv}.

\begin{corollary}
\label{coroDensityPointwiseCv}
Let $d\in \mathbb{N}^\ast$ and assume that $S$ and $\theta$ satisfy \hyperref[hypH1]{$(H_1)$}, \hyperref[hypH2]{$(H_2)$} and \hyperref[hypH3]{$(H_3)$}.
Let $\mu_N$ be the empirical measure associated with $\mu = \theta \cH^d_{| S}$ and $\theta_\delta$, $\theta_{\delta, N} :\mathbb{R}^n \mapsto \mathbb{R}_+$ as in \eqref{eqThetaN}.
Let $(\delta_N)_{N \in \N^\ast}$ be a positive sequence tending to $0$ and such that $\delta_N N^{\frac{1}{d}} \xrightarrow[N \to +\infty]{} +\infty$, then for $\cH^d$--a.e. $x \in S$,
\begin{equation} \label{eqDensityPointwiseCv}
\mathbb{E}\left[\left|\theta_{\delta_N, N}(x) -\theta(x)\right|\right] \leq M \frac{1}{\sqrt{N \delta_N^d}} + |\theta_{\delta_N}(x) - \theta(x) | \xrightarrow[N \to +\infty]{} 0 \: .  
\end{equation}
\end{corollary}

\begin{proof}
Corollary~\ref{coroDensityPointwiseCv} is a direct consequence of Proposition~\ref{propDensityPointwiseCv} and Proposition~\ref{lemDensityPointwiseCv}.
\end{proof}

As we already explained in the introduction, convergence of $\theta_{\delta_N,N}$ does not hold uniformly for $\mu \in \cQ$.
Indeed, in order to be more precise concerning both the choice of $\delta_N$ and the mean convergence rate, it is important to quantify the convergence rate of $\theta_\delta$ to $\theta$. However, as 
illustrated in Example~\ref{exQuantifyL1}, we need to strengthen the regularity framework to this end, which is done in Section~\ref{secUniformPWHolder} and \ref{secSplit} where a piecewise Hölder regularity class $\cP$ \eqref{eqClassP} is considered. It will be then possible to establish the mean convergence of the estimator $\theta_{\delta_N,N}$ with a uniform choice of $\delta_N \to 0$ and a uniform convergence rate in the specified regularity class, see Proposition~\ref{propFinalCV}.

Before moving to the estimation of the measure $\cH^d_{| S}$, we comment on the upper bound assumption on $\theta$.
\begin{remark}[Uniform bounds on the density] \label{rmkUpperDensityBound}
As mentioned in the introduction of the section, it would be possible to relax the uniform density bounds (or equivalently the Ahlfors regularity assumptions) regarding the convergence of $\theta_{\delta,N}$ since Corollary~\ref{coroDensityPointwiseCv} is a pointwise result: given $x \in S$, by definition of $\theta(x) = \lim_{\delta \to 0_+} \frac{\mu(B(x,\delta))}{\delta^d}$ it is possible to assume that for radii $\delta$ small enough, $\mu(B(x,\delta)) \leq C_x \delta^d$ where the constant $C_x > 0$ is not uniform as was $C_0$ but now depends on $x$.
\end{remark}

\section{Estimation of the {\it d}--dimensional measure carried by {\it S}}
\label{secMeasureEstimator}

In this section, we analyse the convergence of an estimator $\nu_{\delta,N}$ (see \eqref{eqNuN}) of the measure $\nu = \cH^d_{| S}$ obtained by weighting the empirical measure according to the estimated local density $\theta_{\delta,N}$. Such an estimator is commonly implemented as a simple way of balancing a non-uniform sampling, hence worthing some investigations. Assuming stronger regularity of $\theta$ and of the underlying set $S$, more intricate constructions interpolating information would achieve better convergence rates as established in \cite[Theorem~$6$]{LevradAamari} for the estimation of $S$ in Hausdorff distance, as well as in \cite[Theorem~$1$]{TangYang22} concerning the estimation of $\mu = \theta \cH^d_{| S}$. It is however interesting to analyse such a commonly used kernel-based estimator in a low-regularity framework despite a deteriorated convergence rate. Of course, a crucial point would be to establish minimax rates in terms of Bounded Lipschitz distance for estimators of $\nu = \cH^d_{| S}$ in a low-regularity model, for instance in the regularity class $\cP$. Yet, we were not able to obtain such minimax rates up to now  
and we intend to carry on our investigations on the question in future work.

In \cite{dud} the author quantifies the mean speed convergence of the empirical measure $\mu_N$ towards the sampled probability law $\mu$ in terms of Bounded Lipschitz distance. More precisely (see Theorem~$3.2$ in \cite{dud}) under some technical assumption ensuring that $\mu$ behaves like a $d$--dimensional measure in $\R^n$, the author established that
\begin{equation} \label{eqDudleyBeta}
\E \left[ \beta (\mu_N, \mu) \right] \lesssim \frac{1}{N^d} \: .
\end{equation}
In order to show the convergence of $\nu_{\delta,N}$ towards $\nu = \cH^d_{| S}$, we adapt the arguments of \cite{dud} but passing from $1$--Lipschitz and bounded by $1$ test functions to random test functions with only local bounds on the random Lipschitz constant (see technical assumption \eqref{eqHypg}).
The proof of Proposition~\ref{propBetaLoc} then follows the structure of the proof of Theorem~$3.2$ in \cite{dud}: we localize the Bounded Lipschitz distance (see Definition~\ref{dfnBetaLoc}) and we adapt the choice of partition larger scale (denoted by ``$\varepsilon_s$'' in the proof) when localizing in balls. Such a localization is natural since the pointwise value of a kernel-based estimator (e.g. $\theta_{\delta,N}$ but also the notions of approximate curvature introduced in \cite{BuetLeonardiMasnou,BuetLeonardiMasnou2}) at some point $x$ depends on the empirical measure $\mu_N$ restricted to the ball $B(x,\delta)$.
From Proposition~\ref{propBetaLoc}, we straightforwardly infer the convergence of $\nu_{\delta,N}$ towards $\cH^d_{| S}$ in Theorem~\ref{thmBetaLocNuN} and Corollary~\ref{corNuDeltaN}. As explained in the introduction, uniform convergence rates are established in Sections~\ref{secUniformPWHolder} and \ref{secSplit} (see more precisely Theorem~\ref{thmCvHolderSplit}) assuming stronger regularity (in the piecewise Hölder regularity class $\cP$ \eqref{eqClassP} as defined in Section~\ref{secPwHolderDef}). We will subsequently use Proposition~\ref{propBetaLoc} when investigating the estimation of the whole varifold measure in Section~\ref{secCVvarifold}.

The section is organised as follows, in Section~\ref{secPartitions} we first recall and adapt technical ingredients that are then applied in Section~\ref{secTechnicalDud} in order to establish Proposition~\ref{propBetaLoc}. Section~\ref{secEstimNudeltaN} is dedicated to the convergence of $\nu_{\delta,N}$ to $\nu = \cH^d_{| S}$ in the regularity class $\cQ$ (i.e. under assumptions \hyperref[hypH1]{$(H_1)$}-\hyperref[hypH2]{$(H_2)$}-\hyperref[hypH3]{$(H_3)$}) as stated in Theorem~\ref{thmBetaLocNuN} and Corollary~\ref{corNuDeltaN}.

\subsection{Nested partitions and mean discrepancy over partitions}
\label{secPartitions}

In this section, we introduce technical ingredients already contained in \cite{dud}, except Lemma~\ref{L1} $(ii)$, which is an adaptation that will be used in the proof of Proposition~\ref{propBetaLoc}.
The following lemma estimates the mean discrepancy between $\mu$ and $\mu_N$ over a given partition.
\begin{lemma}
\label{L1}
Let $T\subset \mathbb{R}^n$ be a Borel set and let $S_j$, $j=1,\dots,m$, be disjoint Borel sets with union $T$. Then,
\begin{enumerate}[$(i)$]
 \item see Proposition 3.1 in \cite{dud} \begin{equation} \label{eqMr}
\mathbb{E}\left[ \sum_{j=1}^m|\mu_N(S_j)-\mu(S_j)|\right]\leq \left( \frac{m\mu(T)}{N} \right)^\frac{1}{2} \: ,
\end{equation}
\item assume that $\mu$ is $d$--Ahlfors regular, then for $0 <R \leq \widetilde{R}$ such that $R \geq N^{-\frac{1}{d}}$ and for $x_j \in S$, $j = 1, \ldots, m$,
\begin{equation} \label{eqMballs}
\mathbb{E}\left[ \sum_{j=1}^m \mu_N (B(x_j, R)) \: |\mu_N(S_j)-\mu(S_j)|\right]\leq 2 C_0^\frac{3}{2} \left( \frac{m \mu(T)}{N} \right)^\frac{1}{2} R^d 
\end{equation}
and
\begin{equation} \label{eqMballsBis}
\mathbb{E}\left[ \sum_{j=1}^m \mu_N (B(x_j, R)) \mu_N (B(x_j, \widetilde{R})) \: |\mu_N(S_j)-\mu(S_j)|\right]\leq 4 C_0^{\frac{7}{2}} \left( \frac{m \mu(T)}{N} \right)^\frac{1}{2} R^d \widetilde{R}^d \: .
\end{equation}
\end{enumerate}
\end{lemma}
\begin{proof}
We focus on proving $(ii)$. First note that using \eqref{eqBasicMuN} we have
\begin{align}
\mathbb{E}\left[\sum_{j=1}^m \left|\mu_N(S_j) -\mu(S_j) \right|^2 \right] & =\sum_{j=1}^m \xvar(\mu_N(S_j))
=\sum_{j=1}^m \frac{\mu(S_j)-\mu (S_j)^2 }{N} =\frac{1}{N} \left(\mu(T)-\sum_{j=1}^m \mu(S_j)^2\right) \nonumber \\
&\leq \frac{\mu(T)}{N} \: .
\label{eqMballs_1}
\end{align}
Then using the $d$--Ahlfors regularity of $\mu$ with $C_0 \geq 1$, we have 
\begin{equation*}
 \frac{1}{N} \leq R^d \leq C_0 (C_0^{-1} R^d) \leq C_0 \mu \left( B(x_j,R) \right) \quad \text{for any } j = 1 \ldots m \: ,
\end{equation*}
which allows to apply \eqref{eqBasicMuNBis} and obtain
\begin{align}
\E \left[ \sum_{j=1}^{m}\mu_N(B(x_j,R))^2\right]  =   \sum_{j=1}^{m} \E \left[ \mu_N(B(x_j,R))^2\right] 
\leq 2 C_0 \sum_{j=1}^{m} \mu (B(x_j,R))^2 \leq 2 C_0^3 m R^{2d} \: , \label{eqMballs_2}
\end{align}
and similarly
\begin{align}
\E \left[ \sum_{j=1}^{m}\mu_N(B(x_j,R))^2 \mu_N(B(x_j,\widetilde{R}))^2\right] 
\leq 15 C_0^3 \sum_{j=1}^{m} \mu (B(x_j,R))^2 \mu (B(x_j, \widetilde{R}))^2 \leq 15 C_0^7 m R^{2d} \widetilde{R}^{2d} \: . \label{eqMballs_2Bis}
\end{align}
Furthermore, by Cauchy-Schwarz inequality,
\begin{equation*}
\sum_{j=1}^m \mu_N (B(x_j, R)) \: |\mu_N (S_j) -\mu (S_j)| \leq \left(\sum_{j=1}^{m} \mu_N(B(x_j,R))^2\right)^\frac{1}{2}\times \left(\sum_{j=1}^m \left|\mu_N(S_j) -\mu(S_j) \right|^2 \right)^\frac{1}{2} \: 
\end{equation*}
and hence, using Cauchy-Schwartz inequality $\E [ X^\frac{1}{2} \: Y^\frac{1}{2} ] \leq \E [X]^\frac{1}{2} \: \E [Y]^\frac{1}{2}$ for non-negative random variables $X,Y$, we get
\begin{equation} \label{eqMballs_3}
\E \left[ \sum_{j=1}^m \mu_N (B(x_j, R)) \: |\mu_N (S_j) -\mu (S_j)| \right]
\leq \E \left[ \sum_{j=1}^{m}\mu_N(B(x_j,R))^2\right]^\frac{1}{2}\times \E \left[ \sum^{m}_{j=1}(\mu_N(S_j)-\mu(S_j))^2\right]^\frac{1}{2} \: .
\end{equation}
From \eqref{eqMballs_1}, \eqref{eqMballs_2} and \eqref{eqMballs_3} we infer \eqref{eqMballs}. Then, \eqref{eqMballsBis} similarly follows from \eqref{eqMballs_1}, \eqref{eqMballs_2Bis} and Cauchy-Schwarz inequality.
\end{proof}

Let us notice that as the partition $(S_j)_{j=1 \dots m}$ is refined with smaller pieces, then $m$ gets larger and the control in the right hand side of \eqref{eqMr} increases. Nevertheless, it will be important to work with thin enough partitions in the proof of Proposition~\ref{propBetaLoc}:
we will use that inside a piece of the partition, Lipschitz functions vary at most like the diameter of the piece. The key idea introduced in \cite{dud} to deal with these opposite requirements is to work with a sequence of thinner and thinner nested partitions rather than working with only one thin partition. We refer to \cite{dud} for the construction of such partitions:

\begin{lemma}[Nested partitions, see \cite{dud}] \label{lemNestedPartitions}
Let $0 < \epsilon \leq 1$, let $t \in \N$ be such that $3^{-(t+1)} < \epsilon \leq 3^{-t}$ and let $s \in \N$, $s \leq t$. Let $T \subset \R^n$ be a bounded Borel set. We assume that for each integer $s \leq u \leq t$, $T$ is the disjoint union of $m_u$ Borel sets of diameter at most $\varepsilon_u := 3^{-u}$. Then, there exists a family of Borel sets 
\[
 \left\lbrace A_j^u \: : \: u = s, \ldots, t \text{ and } j= 1, \ldots, m_u \right\rbrace 
\]
such that for all $u$:
\begin{align}
& \left( A_j^u \right)_{j = 1 \ldots m_u} \text{ is a partition of } T \text{ and for all } j, \, \xdiam A_j^u \leq 3 \varepsilon_u  \text{ and } A_j^u \neq \emptyset \: ,\label{eqPartitionProperty} \\
& \text{if } u > s, \text{ for all }  q \in \{1, \ldots , m_{u-1} \} \text{ there exists } I_{q,u} \subset \{1, \ldots, m_u \} \text{ such that } A_q^{u-1} = \bigsqcup_{j \in I_{q,u}} A_j^u \: . \label{eqNestedProperty} 
\end{align}
\end{lemma}
\noindent We could rephrase \eqref{eqNestedProperty} as follows: pieces in the partition at scale $\varepsilon_{u-1}$ are formed by unions of pieces from the partition at scale $\varepsilon_u$.

\subsection{A technical mean convergence result for a Bounded Lipschitz type term}
\label{secTechnicalDud}
In order to obtain the convergence of the estimator $\nu_{\delta,N}$ (hereafter defined in \eqref{eqNuN}) towards $\cH^d_{| S} = \frac{1}{\theta} \mu$ in terms of localized Bounded Lipschitz distance, more precisely a uniform control on $\E \left[ \beta_B (\nu_{\delta,N}, \nu_\delta) \right]$ in the regularity class $\cQ$ \eqref{eqClassQ}, we adapt the proof of convergence of $\E \left[ \beta (\mu_N, \mu) \right]$ from \cite{dud}.
We note that Corollary~\ref{corNuDeltaN} is sufficient to obtain the convergence of $\nu_{\delta,N}$ (Theorem~\ref{thmBetaLocNuN}),
however, we will need the more general (and probably less comprehensive) Proposition~\ref{propBetaLoc} later on to prove the convergence of the varifold estimator in Proposition~\ref{prop:CV_WrdeltaN}. We hence directly state and prove hereafter Proposition~\ref{propBetaLoc}.
We more precisely adapt the proof of Theorem 3.2 in \cite{dud} to compute a localized Bounded Lipschitz distance as defined in Definition~\ref{dfnBetaLoc}, and also to handle random test functions that are not $1$--Lipschitz but satisfy a Lipschitz--type property of the form \eqref{eqHypg}, where we use the notation
\begin{equation}
\Delta_{t,N} (x,y) = \frac{1}{t^d} \mu_N \left( B(x,t) \cup B(y,t) \right) , \quad \text{for } t>0 \text{ and } x, y \in \R^n \: .
\end{equation}
Note that since $\mu_N$ is a probability measure $\Delta_{t,N} (x,y) \leq t^{-d}$ and \eqref{eqHypg} in particular gives a Lipschitz estimate for $g$ and $x$, $y \in \R^n$:
\begin{equation} \label{eqLipDetermg}
\left| g(x) - g(y) \right| \leq \left(\frac{ \kappa_1}{\delta^{d+1}} + \frac{\kappa_2}{r^{d+1}} + \frac{\kappa_3}{r^d \delta^{d+1}} + \kappa_0 \right) |x-y| \: .
\end{equation}
For the sake of clarity, let us temporarily assume that $0 < \delta = r \leq 1$. 
Hence directly applying Theorem 3.2 in \cite{dud} with the ``deterministic'' estimate \eqref{eqLipDetermg} above is possible and we would obtain
\[
\E \left[ \beta (\nu_{\delta,N}, \nu_\delta ) \right] \lesssim \frac{N^{-\frac{1}{d}}}{\delta^{2d + 1}} \: .
\]
However, applying the following Proposition~\ref{propBetaLoc}--Corollary~\ref{corNuDeltaN} instead, we obtain in Theorem~\ref{thmBetaLocNuN}$(i)$ the significantly improved rate (by a factor $\delta^{2d}$):
\[
\E \left[ \beta (\nu_{\delta,N}, \nu_\delta ) \right] \lesssim \frac{N^{-\frac{1}{d}}}{\delta} \: .
\]
The key observation is that we have the following much better Lipschitz estimate for $g$ (for small $\delta, r$) after taking the mean value in \eqref{eqHypg} and restricted to $x, y \in \supp \mu$, indeed, using the $d$--Ahlfors regularity:
\begin{equation} \label{eqLipRandg}
\E \left[ \left| g(x) - g(y) \right| \right] \leq \left(\frac{C_0 \kappa_1 + C_0^2 \kappa_3}{\delta} + \frac{C_0 \kappa_2}{r} + \kappa_0 \right) |x-y| \: .
\end{equation}
Unfortunately, such mean Lipschitz estimate \eqref{eqLipRandg} does not allow to apply straightforwardly \cite{dud}, though adapting the strategy, we were able to obtain the following result:

\begin{proposition} \label{propBetaLoc}
Let $\mu$ be a probability measure in $\R^n$ and assume that $\mu$ is $d$--Ahlfors regular for some real number $0 < d \leq n$, with regularity constant $C_0 \geq 1$ (see Definition~\ref{dfnAhlfors}).
For $\delta, r,\kappa, \kappa_0, \kappa_1, \kappa_2 \in ]0,+\infty[$, where $\kappa_0$ is allowed to depend on $\delta, r$,\footnote{For instance $\kappa_0 = \frac{M}{\delta} + \frac{M}{r}$ is used in the proof of Theorem~\ref{thmCvHolderV}.} we denote by $\mathfrak{X}$ the set of random Lipschitz functions satisfying $\| g \|_\infty \leq \kappa \kappa_0$ and for all $x$, $y \in \supp \mu$, 
\begin{equation} \label{eqHypg}
\left| g(x) - g(y) \right| \leq \left( \frac{\kappa_1}{\delta} \Delta_{\delta,N} (x,y)  + \frac{\kappa_2}{r} \Delta_{r,N} (x,y) + \frac{\kappa_3}{\delta} \Delta_{\delta,N} (x,y)\Delta_{r,N} (x,y) + \kappa_0 \right) |x-y| \: .
\end{equation}
Then,
\begin{enumerate}[$(i)$]
\item case $d>2$ and $B \subset \R^n$ arbitrary bounded open set: there exists a constant $M = M(d,C_0) > 0$ such that for any $B \subset \R^n$ bounded open set, for all $\delta > 0$, $r>0$ and for all $N \in \N^\ast$ large enough so that $N^{-\frac{1}{d}} \leq \min ( \delta, r)$ we have
\begin{multline*}
\E \left[ \sup \left\lbrace \left| \int_B g \: d \mu_N - \int_B g  \: d \mu \right| \: : \: g \in \mathfrak{X}\right\rbrace \right] \leq M \left( \kappa_0 + \frac{\kappa_1 + \kappa_3}{\delta} + \frac{\kappa_2}{r} \right) N^{-\frac{1}{d}} \mu (B^{\gamma_N}) \\ \quad \text{with } \gamma_N = N^{-\frac{d-2}{d^2}} \xrightarrow[N \to +\infty]{} 0 \: .
\end{multline*}
\item case $B$ open ball of radius $R_B > 0$ and $\kappa \leq R_B$: there exists a constant $M = M(d,C_0) > 0$ such that for any open ball $B \subset \R^n$ satisfying $R_B < 1$, for all $\delta >0$, $r > 0$ and for all $N \in \N^\ast$ large enough so that $N^{-\frac{1}{d}} \leq \min( R_B, r ,\delta)$, we have
\[
\E \left[ \sup \left\lbrace \left| \int_B g \: d \mu_N - \int_B g  \: d \mu \right|  :  g \in \mathfrak{X}\right\rbrace \right]  \leq  M \left( \kappa_0 + \frac{\kappa_1+ \kappa_3}{\delta} + \frac{\kappa_2}{r} \right) \mu (B) \times
\left\lbrace
\begin{array}{lcl}
N^{-\frac{1}{d}}         & \text{if} & d > 2\\
N^{-\frac{1}{2}} \ln N   & \text{if} & d = 2\\
N^{-\frac{1}{2}}         & \text{if} & d < 2
\end{array}
\right. \: .
\]
\end{enumerate}
\end{proposition}

\begin{remark}
 Note that assumption \eqref{eqHypg} naturally arises from the fact that we consider kernel--based estimator, for instance in Section~\ref{secEstimNudeltaN} $\eta$ is a compactly supported kernel with dilations $\eta_\delta = \eta \left( \frac{\cdot}{\delta} \right)$, and thus convolution of the form $\delta^{-d} \mu_N \ast \eta_\delta$ involves (a rough) Lipschitz estimates of the form $\delta^{-1} \Delta_{\delta,N} (x,y)$. Nonetheless, we emphasize that in the piecewise Hölder regularity class $\cP$, a finer Lipschitz estimate can be used (see  \eqref{eqMeanLipThetaN} in Lemma~\ref{lemMeanLipEstimators}) hence leading to a finer convergence rate in Theorem~\ref{thmCvHolderSplit} and Proposition~\ref{propFinalCV}.
\end{remark}

Before proving Proposition~\ref{propBetaLoc}, we draw an easy consequence (Corollary~\ref{coroBetaLoc} below) that we will directly use in the proof of Theorem~\ref{thmBetaLocNuN} while in Sections~\ref{secVarifoldLikeEstimator} and~\ref{secUniformPWHolder} we will need Proposition~\ref{propBetaLoc} itself.

\begin{corollary} \label{coroBetaLoc}
Let $\mu$ be a probability measure in $\R^n$ and assume that $\mu$ is $d$--Ahlfors regular for some real number $0 < d \leq n$, with regularity constant $C_0 \geq 1$ (see Definition~\ref{dfnAhlfors}).
Assume that $h : \R^n \rightarrow \R$ is a random Lipschitz function satisfying: there exist $\kappa_0, \kappa_1 \in [0, +\infty[$ such that $\| h \|_\infty \leq \kappa_0$ and for all $x$, $y \in \supp \mu$ and $\delta$, $r \in ]0, +\infty[$,
\begin{equation} \label{eqHyph}
\left| h(x) - h(y) \right| \leq \frac{\kappa_1}{\delta} \Delta_{\delta,N} (x,y) |x-y| \: .
\end{equation}
Then,
\begin{enumerate}[$(i)$]
\item case $d>2$ and $B \subset \R^n$ arbitrary bounded open set: there exists a constant $M = M(d,C_0) > 0$ such that for any $B \subset \R^n$ bounded open set and for all $N \in \N^\ast$, $\delta > 0$ satisfying $N^{-\frac{1}{d}} \leq \delta$ we have
\[
\E \left[ \beta_B (h \: \mu_N, h \: \mu) \right] \leq M \left( \kappa_0 + \frac{\kappa_1}{\delta}
\right) N^{-\frac{1}{d}} \mu (B^{\gamma_N}) \quad \text{with } \gamma_N = N^{-\frac{d-2}{d^2}} \xrightarrow[N \to +\infty]{} 0 \: .
\]
and in particular for all $N \in \N^\ast$,
\begin{equation} \label{eqpropBetaLoc1}
\E \left[ \beta_B (\mu_N, \mu) \right] \leq M  N^{-\frac{1}{d}} \mu (B^{\gamma_N}) \quad \text{with } \gamma_N = N^{-\frac{d-2}{d^2}} \xrightarrow[N \to +\infty]{} 0 \: .
\end{equation}
\item case $B$ open ball: there exists a constant $M = M(d,C_0) > 0$ such that for any open ball $B \subset \R^n$ of radius $0 < R_B < 1$ and for all $N \in \N^\ast$, $\delta > 0$ satisfying $N^{-\frac{1}{d}} \leq \min( R_B, \delta)$, we have
\[
\E \left[ \beta_B (h \: \mu_N, h \: \mu)  \right] \leq  M \left( \kappa_0 + \frac{\kappa_1 R_B}{\delta} 
\right) \mu (B) \times
\left\lbrace
\begin{array}{lcl}
N^{-\frac{1}{d}}         & \text{if} & d > 2 \\
N^{-\frac{1}{2}} \ln N   & \text{if} & d = 2\\
N^{-\frac{1}{2}}         & \text{if} & d < 2
\end{array}
\right. .
\]
and in particular for all $N \in \N^\ast$ satisfying $N^{-\frac{1}{d}} \leq R_B$,
\begin{equation} \label{eqpropBetaLoc2}
\E \left[ \beta_B (\mu_N, \mu) \right] \leq  M  \mu (B) \times
\left\lbrace
\begin{array}{lcl}
N^{-\frac{1}{d}}         & \text{if} & d > 2 \\
N^{-\frac{1}{2}} \ln N   & \text{if} & d = 2\\
N^{-\frac{1}{2}}         & \text{if} & d < 2
\end{array}
\right. .
\end{equation}
\end{enumerate}
\end{corollary}

We recall that $B^{\gamma_N}$ is the $\gamma_N$--enlargement of $B$ (see \eqref{eqThick}).
If we compare the result of Corollary~\ref{coroBetaLoc}: \eqref{eqpropBetaLoc1} and \eqref{eqpropBetaLoc2} with \eqref{eqDudleyBeta} obtained in \cite{dud}, we observe that the control is renormalized by the $\mu$--mass of the set $B$ relatively to which the distance between $\mu$ and $\mu_N$ is estimated. In \eqref{eqDudleyBeta}, $B$ contains $\supp \mu$ and thus has maximal mass $1$.

\begin{proof}[Proof of Corollary~\ref{coroBetaLoc}]
First of all, \eqref{eqpropBetaLoc1} and \eqref{eqpropBetaLoc2} are direct consequences of the general statement with $h = 1$ that is, $\kappa_0 = 1$ and $\kappa_1 = 0$ (see \eqref{eqHyph}).
Then, as already mentioned, Corollary~\ref{coroBetaLoc} is a straightforward consequence of Proposition~\ref{propBetaLoc}. Indeed, let $B \subset \R^n$ be an open bounded set and $f \in \xC_c (\R^n, \R)$ be a $1$--Lipschitz function such that $\| f \|_\infty \leq 1$ and $\supp f \subset B$. Actually note that $\| f \|_\infty \leq \kappa$ with $\kappa = \min (1, \xdiam(B)/2)$ and in particular $\kappa \leq R_B$ in the case $(ii)$ where $B$ is a ball of radius $R_B> 0$. Let $h : \R^n \rightarrow \R$ be a random Lipschitz function satisfying \eqref{eqHyph} then, $g = fh$ satisfies $\| g \|_\infty \leq \kappa \kappa_0$ and the following Lipschitz estimate type: for all $x$, $y \in \supp \mu$,
\begin{align*} 
\left| g(x) - g(y) \right| & \leq  |f(x)| |h(x) - h(y)| + |h(y)| |f(x) -f(y)| \nonumber \\
& \leq \left( \frac{\kappa \kappa_1}{\delta} \Delta_{\delta,N} (x,y) 
+ \kappa_0 \right) |x-y| \: ,
\end{align*}
that is exactly \eqref{eqHypg} with $\widetilde{\kappa_1} = \kappa \kappa_1$, $\kappa_2 = \kappa_3 = 0$,
$\kappa_0$ and $\kappa$. Consequently, 
\[
\beta_B ( h \: \mu_N , h \: \mu ) \leq \sup \left\lbrace \left| \int_B g \: d \mu_N - \int_B g  \: d \mu \right| \: : \: g \in \mathfrak{X}\right\rbrace
\]
and we can apply Proposition~\ref{propBetaLoc} to conclude the proof of Corollary~\ref{coroBetaLoc}.
\end{proof}

We are left with the proof of Proposition~\ref{propBetaLoc}:

\begin{proof}[Proof of Proposition~\ref{propBetaLoc}]
Let $B \subset \R^n$ be an open bounded set. We fix $N \in \N^\ast$, $ \delta > 0$ and $r > 0$ satisfying $N^{-\frac{1}{d}} \leq \min (r,\delta)$. We use the notations $\epsilon = N^{-\frac{1}{d}} \in (0,1]$, $\varepsilon_u := 3^{-u}$ for $u \in \N$, and $T := B \cap \supp \mu$.

\noindent We define integers $0 \leq s \leq t$ such that
\[
 3^{-(t+1)}<  \epsilon \leq 3^{-t} = \varepsilon_t \quad \text{and}\quad 
3^{-(s+1)}< \epsilon^{\alpha} \leq 3^{-s} = \varepsilon_s \: ,
\]
where $0 < \alpha \leq 1$ is defined hereafter in Step $4$ depending on the case $(i)$ or $(ii)$.
Thanks to Proposition~\ref{propPackingPartition}, we know that $T$ can be partitioned with $m_u$ pieces of diameter $\leq \varepsilon_u$ and
\begin{equation} \label{eqmr}
m_u \leq 4^d C_0 \varepsilon_u^{-d} \mu \left(B^\frac{\varepsilon_u}{4} \right) \: .
\end{equation}
We can apply Lemma~\ref{lemNestedPartitions} to such partitions and define nested partitions $\left\lbrace A_j^u \: : \: u = s, \ldots, t, \, j= 1, \ldots, m_u \right\rbrace$ satisfying \eqref{eqPartitionProperty} and \eqref{eqNestedProperty}.
For each $u=s,\dots , t$ and $j=1,\dots,m_u$ we choose $x^u_j\in A_j^u$ and 
we introduce 
\[
M_u := \sum_{j=1}^{m_u} | \mu_N(A^u_j)-\mu(A^u_j) | 
\quad \text{and} \quad
I_u:=\left|\sum_{j=1}^{m_u}g(x_j^u)(\mu_N(A^u_j)-\mu(A^u_j))\right| \quad \text{for } g \in \mathfrak{X}
 \: .
\]

\smallskip
\noindent {\bf Step $1$:} we can prove the following control:
\begin{align} \label{eqBetag}
\left| \int_B g \: \left( d \mu_{N} -  d \mu \right) \right|  \leq I_t + 9 \kappa_0 \epsilon \left( \mu_N(B) + \mu(B) \right) & + 9 \epsilon  \frac{\kappa_1}{\delta^{d+1}} \int_T \mu_N \left( B(x, 10 \delta) \right) \: d (\mu_N + \mu)(x) \nonumber\\
& + 9 \epsilon  \frac{\kappa_2}{r^{d+1}} \int_T \mu_N \left( B(x, 10 r) \right) \: d (\mu_N + \mu)(x) \nonumber \\
& + 9 \epsilon  \frac{\kappa_3}{\delta^{d+1} r^d} \int_T \mu_N \left( B(x, 10 \delta) \right) \mu_N \left( B(x, 10 r) \right) \: d (\mu_N + \mu)(x). 
\end{align}
Indeed, we remind that $(A_j^t)_{j=1}^{m_t}$ is a partition of $T = B \cap \supp \mu$ and for all $j$, $\xdiam A_j^t \leq 3 \varepsilon_t \leq 9 \epsilon = 9 N^{-\frac{1}{d}} \leq 9 \delta$, therefore $B(x,\delta) \cup B(x_j^t,\delta) \subset B(x,10\delta)$ so that $\Delta_{\delta, N}(x,x_j^t) \leq \delta^{-d} \mu_N \left( B(x, 10 \delta) \right)$ and similarly 
$\Delta_{r, N}(x,x_j^t) \leq r^{-d} \mu_N \left( B(x, 10 r) \right)$.
Using in addition \eqref{eqPartitionProperty} and \eqref{eqHypg} we obtain 
\begin{align*}
& \Bigg| \int_B  g  \: \left( d \mu_{N} -  d \mu \right) \Bigg|  = \left| \sum^{m_t}_{j=1} \int_{A_j^t} g(x) \: d(\mu_N-\mu)(x)\right| 
\leq I_t + \left| \sum^{m_t}_{j=1} \int_{A_j^t} \left(g(x) - g(x_j^t) \right) \: d(\mu_N-\mu)(x)\right| \\
& 
\leq  I_t+ \sum^{m_t}_{j=1}\int \limits_{A_j^t}\left( \kappa_0 + \frac{\kappa_1}{\delta^{d+1}} \mu_N \left( B(x,10 \delta) \right)  + \frac{\kappa_2}{r^{d+1}} \mu_N \left( B(x,10 r)\right) \right. \nonumber \\
& \left. \qquad \qquad \qquad \quad \; \; \, + \frac{\kappa_3}{\delta^{d+1} r^d}  \mu_N \left( B(x,10 \delta) \right)\mu_N \left( B(x,10 r) \right) \right) (3 \varepsilon_t) d(\mu_N+\mu)(x)\\
& \leq I_t +  9 \kappa_0 \epsilon \left( \mu_N(B) + \mu(B) \right) \\
& + 9 \epsilon  \int_T \left[ \frac{ \kappa_1}{\delta^{d+1}} \mu_N \left( B(x, 10 \delta) \right) + \frac{\kappa_2}{r^{d+1}} \mu_N \left( B(x, 10 r) \right) + \frac{\kappa_3}{\delta^{d+1} r^d}  \mu_N \left( B(x,10 \delta) \right)\mu_N \left( B(x,10 r) \right) \right] \: d (\mu_N + \mu)(x) \: .
\end{align*}

\smallskip
\noindent {\bf Step $2$:} noting that the dependence on $g$ only lies in $I_t$ in the r.h.s. of \eqref{eqBetag}, let us check that
\begin{equation} \label{eqpropBetaLoc1Step2}
 \E \left[ \sup \left\lbrace \left| \int_B g \: d \mu_N - \int_B g  \: d \mu \right| \: : \: g \in \mathfrak{X}\right\rbrace \right] \leq  \E \left[ \sup \left\lbrace I_t \: : \: g \in \mathfrak{X} \right\rbrace \right] + M \epsilon \left( \kappa_0 + \frac{\kappa_1 + \kappa_3}{\delta} + \frac{\kappa_2}{r} \right) \mu(B) \: .
\end{equation}
Indeed, taking the mean value in the other terms (different from $I_t$) of \eqref{eqBetag}, we first have $\E \left[ \mu_N(B) + \mu(B) \right] = 2 \mu(B)$.
Then, from \eqref{eqBasicMuNTer} and \eqref{eqBasicMuN4}, we have
\[
 \E \left[ \frac{1}{\delta^d} \int_T \mu_N \left( B(x,10 \delta) \right) \: d (\mu_N + \mu)(x) \right] \leq M \mu(T) \quad \text{and similarly with $r$ instead of } \delta \: ,
\]
and furthermore
\[
 \E \left[ \frac{1}{\delta^d r^d} \int_T \mu_N \left( B(x,10 \delta )\right) \mu_N \left( B(x,10 r )\right)  \: d (\mu_N + \mu)(x) \right] \leq M \mu(T) \: ,
\]
which yields \eqref{eqpropBetaLoc1Step2}.

\smallskip
\noindent {\bf Step $3$:} we now estimate $I_t$:
\begin{equation} \label{eqItg}
\E \left[ \sup \left\lbrace I_t \: : \: g \in \mathfrak{X} \right\rbrace \right]
\leq \frac{M}{N^\frac{1}{2}} \mu \left( B^\frac{\varepsilon_s}{4} \right) \left[ \kappa \kappa_0 \varepsilon_s^{-\frac{d}{2}} + 9 \left( \kappa_0 + \frac{\kappa_1 + \kappa_3}{\delta} + \frac{ \kappa_2}{r} \right)  \sum_{u=s+1}^t \varepsilon_u^{-\frac{d}{2} +1} \right] \: .
\end{equation}
First recall that $\| g \|_\infty \leq \kappa \kappa_0$ so that we roughly have $I_s \leq \kappa \kappa_0 M_s$.
Let now $s -1 \leq u \leq t$, using \eqref{eqNestedProperty} and the Lipschitz estimate \eqref{eqHypg} for $g$ we obtain the following recurrence relation:
\begin{align}
I_u & \leq\Bigg| \sum^{m_{u-1}}_{q=1} \sum_{j \in I_{q,u}}  \left( g(x_j^u) -g(x_q^{u-1}) \right)  \: \left(\mu_N (A_j^u) -\mu(A_j^u) \right) \Bigg| + \Bigg|\sum^{m_{u-1}}_{q=1} g(x_q^{u-1}) \: \underbrace{ \sum_{j \in I_{q,u}}  \left(\mu_N (A_j^u) -\mu(A_j^u) \right) }_{= \mu_N(A_q^{u-1}) - \mu (A_q^{u-1}) } \Bigg| \nonumber \\
& \leq  \sum^{m_{u-1}}_{q=1} \sum_{j \in I_{q,u}} \left( \kappa_0 + W_{j,q}^u(\delta,r) \right)(3 \varepsilon_{u-1})   \left| \mu_N (A_j^u) -\mu(A_j^u) \right|  + I_{u-1} \nonumber \\
& \quad \text{with} \quad W_{j,q}^u(\delta,r) = \frac{ \kappa_1}{\delta} \Delta_{\delta,N} \left( x_j^{u},x_q^{u-1} \right)  + \frac{\kappa_2}{r} \Delta_{r,N} \left( x_j^{u},x_q^{u-1} \right) + \frac{ \kappa_3}{\delta} \Delta_{\delta,N} \left( x_j^{u},x_q^{u-1} \right) \Delta_{r,N} \left( x_j^{u},x_q^{u-1} \right)  \nonumber \\
& \leq 9 \varepsilon_u \:   \sum^{m_{u}}_{j=1}  W_{j,q(j)}^u(\delta,r) \left| \mu_N (A_j^u) -\mu(A_j^u) \right|  + 9 \kappa_0 \varepsilon_{u} M_u + I_{u-1}  \: ,
\label{eqItRec} 
\end{align}
where $q(j) \in \{1, \ldots, m_{u-1} \}$ is such that $j \in I_{q,u}$ (that is $q(j)$ is the index of the set $A_{q(j)}^{u-1}$ containing $A_j^u$).
By induction from $u=t$ down to $s$ on \eqref{eqItRec} and recalling that $I_s \leq \kappa \kappa_0 M_s$, we have the following control:
\begin{equation} \label{eqItBis}
I_t \leq \kappa \kappa_0 M_s + 9 \sum_{u=s+1}^t \varepsilon_u \left[ \kappa_0 M_u +  \sum^{m_{u}}_{j=1} W_{j,q(j)}^u(\delta,r)  \left| \mu_N (A_j^u) -\mu(A_j^u) \right| \right]  \: .
\end{equation}
Note that the right hand side of \eqref{eqItBis} is now independent of $g$ and we can then proceed with taking the mean value.
We recall that applying Lemma~\ref{L1} $(i)$ and using the bound \eqref{eqmr} on $m_u$, $\varepsilon_u \leq \varepsilon_s$ and $\displaystyle \mu(T) \leq \mu \left( B^\frac{\varepsilon_s}{4} \right)$ we have
\begin{equation} \label{eqmrMuN}
\E\left[ M_u \right] \leq \left( \frac{m_u \mu(T)}{N} \right)^\frac{1}{2} \leq \frac{2^d C_0^\frac{1}{2} }{N^\frac{1}{2}} \varepsilon_u^{-\frac{d}{2}} \mu \left( B^\frac{\varepsilon_s}{4} \right) \: .
\end{equation}
Moreover,\footnote{Note that here, we only know that $ |x_j^u - x_q^{u-1}| \leq \xdiam (A_q^{u-1}) \leq 3 \epsilon_{u-1}$ which is not less than $\delta$ or $r$ in general. Thus we cannot say that the union of the two balls lie in a larger one of radius proportional to $\delta$ as we did when proving Step $1$.} $\Delta_{\delta,N} \left( x_j^{u},x_q^{u-1} \right) \leq \frac{1}{\delta^d} \left( \mu_N \left( B (x_j^{u} , \delta) \right) + \mu_N \left( B (x_q^{u-1} , \delta) \right) \right)$ and similarly with $r$ instead of $\delta$, then by assumption $\min(\delta,r) \geq N^{-\frac{1}{d}}$ so that 
applying Lemma~\ref{L1} $(ii)$ with $R, \widetilde{R} = \delta, r$, $x_j = x_{q(j)}^{u-1}$, $x_j^u$ we have
\begin{align}
\E & \Bigg[ \sum^{m_{u}}_{j=1} W_{j,q(j)}^u(\delta,r)  \left| \mu_N (A_j^u) -\mu(A_j^u) \right|\Bigg] \nonumber \\
& \leq   \left( \frac{m_u \mu(T)}{N} \right)^\frac{1}{2} \left(  \frac{\kappa_1}{\delta^{d+1}} 4 C_0^\frac{3}{2} \delta^d  +   \frac{\kappa_2}{r^{d+1}} 4 C_0^\frac{3}{2} r^d + \frac{\kappa_3}{\delta^{d+1} r^d} 60 C_0^{\frac{7}{2}} \delta^d r^d \right) \nonumber \\
& \leq \frac{M}{N^\frac{1}{2}}  \left( \frac{\kappa_1 + \kappa_3}{\delta} + \frac{ \kappa_2}{r} \right) \varepsilon_u^{-\frac{d}{2}} \mu \left( B^\frac{\varepsilon_s}{4} \right) \: .  \label{eqItBis_0}
\end{align}
Coming back to \eqref{eqItBis} using \eqref{eqItBis_0} and applying \eqref{eqmrMuN} once more to control $\kappa_0 M_u$ and $\kappa \kappa_0 M_s$, we conclude the proof of \eqref{eqItg} (i.e. Step $2$).

\smallskip
We can now draw an intermediate conclusion in the proof of Proposition~\ref{propBetaLoc}, before considering more specifically the different cases at hand in $(i)$ and $(ii)$. Indeed, thanks to \eqref{eqpropBetaLoc1Step2} and \eqref{eqItg} we can infer that

\noindent {\bf Intermediate conclusion:} 
\begin{multline} \label{eqIntermediateConcl}
\E \left[ \sup \left\lbrace \left| \int_B g \: d \mu_N - \int_B g \:  d\mu \right| \: : \: g \in \mathfrak{X} \right\rbrace \right]  \\ \leq M \left(  \kappa_0 + \frac{\kappa_1 + \kappa_3}{\delta} + \frac{\kappa_2}{r}  \right) \epsilon \mu(B) + \frac{ M}{N^\frac{1}{2}}\left[ \kappa_0 \kappa \varepsilon_s^{-\frac{d}{2}} + \left( \kappa_0 + \frac{\kappa_1 + \kappa_3}{\delta} + \frac{\kappa_2}{r} \right)  \sum_{u=s+1}^t \varepsilon_u^{-\frac{d}{2} +1} \right]  \mu \left( B^\frac{\varepsilon_s}{4} \right)  \: .
\end{multline}

\smallskip
\noindent {\bf Step $4$:} 
We are left with estimating $\displaystyle \sum_{u=s+1}^t \varepsilon_u^{-\frac{d}{2} +1}$ whose computation depends on the sign of $-\frac{d}{2} +1$.

{\bf -- case $d > 2$:}
Note that for $d > 2$, $\varepsilon_u^{-d/2+1}=(3^{-u})^{-d/2+1}=(3^{d/2-1})^{u}$ and $3^{d/2-1} > 1$ so that 
\begin{equation} \label{eqBetaT_2}
\sum_{u=s+1}^t\varepsilon_u^{-d/2+1} \leq \sum_{u=0}^t (3^{d/2-1})^{u} \leq \frac{(3^{d/2-1})^{t+1}}{3^{d/2-1} - 1} \leq \frac{3^{d/2-1}}{3^{d/2-1} - 1} \varepsilon_t^{-d/2+1} \leq M \epsilon^{-d/2+1} \leq M N^{1/2} \epsilon \: .
\end{equation}
In addition, in the case where $B$ is an arbitrary bounded open set, $\kappa \leq 1$ and by choosing $\alpha = \frac{d-2}{d}$ we have 
\begin{equation} \label{eqAlphad}
\kappa \varepsilon_s^{-d/2} \leq \epsilon^{- \alpha d /2} = \epsilon^{-d/2 +1} = N^{1/2} \: \epsilon \text{ and } \varepsilon_s \leq 3 \epsilon^\alpha = 3 \gamma_N \: , 
\end{equation}
then together with \eqref{eqBetaT_2}:
\begin{align}
\frac{ 1 }{N^\frac{1}{2}} & \left[ \kappa_0 \kappa \varepsilon_s^{-\frac{d}{2}} + \left( \kappa_0 + \frac{\kappa_1 + \kappa_3}{\delta} + \frac{\kappa_2}{r} \right) \sum_{u=s+1}^t \varepsilon_u^{-\frac{d}{2} +1} \right]  \mu \left( B^\frac{\varepsilon_s}{4} \right) \nonumber \\
& \leq \frac{ 1 }{N^\frac{1}{2}}\left( \kappa_0 N^\frac{1}{2} \epsilon + \left( \kappa_0 + \frac{\kappa_1 + \kappa_3}{\delta} + \frac{\kappa_2}{r} \right) M N^{\frac{1}{2}} \epsilon \right)  \mu \left( B^{\gamma_N} \right)  \leq M \left( \kappa_0 + \frac{\kappa_1 + \kappa_3}{\delta} + \frac{\kappa_2}{r} \right) \epsilon  \mu \left( B^{\gamma_N} \right) \label{eqBetaT_3} \: .
\end{align}
We can then conclude the proof of $(i)$ ($d> 2$ thanks to \eqref{eqIntermediateConcl} and \eqref{eqBetaT_3}. 

\smallskip
\noindent If now $B$ is an open ball of radius $R_B < 1$ and $\kappa \leq R_B$. Let $N \in \N^\ast$ be large enough so that $\epsilon = N^{-1/d} \leq R_B < 1$ and let $0 < \gamma \leq 1$ such that $R_B = \epsilon^\gamma$. By choosing $\alpha = 1 + \frac{2}{d}(\gamma -1)$, we have $\alpha - \gamma =  \frac{d-2}{d} (1 - \gamma)$ hence $0 < \gamma \leq  \alpha \leq 1$ and
\begin{equation} \label{eqBetaT_4}
\kappa \varepsilon_s^{-d/2} \leq R_B \epsilon^{- \alpha d /2} = \epsilon^{-d/2 + 1} = N^{1/2} \epsilon \quad \text{and} \quad \varepsilon_s \leq 3 \epsilon^\alpha \leq 3 \epsilon^\gamma = 3 R_B \: .
\end{equation}
From \eqref{eqIntermediateConcl}, \eqref{eqBetaT_2} and \eqref{eqBetaT_4} we infer as before that
\[
\E \left[ \sup \left\lbrace \left| \int_B g \: d \mu_N - \int_B g \:  d\mu \: \right| \: : \: g \in \mathfrak{X} \right\rbrace \right]  \leq M(d,C_0) \left( \kappa_0 + \frac{\kappa_1 + \kappa_3}{\delta} + \frac{\kappa_2}{r} \right) N^{-\frac{1}{d}} \mu \left(B^{(R_B)} \right) \: .
\]
Note that $\mu \left(B^{(R_B)} \right) = \mu (2B) \leq 2^d C_0^2 \mu(B)$ by Ahlfors regularity which concludes the proof of $(ii)$ for $d > 2$.

\smallskip
{\bf case $d = 2$:} in this case, $B$ is an open ball of radius $R_B \geq \kappa$ satisfying $\epsilon = N^{-1/d} < R_B < 1$, as $d =2$, we have
$N^{-1/2} = \epsilon$, $\varepsilon_u^{-d/2+1}= 1$, and by choosing $0 < \alpha < 1$ such that $\epsilon^\alpha = R_B$ then $\kappa \varepsilon_s^{-d/2} \leq R_B \epsilon^{-\alpha} \leq 1$ and we obtain
\begin{align}
\frac{ 1 }{N^\frac{1}{2}} & \left[ \kappa_0 \kappa \varepsilon_s^{-\frac{d}{2}} + \left( \kappa_0 + \frac{\kappa_1 + \kappa_3}{\delta} + \frac{\kappa_2}{r} \right) \sum_{u=s+1}^t \varepsilon_u^{-\frac{d}{2} +1} \right]  \mu \left( B^\frac{\varepsilon_s}{4} \right) \nonumber \\
\leq & N^{-\frac{1}{2}} \left( \kappa_0  + \left( \kappa_0 + \frac{\kappa_1 + \kappa_3}{\delta} + \frac{ \kappa_2}{r} \right) |t-s| \right)  \mu \left( B^{\frac{\epsilon_s}{4}} \right) 
\leq \left( \kappa_0 + \left( \kappa_0 + \frac{\kappa_1 + \kappa_3}{\delta} + \frac{\kappa_2}{r} \right) \ln N \right) \epsilon  \mu \left( 2B \right) \label{eqBetaT_5}
\end{align}
since $\epsilon_s \leq 3 \epsilon^\alpha \leq 3 R_B$ and $\displaystyle |t - s| \leq t \leq \frac{|\ln(\epsilon)|}{\ln 3} \leq \frac{\ln N}{2 \ln 3}$.
Using the $2$--Ahlfors regularity of $\mu$ together with \eqref{eqIntermediateConcl} and \eqref{eqBetaT_5} concludes the proof of $(ii)$ for $d = 2$ and eventually the proof of Proposition~\ref{propBetaLoc}.

\smallskip
{\bf-- case $d < 2$:} with the same computations as for $d = 2$ and choosing $\epsilon^\alpha = R_B$, we still have $\kappa \epsilon_s^{-d/2} \leq R_B^{1 - \alpha d/2} \leq 1$. Moreover, $\epsilon_u^{-\frac{d}{2} +1} = (3^{d/2 - 1})^{u}$ with $3^{d/2 - 1} < 1$ so that
\begin{align}
\frac{ 1 }{N^\frac{1}{2}} & \left[ \kappa_0 \kappa \varepsilon_s^{-\frac{d}{2}} + \left( \kappa_0 + \frac{\kappa_1 + \kappa_3}{\delta} + \frac{\kappa_2}{r} \right) \sum_{u=s+1}^t \varepsilon_u^{-\frac{d}{2} +1} \right]  \mu \left( B^\frac{\varepsilon_s}{4} \right) \nonumber \\
\leq & N^{-\frac{1}{2}} \left( \kappa_0  + \left( \kappa_0 + \frac{\kappa_1 + \kappa_3}{\delta} + \frac{ \kappa_2}{r} \right) \frac{1}{1 - 3^{d/2 - 1}} \right)  \mu \left( B^{\frac{\epsilon_s}{4}} \right) 
\leq M N^{-\frac{1}{2}} \left(  \kappa_0 + \frac{\kappa_1 + \kappa_3}{\delta} + \frac{\kappa_2}{r}  \right)  \mu \left( 2B \right) \label{eqBetaT_6} \: .
\end{align}
\end{proof}

\subsection{Estimation of the measure carried by S}
\label{secEstimNudeltaN}

In Section~\ref{secPointwiseDensity}, we have established the pointwise convergence of the estimator $\theta_{\delta_N, N}$ towards $\theta$ when $N \rightarrow +\infty$ and $\delta_N$ suitably chosen. In this section, our purpose is to build upon this result to design an estimator of $\nu = \cH^d_{| S}$. More precisely, given $N \in \N^\ast$ and $\delta > 0$, we introduce the measure $\nu_\delta$ as well as the random measure $\nu_{\delta,N}$ obtained by attributing weights to the points of the sample $(X_1, \ldots, X_N)$ as follows:
\begin{equation} \label{eqNuN}
 \nu_{\delta, N} = \frac{1}{N} \sum_{i=1}^N \Phi(\theta_{\delta, N}(X_i)) \delta_{X_i} = \left( \Phi \circ \theta_{\delta, N} \right) \mu_N \quad \text{and} \quad \nu_\delta = \left( \Phi \circ \theta_{\delta} \right) \mu \: ,
\end{equation}
where $\Phi$ is a truncation of the inverse function, for $0 <\tau \leq 1$ and $t > 0$,
\begin{equation} \label{eqPhiChi}
\Phi (t) = \frac{\chi_\tau(t)}{t} \quad \text{and} \quad
 \chi_{\tau}(t) =\left\{ \begin{array}{ll}
     0& \text{ if }  0<t<\frac{\tau}{2}\\
     \frac{2}{\tau} t - 1 &\text{ if } \frac{\tau}{2} \leq t \leq \tau \\
     1 & \text{ if } t>\tau
\end{array} \right. \: .
\end{equation}
Let us comment on these definitions. Observe that $\displaystyle \cH^d_{| S} = \frac{1}{\theta} \: \mu$ and it would be natural to similarly consider the random measure $\displaystyle \frac{1}{\theta_{\delta, N}} \: \mu_N$, however, while we assume that $\theta$ is uniformly lower bounded, it is more delicate concerning $\theta_{\delta, N}$: it will be possible to infer a lower bound for $\E \left[ \theta_{\delta, N} \right]$ and not directly for $\theta_{\delta, N}$. To circumvent this point, we multiply the inverse function by the cutoff $\chi_\tau$.

\begin{remark}
 Note that \eqref{eqNuN} defines finite measures though not probability measures in general.
\end{remark}

We start with some elementary properties of $\Phi$ and $\theta_{\delta , N}$ that will be subsequently useful. 

\begin{lemma} \label{lemPropertiesPhiThetaN}
Let $\Phi : \R_+ \rightarrow \R_+$ and $\theta_{\delta, N} : \R^n \rightarrow \R_+$ be defined as in \eqref{eqPhiChi} and \eqref{eqThetaN}, then
\begin{enumerate}[$(i)$]
 \item $\Phi$ is bounded and Lipschitz: $\displaystyle \| \Phi \|_\infty \leq \frac{1}{\tau}$ and $\displaystyle \xLip (\Phi) \leq \frac{4}{\tau^2}$,
\item $\theta_{\delta, N}$ is Lipschitz and more precisely, for all $x$, $y \in \R^n$,
\begin{equation*}
 \left| \theta_{\delta, N}(x) - \theta_{\delta, N}(y) \right| \leq \frac{\xLip(\eta)}{C_\eta} \underbrace{\frac{\mu_N \left( B(x,\delta) \cup B(y,\delta) \right)}{ \delta^d}}_{\Delta_{\delta,N}(x,y)} \frac{|x-y|}{\delta} \: . 
\end{equation*}
\end{enumerate}
\end{lemma}

\begin{proof}
We first check $(i)$, $\Phi_\tau$ is continuous and piecewise smooth and
\begin{equation*}
\Phi_{\tau}'(t) =\left\{ \begin{array}{ll}
     0& \text{ if }  0<t<\frac{\tau}{2}\\
      + \frac{1}{t^2} &\text{ if } \frac{\tau}{2} < t < \tau \\
     -\frac{1}{t^2} & \text{ if } t>\tau
\end{array} \right. \: .  
\end{equation*}
We deduce that $\xLip(\Phi)\leq \frac{4}{\tau^2}$ and $||\Phi(t)||_{\infty}\leq\frac{1}{\tau}$.
Concerning $(ii)$, for all $x$, $y \in \mathbb{R}^n$,
\begin{align*}
    |\theta_{\delta, N}(x)-\theta_{\delta,N}(y)|\leq &\frac{1}{C_\eta \delta^d}|\mu_N*\eta_{\delta}(x)-\mu_N*\eta_{\delta}(y)| \\
     \leq& \frac{1}{C_\eta \delta^d}\int_{z \in \R^n} \left |\eta \left(\frac{|x-z|}{\delta}\right)-\eta\left(\frac{|y-z|}{\delta}\right)\right| d\mu_N(z)\\
     \leq&\frac{1}{C_\eta \delta^d} \xLip(\eta)\frac{|x-y|}{\delta}\mu_N(B(x,\delta)\cup B(y,\delta))  \: .
\end{align*}
\end{proof}

By triangular inequality, we can write
\begin{align} \label{eqTrianIneqBeta}
\E \left[ \beta_B (\nu_{\delta, N}, \: \cH^d_{| S}) \right]\leq \beta_B (\nu_{\delta}, \: \cH^d_{| S}) +\E \left[ \beta_B (\nu_{\delta,N}, \: \nu_{\delta}) \right]
\end{align}
and we study the convergence of $\beta_B (\nu_{\delta}, \: \cH^d_{| S})$ in Proposition~\ref{nudel}, then $\E \left[ \beta_B (\nu_{\delta,N}, \: \nu_{\delta}) \right]$ in Theorem~\ref{thmBetaLocNuN} before concluding in Corollary~\ref{corNuDeltaN}. Note that in the three aforementioned statements, we assume that $\mu = \theta \cH^{| S}$ with $S$, $\theta$ satisfying \hyperref[hypH1]{$(H_1)$}, \hyperref[hypH2]{$(H_2)$} and \hyperref[hypH3]{$(H_3)$} and we recall that it is by definition equivalent to $\mu \in \cQ$.

\begin{proposition}
\label{nudel}
Let $d \in \N^\ast$ and assume that $S$, $\theta$ satisfy \hyperref[hypH1]{$(H_1)$}, \hyperref[hypH2]{$(H_2)$} and \hyperref[hypH3]{$(H_3)$} and let $C_0 \geq 1$ be a regularity constant of $\mu = \theta \cH^d_{| S}$ (see Remark~\ref{rkHyp}$(i)$).
We recall that $\theta_\delta : \R \rightarrow \R_+$ is defined as in \eqref{eqThetaN}. Then, 
there exists $m = m(d, C_0,\eta) > 0$ such that
 \begin{equation} \label{eqmtau}
  \forall \delta > 0 , \, \forall x \in S , \, \theta_{\delta} (x) \geq m \quad \text{and} \quad \theta(x) \geq m \: .
 \end{equation}
Assume that $0< \tau \leq m$ and let $\chi_\tau : (0,+\infty) \rightarrow \R_+$ be as in \eqref{eqPhiChi}, then,
\begin{equation*} 
  \forall \delta > 0 , \, \forall x \in S , \quad \Phi\left( \theta_\delta(x) \right) = \frac{1}{\theta_\delta(x)} \quad \text{and} \quad \Phi\left( \theta(x) \right) = \frac{1}{\theta(x)} \: ,
\end{equation*}
and for any bounded open set $B \subset \R^n$,
\begin{equation*}
 \beta_B (\nu_\delta, \cH^d_{| S}) \leq |\nu_\delta - \cH^d_{| S}|(B)  \leq \int_B \left| \frac{1}{\theta_\delta} - \frac{1}{\theta}  \right| \: d\mu \xrightarrow[\delta \to 0_+]{} 0 \: .
\end{equation*}
\end{proposition}

\begin{proof}
Let $\displaystyle m_\eta = \min \left\lbrace \eta(t) \: : \: t \in \left[0 , \tfrac{1}{2} \right] \right\rbrace > 0$ since $\eta$ is supposed to be positive and continuous in $\left[ -\frac{1}{2}, \frac{1}{2} \right]$.
Let $x \in S$ and $\delta > 0$,
\[
\theta_\delta (x) = \frac{1}{C_\eta \delta^d} \int_{B(x,\delta) \cap S} \eta \left( \frac{|x-y|}{\delta} \right) \theta(y) \: d \cH^d(y) \geq m_\eta \frac{\mu \left( B \left(x, \frac{\delta}{2}\right) \right) }{C_\eta \delta^d} \geq  2^{-d} m_\eta C_\eta^{-1} C_0^{-1}  \: ,
\]
and since $\theta \geq \theta_{min} > 0$, we can infer \eqref{eqmtau}.

\noindent Let now $0< \tau \leq m$, then for all $x \in S$ and $\delta > 0$, $\theta_\delta(x) \geq \tau$ and $\theta(x) \geq \tau$ hence $\displaystyle \chi_\tau(\theta_\delta(x)) = \chi_\tau(\theta(x)) = 1$. Furthermore, we already know that $\theta_\delta \xrightarrow[\delta \to 0_+]{} \theta$ a.e. in $S$ thanks to Proposition~\ref{lemDensityPointwiseCv} and by \eqref{eqmtau}
$\displaystyle
 \left| \frac{1}{\theta_\delta} - \frac{1}{\theta}  \right| \leq \frac{2}{m} 
$ in $S$.
As $\mu( B ) < +\infty$, we conclude by dominated convergence that
\begin{equation} \label{eqLemmaNudelConcl}
 \beta_B (\nu_\delta, \cH^d_{| S}) \leq |\nu_\delta - \cH^d_{| S}|(B) \leq \int_B \left| \frac{1}{\theta_\delta} \chi_\tau (\theta_\delta) - \frac{1}{\theta}  \right| \: d\mu = \int_{B \cap S} \left| \frac{1}{\theta_\delta} - \frac{1}{\theta}  \right| \: d\mu \xrightarrow[\delta \to 0_+]{} 0 \: .
\end{equation}
\end{proof}
\noindent Note that we similarly have the uniform upper bound, for $x \in S$:
\begin{equation}
\label{eqThetaDeltaUpperBound}
 \theta_\delta (x) = \frac{1}{C_\eta \delta^d} \int_{B(x,\delta) \cap S} \eta \left( \frac{|x-y|}{\delta} \right) \: d \mu(y) \leq \frac{\| \eta \|_\infty}{C_\eta} \frac{\mu \left( B \left(x, \delta\right) \right) }{\delta^d} \leq  \frac{\| \eta \|_\infty C_0}{C_\eta} \leq M(C_0, \eta) \: .
\end{equation}
\begin{theorem} \label{thmBetaLocNuN}
Let $0 < d \leq n$ and assume that $S$, $\theta$ satisfy $(H_1)$ and $(H_2)$ and let $C_0 \geq 1$ be the regularity constant of $\mu = \theta \cH^d_{| S}$ (see Remark~\ref{rkHyp}$(i)$).
Let $\nu_{\delta}$ and $\nu_{\delta,N}$ be defined as in \eqref{eqNuN}. Then,
\begin{enumerate}[$(i)$]
\item case $d>2$ and $B$ arbitrary bounded open set: there exists a constant $M = M(d,C_0,\xLip(\eta)) > 0$ such that for any $B \subset \R^n$ bounded open set and for all $N \in \N^\ast$, $0 < \delta < 1$ satisfying $N^{-\frac{1}{d}} \leq \delta$,
\[ 
\E \left[ \beta_B (\nu_{\delta, N}, \: \nu_\delta) \right] \leq  \frac{M}{\tau^2} \frac{N^{-\frac{1}{d}} }{\delta} \mu (B^{\gamma_N})  \quad \text{with } \gamma_N = N^{-\frac{d-2}{d^2}} \xrightarrow[N \to +\infty]{} 0 \: .
\]
\item case $B$ open ball: there exists a constant $M = M(d,C_0,\xLip(\eta)) > 0$ such that for any open ball $B \subset \R^n$ of radius $0 < R_B < 1$ and for all  $N \in \N^\ast$, $0 < \delta < 1$ satisfying $N^{-\frac{1}{d}} \leq \min( R_B, \delta)$, 
\[
\E \left[ \beta_B (\nu_{\delta, N}, \: \nu_\delta) \right] \leq \frac{M}{\tau^2} \frac{R_B}{\delta} \mu (B) \times
\left\lbrace
\begin{array}{lcl}
N^{-\frac{1}{d}}         & \text{if} & d > 2\\
N^{-\frac{1}{2}} \ln N   & \text{if} & d = 2\\
N^{-\frac{1}{2}}         & \text{if} & d < 2
\end{array}
\right.  \: .
\]
\end{enumerate}
\end{theorem}
\begin{proof}
Let $N \in \N^\ast$ and $\delta > 0$ satisfying $\delta \geq N^{-\frac{1}{d}}$.
Let $B \subset \R^n$ be an open bounded set and $T = B \cap S$.
Let $f \in {\xC}_c (\R^n, \R)$ be a $1$--Lipschitz function such that $\| f \|_\infty \leq 1$ and $\supp f \subset B$.
\begin{align}
\left| \int_B f \: d \nu_{\delta, N} - \int_B f \: d \nu_\delta \right| & = \left| \int_B f \Phi(\theta_{\delta,N}) \: d \mu_{N} - \int_B f \Phi (\theta_\delta) \: d \mu \right| \nonumber \\
& \leq \left| \int_B f \Phi(\theta_{\delta,N}) \: \left( d \mu_{N} -  d \mu \right) \right|  + \| f \|_\infty \int_B \left| \Phi(\theta_{\delta,N}) - \Phi(\theta_\delta) \right|  \: d \mu \: . \label{eqThmBetaCv_0}
\end{align}

\noindent We first deal with the second term in the right hand side of \eqref{eqThmBetaCv_0}. Thanks to Lemma~\ref{lemPropertiesPhiThetaN} and Proposition~\ref{propDensityPointwiseCv}, there exists $C = C(d,C_0,\eta) > 0$ such that
\begin{align} \label{eqThmBetaCv_1}
\E \left[ \int_B \left| \Phi(\theta_{\delta,N}) - \Phi(\theta_\delta) \right|  \: d \mu \right] & \leq \frac{4}{\tau^2} \int_B \E \left[ \left| \theta_{\delta,N} - \theta_\delta \right| \right] \: d\mu \leq \frac{M}{\tau^2} \frac{1}{\sqrt{N \delta^d}}  \mu(B) = \frac{M}{\tau^2} \left( \frac{N^{-\frac{1}{d}}}{\delta} \right)^\frac{d}{2} \mu(B) \nonumber \\
& \leq \frac{M}{\tau^2} \frac{1}{\delta} \mu (B) \times
\left\lbrace
\begin{array}{lcl}
N^{-\frac{1}{d}}         & \text{if} & d > 2\\
N^{-\frac{1}{2}}         & \text{if} & d = 2\\
N^{-\frac{1}{2}}         & \text{if} & d < 2
\end{array}
\right.
\: ,
\end{align}
where the last inequality follows from the assumption $\displaystyle \frac{N^{-\frac{1}{d}}}{\delta} \leq 1$ implying 
$ \displaystyle \left( \frac{N^{-\frac{1}{d}}}{\delta} \right)^\frac{d}{2} \leq \frac{N^{-\frac{1}{d}}}{\delta}$ for $d \geq 2$, while for $d < 2$, we have $\displaystyle \frac{1}{\sqrt{N \delta^d}} = N^{-\frac{1}{2}} \delta^{-\frac{d}{2}} \leq N^{-\frac{1}{2}} \delta^{-1}$ since $0 < \delta < 1$.

\noindent As for the first term in the right hand side of \eqref{eqThmBetaCv_0}, we can apply Corollary~\ref{coroBetaLoc} with $h:= \Phi (\theta_{\delta, N})$. Indeed, thanks to Lemma~\ref{lemPropertiesPhiThetaN}, $\| h \|_\infty \leq \frac{1}{\tau}$ and for all $x$, $y \in S$,
\begin{align*}
|h(x)-h(y)| & \leq  \left| \Phi \left( \theta_{\delta,N} (x) \right) -  \Phi \left( \theta_{\delta,N} (y) \right) \right| \leq  \frac{4}{\tau^2} \left| \theta_{\delta,N} (x) -  \theta_{\delta,N} (y) \right| 
 \leq \frac{4}{\tau^2} \frac{\xLip(\eta)}{C_\eta} \Delta_{\delta,N}(x,y) \frac{|x-y|}{\delta} \: ,
\end{align*}
which concludes the proof of Theorem~\ref{thmBetaLocNuN}
applying Corollary~\ref{coroBetaLoc} with $\kappa_0 = \displaystyle \frac{1}{\tau}$ and $\displaystyle \kappa_1 = \frac{4 \xLip(\eta)}{\tau^2 C_\eta}$, and additionally noting that in the case where $B$ is a ball of radius $0 < R_B < 1$ one has $\| f \|_\infty \leq R_B$ in \eqref{eqThmBetaCv_0}.
\end{proof}

\noindent Eventually combining Proposition~\ref{nudel},  Theorem~\ref{thmBetaLocNuN} and triangular inequality \eqref{eqTrianIneqBeta}
leads to the Corollary~\ref{corNuDeltaN}. Note that unlike Theorem~\ref{thmBetaLocNuN}, the convergence obtained in Corollary~\ref{corNuDeltaN} is no longer uniform in the regularity class $\cQ$ for it is not uniform either in Proposition~\ref{nudel}.

\begin{corollary} \label{corNuDeltaN}
Let $d \in \N^\ast$ and assume that $S$, $\theta$ satisfy \hyperref[hypH1]{$(H_1)$}, \hyperref[hypH2]{$(H_2)$} and \hyperref[hypH3]{$(H_3)$} and let $C_0 \geq 1$ be a regularity constant of $\mu = \theta \cH^d_{| S}$.
Let $\nu_{\delta}$ and $\nu_{\delta,N}$ be defined as in  \eqref{eqNuN}.
Let $B \subset \R^n$ be an open ball of radius $0 < R_B < 1$ and
let $\nu_{\delta,N}$ be defined as in \eqref{eqNuN}. Let $(\delta_N)_N \subset (0,1)^{\N^\ast}$ be a sequence satisfying
\[
\delta_N \xrightarrow[N \to +\infty]{} 0 \quad \text{and} \quad \frac{1}{\delta_N}  \left\lbrace 
\begin{array}{lcl}
N^{-\frac{1}{d}}  \xrightarrow[N \to +\infty]{} 0 & \text{if} & d > 2 \\
N^{-\frac{1}{2}} \ln N  \xrightarrow[N \to +\infty]{} 0 & \text{if} & d = 2 \\
N^{-\frac{1}{2}}   \xrightarrow[N \to +\infty]{} 0 & \text{if} & d < 2
\end{array}
\right. \: .
\]
Then, there exists a constant $m = m(d,C_0,\xLip(\eta)) \in (0,1)$ only depending on $d$, $C_0$ and $\eta$ such that for any $0 < \tau \leq m$,
\[
\E \left[ \beta_B (\nu_{\delta_N, N}, \: \cH^d_{| S}) \right] \xrightarrow[N \to +\infty]{} 0 \: .
\]
\end{corollary}

\begin{remark}[Choice of the parameter $\tau$]  \label{eqtauchoice}
 It is possible to assume that $\tau$ is fixed and chosen so that $\tau \in \left[ \frac{m}{2} , m \right]$ (with the notation of Proposition~\ref{nudel} above) and therefore $\tau$ only depends on $d$, $C_0$, $\eta$ and could be absorbed in the generic rate convergence constant $M$ in the above statements (Theorem~\ref{thmBetaLocNuN} --  Corollary~\ref{corNuDeltaN}). However, we note that $\tau$ does not only appear in such a constant in front of the convergence rate: choosing $\tau$ is already required to define the estimator $\nu_{\delta_N, N}$. If we drop the $\tau$ dependency, we implicitly assume that we are able to fulfill the requirement $\frac{m}{2} \leq \tau \leq m$, that relies on $C_0$ that may not be explicit in general, and we would need to estimate it first. Another possibility is to keep track of the $\tau$ dependency and choose $(\tau_N)_N \subset (0,1)^{\N^\ast}$ tending to $0$ to define $\nu_{\delta_N, N}$ and under the assumptions of Corollary~\ref{corNuDeltaN}, we first retrieve
 \[
\E \left[ \beta_B (\nu_{\delta_N, N}, \: \nu_{\delta_N}) \right] \xrightarrow[N \to +\infty]{} 0 \quad \text{as soon as} \frac{N^{-\frac{1}{d}} }{\tau_N^2 \delta_N} \xrightarrow[N \to +\infty]{} 0  \: .
\]
Then, for $N \geq N_0$ large enough, we have $\tau_N \leq m$ and the conclusion of Proposition~\ref{nudel} still holds for $\delta_N \to 0$:
\[
\beta_B (\nu_{\delta_N}, \: \nu ) \xrightarrow[N \to +\infty]{} 0 \: ,
\]
so that the convergence result stated in Corollary~\ref{corNuDeltaN} is true provided that $\delta_N$ is adapted to $\tau_N$. In the next sections, we do not keep track of constant $\tau$ and rather absorb the $\tau$ dependency in a generic constant denoted by $M$.
Nonetheless, estimating such a parameter $\tau$ remains an important issue as already mentioned at the very end of Section~\ref{SecCommentsPerspectives}.
\end{remark}

\section{Tangent space estimation and varifold-type estimation}
\label{secVarifoldLikeEstimator}

The next step towards the varifold inference from $\mu_N$ is to define a convergent estimator for the tangent space. We analyse a classical way of estimating the tangent space at $x$ relying on a weighted covariance matrix centered at $x$. We investigate both the case where the covariance matrix is computed directly from $\mu_N$ (i.e. considering $\sigma_{r,\delta,N}$ as defined in \eqref{eqDfnsigmardeltaN}) or from $\nu_{\delta,N}$, that is after correcting the density (i.e. considering $\Sigma_r (\cdot, \nu_{\delta,N})$ according to Definition~\ref{dfn:covMatrix}).

In Section~\ref{secTgtSpaceEstimator}, we introduce the two aforementioned tangent space estimators $\sigma_{r,\delta,N}$ and $\Sigma_r (\cdot, \nu_{\delta,N})$ whose definitions rely on the notion of covariance matrix associated with a given measure Definition~\ref{dfn:covMatrix}. On one hand, we emphasize that such a covariance matrix $\Sigma_r (\cdot, \lambda) = (C_\phi r^d)^{-1} \lambda \ast \psi_r$ is kernel based, similarly to the way $\Theta_\delta (\cdot , \lambda) = (C_\eta \delta^d)^{-1} \eta_\delta \ast \lambda$ is defined, and consequently some arguments are similar in the proof of Proposition~\ref{propCVSigmaN} for $\Sigma_r$ (see also the proof of Proposition~\ref{propDensityPointwiseCv} for $\Theta_\delta$). On the other hand, the a.e. convergence of the deterministic part $\Sigma_{r} (x, \nu_\delta)$ and $\sigma_{r,\delta}(x)$ towards $\Pi_{T_x S}$ is established in Proposition~\ref{propCVSigmardelta}. We then combine both Proposition~\ref{propCVSigmardelta} and \eqref{eqConcentrationSigmaN0} to state the mean convergence of the tangent space estimators $\sigma_{r_N,\delta_N,N}$ and $\Sigma_{r_N} (\cdot, \nu_{\delta_N,N})$ in the second part of Proposition~\ref{propCVSigmaN}, though without uniform rates nor uniform choice of $(\delta_N)_N$, $(r_N)_N \to 0$ in the regularity class $\cQ$. Handling such a lack of uniformity issue is then the purpose of Sections~\ref{secUniformPWHolder} and~\ref{secSplit}. 
In Section~\ref{secVarifoldTypeEstimator}, we build upon such tangent space estimators to introduce two varifold--type estimators $W_{r,\delta,N}$ and $\widetilde{W}_{r,\delta,N}$ (see \eqref{eq:defWrdeltaN}) of $W_S = \cH^d_{| S} \otimes \delta_{\Pi_{T_x S}}$.
Thanks to the a.e. pointwise convergence obtained for $\Sigma_{r} (x, \nu_\delta)$ and $\sigma_{r,\delta}(x)$ towards $\Pi_{T_x S}$ in Proposition~\ref{propCVSigmaN}, it is possible to obtain their $\xL^1(\mu)$--convergence that leads to the convergence of the deterministic part $\beta (W_{r,\delta}, W_S)$, $\beta (\widetilde{W}_{r,\delta}, W_S)$ in Proposition~\ref{propBoundWrdeltaW}. The mean convergence of $W_{r,\delta,N}$ and $\widetilde{W}_{r,\delta,N}$ towards $W_{r,\delta}$ and $\widetilde{W}_{r,\delta}$ (see \eqref{eq:defWrdelta}) is obtained in Proposition~\ref{prop:CV_WrdeltaN} by application of Proposition~\ref{propBetaLoc}. The resulting convergence $\E \left[ \beta_B (W_{r_N,\delta_N, N}, W_S) \right] \xrightarrow[N \to +\infty]{} 0$ is finally stated in Corollary~\ref{coroCvWrdeltaNW}, with the same lack of uniform rate in the regularity class $\cQ$ as mentioned for the pointwise tangent space and density estimators.

\begin{remark}[Constant $M$] \label{remkCstMSecTgt}
 All along the current section, $M$ stands for a generic constant that may vary from one statement to another one, and only depends on $n, d, C_0, \eta, \phi$.
\end{remark}

\subsection{Pointwise tangent space estimator}
\label{secTgtSpaceEstimator}
We now fix another family of radial kernels: given a Lipschitz even function $\phi : \mathbb{R}\mapsto \mathbb{R}_+$, compactly supported in $(-1,1)$, we define
\begin{equation} \label{eqKernelPhi}
C_\phi = \omega_d \int_{t=0}^1 \phi(t)t^{d+1} \: dt
\quad \text{and for } r > 0, \: x \in \R^n, \, 
\phi_r(x) =  \phi \left(\frac{|x|}{r} \right) \: .
\end{equation}
We also associate with $\phi$ the Lipschitz and compactly supported function $\psi = \psi_\phi \in  \xC_c(\R^n)$ defined as follows, for $z \in \R^n$, we recall that $z \otimes z$ is the matrix of $(i,j)$--coefficient $z_i z_j$ and we set
\begin{equation} \label{eqDfnPsiToPhiKernel}
  \psi(z) = \phi (|z|) z \otimes z \in \xSym_+(n) \quad
  \text{and for } r > 0, \, \psi_r (z) = \psi \left( \frac{z}{r} \right) \: .
\end{equation}
Let us check that $\psi$ is bounded and Lipschitz, and more precisely
\begin{equation} \label{eq:lipPsiij}
\supp \psi \subset B(0,1), \quad  \|\psi\|_\infty \leq \|\varphi\|_\infty  \quad \text{and}  \quad  \xLip (\psi ) \leq  \| \varphi \|_\infty + \xLip(\varphi) \: .
\end{equation}
We recall that the matrix $z \otimes z$ is the orthogonal projector onto the line spanned by $z$ if $|z| = 1$ and consequently, for $z \in \R^n$, $\| z \otimes z \| = | z |^2$.
The first two assertions are then consequences of the definitions of $\phi$ and $\psi$, and we are left with the Lipschitz property.
Let $z, w, x \in \R^n$, $(z \otimes z) x = (x \cdot z) z$ and thus by triangular and Cauchy-Schwarz inequalities, $|\left((z \otimes z) - (w \otimes w)\right)x| = | (x \cdot z) z - (x \cdot w) w | \leq \left( |z| + |w| \right) |z-w| |x|$ so that
\[
 \| z \otimes z - w \otimes w \| \leq \left( |z| + |w| \right) |z-w| \: .
\]
If $z, w \in B(0,1)$, we then have 
\begin{align*}
\| \psi(z) - \psi(w) \| & \leq \| \phi \|_\infty \| z \otimes z - w \otimes w \| + \| w \otimes w \| | \phi(|z|) - \phi(|w|) | \\
& \leq \| \phi \|_\infty (|z|+|w|) |z-w| + | w |^2 \xLip(\phi) | z-w | \\
& \leq \left(2 \| \phi \|_\infty + \xLip(\phi) \right) |z-w| \: .
\end{align*}
If now $z \in B(0,1)$ and $w \notin B(0,1)$, then $\psi(w) = \phi(w) = 0$ and
\[
\| \psi(z) - \psi(w) \| = \| \phi(z) z \otimes z - \phi(w) z \otimes z \| \leq | z |^2 \xLip(\phi) |z-w| \: ,
\]
and we can conclude the proof of \eqref{eq:lipPsiij}.

\begin{definition}[Covariance matrix associated with a measure]
\label{dfn:covMatrix}
Let $0 < d \leq n$, $r > 0$ and let $\lambda$ be a Radon measure in $\R^n$. We define for $x \in \R^n$, the $n \times n$ matrix
\[
\Sigma_r (x, \lambda) =  \frac{1}{C_\phi r^d} \int_{\R^n} \phi \left( \frac{|y-x|}{r} \right) \frac{y-x}{r} \otimes \frac{y-x}{r} \: d \lambda(y) 
= \frac{1}{C_\phi r^d} \int_{\R^n} \psi_r (y-x) \: d \lambda(y) = \frac{\lambda \ast \psi_r (x)}{C_\phi r^d} \: .
\]
Note that $\Sigma_r (x, \lambda) \in \xSym_+(n)$.
\end{definition}
\noindent As it will be useful hereafter, we point out that if $\lambda$ is $d$--Ahlfors regular with regularity constant $C_0$, then for all $x \in \supp \lambda$,
\begin{equation} \label{eqInfBoundSigma}
 \| \Sigma_r (x, \lambda) \| \leq \frac{1}{C_\phi r^d} \| \phi \|_\infty \lambda (B(x,r)) \leq \frac{C_0}{C_\phi} \| \phi \|_\infty \leq M(C_0, \phi) \: .
\end{equation}

In geometric measure theory, the minimal assumption allowing to define a notion of tangent space is rectifiability. We show in Proposition~\ref{propCVSigmar} that as soon as we consider a $d$--rectifiable measure, $\Sigma_r$ converges almost everywhere towards the approximate tangent space.

\begin{proposition} \label{propCVSigmar}
Let $\mu = \theta \cH^d_{| S}$ be a $d$--rectifiable measure, then for $\cH^d$--a. e. $x \in S$,
\[
\forall r > 0, \, \Sigma_r \left(x, \cH^d_{| x+T_x S} \right) = \Pi_{T_x S} \quad \text{and} \quad \Sigma_r (x, \mu) \xrightarrow[r \to 0_+]{} \theta(x) \Pi_{T_x S} \: ,
\]
where we recall that $\Pi_{T_x S}$ is the matrix of orthogonal projection on the approximate tangent space $T_x S$.
In particular, for $\nu = \cH^d_{| S}$,  $\Sigma_r (x, \nu) \xrightarrow[r \to 0_+]{} \Pi_{T_x S}$.
\end{proposition}

\begin{proof}
For $r > 0$ and $x \in S$ (such that $T_x S$ exists), we have by translation and dilation:
\[
 \Sigma_r \left(x, \cH^d_{| x+T_x S} \right) = \frac{1}{C_\phi r^d} \int_{\R^n \cap T_x S} \psi \left(\frac{y}{r} \right) \: d \cH^d(y) = \frac{1}{C_\phi} \underbrace{ \int_{\R^n \cap T_x S} \psi(z) \: d \cH^d(z) }_{ =: \Sigma(x)} \: .
\]
Moreover, as $\psi$ is continuous and compactly supported, by definition of approximate tangent plane (see Definition~\ref{ATP} and Proposition~\ref{propTgtSpaceToRectifMeasure}), for $\cH^d$--a. e. $x \in S$,
\[
\frac{1}{r^d}\int_{\R^n} \psi \left( \frac{y-x}{r} \right) \: d \mu (y) \xrightarrow[r \to 0_+]{} \theta(x)  \int_{\R^n \cap T_x S} \psi(z) \: d \cH^d(z)  = \theta(x) \Sigma(x) \: ,
\]
and it remains to check that $\Sigma(x) = \Pi_{T_x S}$.
Let $(e_1, \ldots, e_n)$ be the canonical basis of $\R^n$.
Let $A \in \mathrm{M}_n(\R)$ be an orthogonal matrix  such that $T_xS=A \: \mathrm{span} \: (e_1,\dots,e_d)$ and we define for all $i = 1, \ldots, n$, $\tau_i = A e_i$ so that $(\tau_1, \ldots, \tau_d)$ is an orthonormal basis of $T_x S$. Let $y \in \R^d \times \{0\}$, then $|Ay| = |y|$ and $(Ay \cdot e_i) = \sum_{k=1}^d a_{ik} y_k = \sum_{k=1}^d (\tau_k \cdot e_i) y_k$, and
thanks to the change of variables $z = Ay$, we obtain noting that $A$ is an isometry:
\begin{align} 
(\Sigma(x))_{ij} & =\int_{\mathbb{R}^d\times\{0\}} \varphi(|Ay|) (Ay \cdot e_i ) (Ay \cdot e_j) \: d\mathcal{H}^d(y) = \sum_{k,l = 1}^d (\tau_k \cdot e_i) (\tau_l \cdot e_j) \int_{\mathbb{R}^d\times\{0\}} \varphi(|Ay|) y_k y_l \: d\mathcal{H}^d(y) \nonumber \\
& = \sum_{k,l = 1}^d (\tau_k \cdot e_i) (\tau_l \cdot e_j) \int_{r=0}^\infty \varphi(r) \int_{ y \in \R^d \times \{0\} \cap S(0,r)} y_k y_l \: d\mathcal{H}^{d-1}(y) \: dr \: . \label{eqSigmaCoarea}
\end{align}
Note that if $k\neq l$, by symmetry,
$\displaystyle \int_{\R^d \times \{0\} \cap S(0,r)}y_ky_ld\mathcal{H}^{d-1}(y)= 0$, and if $k = l$:
$$\int_{S(0,r)}y_k^2\mathcal{H}^{d-1}(y)=\frac{1}{d}\sum_{m=1}^d\int_{S(0,r)}y_m^2\mathcal{H}^{d-1}(y)=\frac{1}{d}r^2\mathcal{H}^{d-1}(S(0,r))=r^{d+1} \omega_d \: .$$
Using the above observations in \eqref{eqSigmaCoarea} we can conclude the proof of Proposition~\ref{propCVSigmar}:
\begin{equation*}
(\Sigma(x))_{ij} = \sum_{k = 1}^d (\tau_k \cdot e_i) (\tau_k \cdot e_j) \omega_d \int_{r=0}^\infty \varphi(r) r^{d+1} \: dr = C_\phi \left( \Pi_{T_x S} \right)_{ij} \: .
\end{equation*}
\end{proof}
Proposition~\ref{propCVSigmar} suggests at least two strategies in order to estimate the approximate tangent space at $x \in S$:
\begin{enumerate}[$(i)$]
 \item as $\Sigma_r (x, \nu) \xrightarrow[r \to 0_+]{} \Pi_{T_x S}$, one can consider $\Sigma_r(x,\nu)$, $\Sigma_{r} (x, \nu_{\delta})$ and $\Sigma_{r} (x, \nu_{\delta,N})$,
 \item or similarly, as $\Sigma_r (x, \mu) \xrightarrow[r \to 0_+]{} \theta(x) \Pi_{T_x S}$, one can alternatively consider 
\begin{equation} \label{eqDfnsigmardeltaN}
\sigma_r(x) = \Phi (\theta(x)) \Sigma_r(x,\mu) , \quad
\sigma_{r,\delta}(x) = \Phi \left(\theta_{\delta}(x) \right) \Sigma_{r} (x,\mu)
\quad \text{and} \quad
\sigma_{r,\delta, N}(x) = \Phi \left(\theta_{\delta,N}(x) \right) \Sigma_{r} (x,\mu_N) \: ,
\end{equation}
\end{enumerate}
and we recall that for $x \in S$, $ \Phi(\theta(x)) = \frac{1}{\theta(x)}$ and $ \Phi(\theta_\delta(x)) = \frac{1}{\theta_\delta(x)}$ (see Proposition~\ref{nudel}).
The difference between both choices is the following: on one hand, with $\Sigma_r ( x , \nu_{\delta,N})$ we compute the covariance matrix at $x \in S$ with respect to the measure $\nu_{\delta,N}$, that is taking into account the density after correction~; on the other hand, with $\sigma_{r,\delta,N}$ we compute the covariance matrix $\Sigma_r (\cdot, \mu_N)$ directly with respect to the empirical measure, not taking into account any density correction, and then we globally multiply by $\Phi \left( \theta_{ \delta, N} \right)$ in order to obtain the correct density in the limit. Both choices are reasonable, nevertheless
decoupling density and covariance estimation is significantly more adapted to apply concentration inequality as evidenced below in Proposition~\ref{propCVSigmaN}.
We start with investigating the convergence of the deterministic part $\Sigma_r(x, \nu_\delta)$, $\sigma_{r,\delta}(x)$ towards $\Pi_{T_x S}$ in Proposition~\ref{propCVSigmardelta}.
\begin{proposition}
\label{propCVSigmardelta}
 Assume that $S$ and $\theta$ satisfy $(H_1)$, $(H_2)$ and $(H_3)$ and let $\mu = \theta \cH^d_{| S}$. Then, for $x \in S$ such that $T_x S$ exists (in particular for $\cH^d$--a.e. $x \in S$),
 \begin{equation} \label{eqsigmardeltaCV}
  \| \sigma_{r,\delta}(x) - \Pi_{T_x S} \| \leq M \left( \left| \theta_\delta(x) - \theta(x) \right| + \left\| \Sigma_r (x,\mu) - \theta(x) \Pi_{T_x S} \right\| \right) \xrightarrow[\delta,r \to 0_+]{} 0 \: ,
 \end{equation}
and
\begin{equation} \label{eqSigmardeltaCV}
  \| \Sigma_{r}(x, \nu_\delta) - \Pi_{T_x S} \| \leq M \left( \frac{1}{r^d} \int_{B(x,r)} \left| \theta_\delta - \theta \right| \: d \mu + \left\| \Sigma_r (x,\nu) - \Pi_{T_x S} \right\| \right) \: .  
 \end{equation}
\end{proposition}
\begin{remark} \label{remkNonUnifCVdeltar}
Note that the convergence of $\Sigma_{r}(x, \nu_\delta)$ can be inferred provided that $\delta$, $r$ tending to $0$ additionally satisfy $\frac{1}{r^d} \int_{B(x,r)} \left| \theta_\delta - \theta \right| \: d \mu \rightarrow 0$, as done in the proof of Proposition~\ref{propCVSigmaN}.
\end{remark}
\begin{proof}
 We recall that $\Phi (\theta(x)) \theta(x) = 1$ (see Proposition~\ref{nudel}) and $\Sigma_r( \cdot, \mu)$ is uniformly bounded by $\frac{C_0 \| \phi \|_\infty}{C_\phi}$ (see \eqref{eqInfBoundSigma}).
 Hence, if there exists an approximate tangent space $T_x S$ at $x \in S$, we recall that $\| \Phi \|_\infty + \xLip(\Phi) \leq M$ thanks to Lemma~\ref{lemPropertiesPhiThetaN}$(i)$ and thus,
\begin{align*}
\left\| \sigma_{r,\delta}(x) - \Pi_{T_x S} \right\| & = \left\| \Phi \left( \theta_\delta(x) \right) \Sigma_r(x,\mu)  - \Phi (\theta(x)) \theta(x) \Pi_{T_x S} \right\| \\
& \leq \left| \Phi(\theta_\delta(x)) - \Phi(\theta(x))  \right| \| \Sigma_r (x,\mu) \| + |\Phi(\theta(x))| \| \Sigma_r(x,\mu) - \theta(x) \Pi_{T_x S} \| \\
& \leq M \left( | \theta_\delta(x) - \theta(x) | + \| \Sigma_r(x,\mu) - \theta(x) \Pi_{T_x S} \| \right) \xrightarrow[\delta,r \to 0_+]{} 0 \: ,
\end{align*}
where the convergence to $0$ holds thanks to Proposition~\ref{lemDensityPointwiseCv} and Proposition~\ref{propCVSigmar}. 
Then, for $x \in S$, recalling that $\| \psi_r \|_\infty \leq \| \phi \|_\infty$,
\begin{align*}
 \left\| \Sigma_r(x,\nu_{\delta}) - \Sigma_r(x,\nu) \right\| & = \frac{1}{C_\phi r^d} \left\| \int_{B(x,r)} \psi_r (y-x) \Phi (\theta_\delta(y)) \: d \mu(y)- \int_{B(x,r)} \psi_r (y-x) \Phi (\theta(y)) \: d \mu(y) \right\| \\
 & \leq \frac{\| \phi \|_\infty}{C_\phi r^d} \xLip(\Phi) \int_{B(x,r)} |\theta_\delta - \theta | \: d \mu
\end{align*}
and we can conclude the proof of Proposition~\ref{propCVSigmardelta} by triangular inequality.
\end{proof}
We can now infer in Proposition~\ref{propCVSigmaN} the mean convergence of both tangent space estimators $\sigma_{r, \delta, N}$ and $\Sigma_r (\cdot , \nu_{\delta,N})$.
\begin{proposition}[Pointwise tangent space estimator] \label{propCVSigmaN}
Assume that $S$ and $\theta$ satisfy $(H_1)$, $(H_2)$ and $(H_3)$ and let $\mu = \theta \cH^d_{| S}$.
Let $\mu_N$ be the empirical measure associated with $\mu$. Let $\nu_{\delta}$ and $\nu_{\delta,N}$ be defined as in \eqref{eqNuN}.
Then there exists a constant $M \geq 0$
such that for all $x \in S$ and for all $\delta,r >0$, $N \in \N^\ast$ 
\begin{equation} \label{eqConcentrationSigmaN0}
\E  \left[ \left\| \Sigma_{r}(x, \mu_N) - \Sigma_{r}(x, \mu) \right\| \right] \leq \frac{M}{\sqrt{N r^d}}
\quad \text{and} \quad 
\E \left[ \left\| \sigma_{r,\delta,N}(x) - \sigma_{r,\delta}(x) \right\| \right] 
 \leq M \left(\frac{1}{\sqrt{N r^d}} + \frac{1}{\sqrt{N \delta^d}} \right) \: ,
\end{equation}
and furthermore assuming $N^{-\frac{1}{d}} \leq \delta, r < 1$,
\begin{equation}
\label{Sigmaestim}
\mathbb{E}\left[\left\| \Sigma_r (x, \nu_{\delta,N}) - \Sigma_r (x, \nu_\delta )  \right\|\right]\leq \frac{M}{\delta} \left\lbrace
\begin{array}{lcl}
N^{-1/d}    & \text{if} & d > 2\\
N^{-1/2} \ln N  & \text{if} & d = 2\\
N^{-1/2}    & \text{if} & d < 2
\end{array}
\right. \: .  
\end{equation}
Moreover,  
there exist $(\delta_N)_N$, $(r_N)_N \subset (0,1)^{\N^\ast}$ both tending to $0$ such that for $\cH^d$--a.e. $x \in S$, 
\begin{equation} \label{eqPointwiseCVSigma2}
\E \left[\left\| \sigma_{r_N,\delta_N,N}(x) - \Pi_{T_x S} \right\|\right] \xrightarrow[N \to +\infty]{} 0 
\quad \text{and} \quad
\E \left[\left\| \Sigma_{r_N} (x, \nu_{\delta_N,N}) - \Pi_{T_x S} \right\|\right] \xrightarrow[N \to +\infty]{} 0 \: .
\end{equation}
\end{proposition}
\begin{remark} \label{remkFasterRatesigmaVsSigma}
For the sake of simplicity, let us consider the case $r = \delta$ and let us recall that we are interested in the convergence regime that is $N \delta^d \to + \infty$ so that the rate $(N \delta^d)^{-\frac{1}{2}}$ obtained in \eqref{eqConcentrationSigmaN0} is faster than the rate $(N \delta^d)^{-\frac{1}{d}}$ obtained in \eqref{Sigmaestim} as soon as $d > 2$.
\noindent Note also that the a.e. convergence of $\Sigma_r (x, \nu_\delta)$ to $\Pi_{T_x S}$ is proven assuming some relation between $\delta$ and $r$ tending to $0$ while in the case of $\sigma_{r,\delta}$ it holds as soon as $\delta$ and $r$ tend to $0$ without further requirement. 
\end{remark}
\begin{proof}
Let us start with \eqref{eqConcentrationSigmaN0}, similarly to the proof of Proposition~\ref{propDensityPointwiseCv}. By linearity, $\displaystyle \E \left[  \Sigma_{r} (x , \mu_N) \right] =  \Sigma_{r} (x , \mu) $ and note that since for $A = (a_{ij}) \in \xM_n(\R)$, $\| A \|^2 \leq \| A \|_F^2 = \sum_{ij} a_{ij}^2$,
\begin{align*}
 \E  \left[ \left\| \Sigma_{r}(x, \mu_N) - \Sigma_{r}(x, \mu) \right\|^2 \right] & \leq \sum_{i,j = 1}^n \xvar \left( \Sigma_{r}(x, \mu_N)_{ij} \right)  = \frac{1}{C_\phi^2 r^{2d}} \sum_{i,j = 1}^n \frac{1}{N} \xvar \left( (Z_1)_{ij} \right) \leq \frac{n^2 \|\phi\|_\infty}{C_\phi^2 N r^{2d}} \E \left[ \one_{ \{ |X_1 - x| < r \} } \right] \\
 & \leq \frac{n^2 C_0 \|\phi\|_\infty^2}{C_\phi^2 N r^{d}} \: ,
\end{align*}
where $Z_1, \ldots, Z_N$ are independent and satisfy
\[
 \Sigma_{r} (x , \mu_N) = \frac{1}{C_\phi r^{d}} \frac{1}{N} \sum_{k=1}^N Z_k \quad \text{with} \quad Z_k = \psi \left( \frac{X_k - x}{r} \right) \quad \text{and in particular,} \quad \|Z_k \| \leq \| \phi \|_\infty  \: .
\]
We obtain
\begin{equation} \label{eqEfronSteinSigma}
\E\left[\left\| \Sigma_{r} (x , \mu_N)  -\Sigma_{r} (x , \mu ) \right\| \right]\leq \frac{M}{\sqrt{N r^d}}  \: .
\end{equation}
Moreover, thanks to the uniform bounds $\| \Sigma_r (\cdot , \mu) \| \leq M$ and $\| \Phi \|_\infty + \xLip(\Phi) \leq M$ (see \eqref{eqInfBoundSigma} and Lemma~\ref{lemPropertiesPhiThetaN}$(i)$), we have
\begin{align} \label{eqsigmardeltaNrdelta}
 \left\| \sigma_{r,\delta,N}(x) - \sigma_{r,\delta}(x) \right\|  & =  \left\| \Phi \left( \theta_{\delta,N}(x) \right) \Sigma_r (x, \mu_N) - \Phi \left( \theta_{\delta}(x) \right) \Sigma_r (x, \mu) \right\|  \nonumber \\
& \leq M \left\| \Sigma_r (x, \mu_N) -  \Sigma_r (x, \mu) \right\| + \| \Sigma_r (x, \mu) \| \left| \Phi \left( \theta_{\delta,N}(x) \right)  - \Phi \left( \theta_{\delta}(x) \right)  \right| \nonumber \\
& \leq M \left( \left\| \Sigma_r (x, \mu_N) -  \Sigma_r (x, \mu) \right\| +  \left|  \theta_{\delta,N}(x)  - \theta_{\delta}(x)   \right| \right) \: .
\end{align}
We then infer \eqref{eqConcentrationSigmaN0} thanks to Proposition~\ref{propDensityPointwiseCv} and \eqref{eqEfronSteinSigma} taking the mean value in \eqref{eqsigmardeltaNrdelta} above.\\
Combining the a.e. convergence \eqref{eqsigmardeltaCV} with \eqref{eqConcentrationSigmaN0}, we infer that the convergence $\E \left[\left\| \sigma_{r_N,\delta_N,N}(x) - \Pi_{T_x S} \right\|\right] \xrightarrow[N \to +\infty]{} 0 $ in \eqref{eqPointwiseCVSigma2} holds for any sequence $(r_N)_N$, $(\delta_N)_N$ tending to $0$ and satisfying both $N \delta_N^d \to +\infty$ and $N r_N^d \to +\infty$.\\
We now prove \eqref{Sigmaestim}. First note that in general $\displaystyle \E \left[ \Sigma_r (x , \nu_{\delta,N} ) \right]$ is not equal\footnote{Actually, we have $$\E \left[ \Sigma_r (x , \nu_{\delta,N}) \right] = \frac{1}{C_\phi r^d} \int_{(x_1, \ldots, x_N) \in (\R^n)^N} \psi_r (x_1 -x) \Phi \left( \frac{1}{N} \sum_{k=1}^N \frac{1}{C_\eta \delta^d} \eta \left( \frac{|x_1 - x_k|}{\delta} \right) \right)\: d \mu(x_1) \ldots d \mu(x_N) \: .$$} to $\displaystyle  \Sigma_r (x , \nu_{\delta})$, and we did not manage to directly obtain concentration as we did previously with $\sigma_{r, \delta, N}$. However,
thanks to the property of $\psi$ given in \eqref{eq:lipPsiij}, we have that for $r > 0$ and $x \in \R^n$, the function $\displaystyle f : y \mapsto \frac{1}{C_\phi r^d} \psi \left( \frac{y-x}{r} \right)$ satisfies 
\begin{equation*}
 \xLip(f) \leq \frac{1}{C_\phi r^d} \frac{1}{r} \left( \| \phi \|_\infty + \xLip(\varphi) \right) \quad \text{and} \quad \| f \|_\infty \leq \frac{1}{C_\phi r^d} \| \phi \|_\infty  \quad \text{and} \quad \supp(f) \subset B(x,r) \: ,
\end{equation*}
and we can directly apply Theorem~\ref{thmBetaLocNuN}$(ii)$ in the ball $B = B(x,r)$ of radius $R_B = r < 1$, together with the $d$--Ahlfors regularity of $\mu$, to obtain \eqref{Sigmaestim} as follows:
\begin{align*}
\mathbb{E}\left[\left| \Sigma_r (x, \nu_{\delta,N}) - \Sigma_r (x, \nu_\delta )  \right|\right]  & = \E \left[\left| \int_{\R^n} f \: d \nu_{\delta,N} - \int_{\R^n} f \: d \nu_\delta \right|\right] \\
& \leq \mathbb{E}\left[ \frac{1}{C_\phi r^{d+1}} \left( \| \phi \|_\infty + \xLip(\varphi) \right) \beta_{B(x,r)} (\nu_{\delta,N}, \nu_\delta) \right] \\
& \leq \left\lbrace
\begin{array}{lclcl}
\displaystyle \frac{M}{r^{d+1}} \frac{ r}{\delta} \mu(B(x,r))  N^{-\frac{1}{d}} & \leq &  \displaystyle  M  \frac{N^{-\frac{1}{d}}}{\delta}   & \text{if} & d > 2 \\
&&&& \\
\displaystyle \frac{M}{r^{d+1}} \frac{ r}{\delta} \mu(B(x,r)) N^{-\frac{1}{2}} \ln N & \leq &   \displaystyle M\frac{N^{-\frac{1}{2}}  \ln N }{\delta}  & \text{if} & d = 2\\
&&&& \\
\displaystyle \frac{M}{r^{d+1}} \frac{ r}{\delta} \mu(B(x,r)) N^{-\frac{1}{2}} & \leq &  \displaystyle M \frac{N^{-\frac{1}{2}} }{\delta}   & \text{if} & d < 2
\end{array}
\right. \: .
\end{align*}
We are left with proving the last past of statement \eqref{eqPointwiseCVSigma2}. 
Assuming for instance $d> 2$, the case $d \leq2$ does not differ hereafter, and  
$(\delta_N)_N$, $(r_N)_N$ both tend to $0$ and $\frac{N^\frac{-1}{d}}{\delta_N}$ tends to $0$ as well. We then infer from \eqref{eqSigmardeltaCV} in Proposition~\ref{propCVSigmardelta} and \eqref{Sigmaestim} that 
\begin{align*}
 \mathbb{E}\left[\left\| \Sigma_{r_N} (x, \nu_{\delta_N,N}) - \Sigma_{r_N} (x, \nu)  \right\|\right] \leq M \underbrace{ \frac{N^\frac{-1}{d}}{\delta_N}}_{\longrightarrow 0} + \frac{M}{r_N^{d}} \int_{B(x,r_N)} \left| \theta_{\delta_N} - \theta \right| \: d \mu +  M \underbrace{ \left\| \Sigma_{r_N}(x,\nu) - \Pi_{T_x S} \right\|}_{\longrightarrow 0 \text{ by Proposition~\ref{propCVSigmar}}} \: .
\end{align*}
Finally noting that
$\displaystyle
\frac{1}{r_N^{d}}  \int_{B(x,r_N)} \left| \theta_{\delta_N} - \theta \right| \: d \mu \leq  \frac{1}{r_N^{d}}  \underbrace{\int_{\R^n} \left| \theta_{\delta_N} - \theta \right| \: d \mu }_{\xrightarrow[\delta_N \to 0]{} 0}
$,
one can adapt $(r_N)_N \to 0$ to ensure that this last term converges to $0$ as well.
\end{proof}

%
\subsection{A varifold type estimator}
\label{secVarifoldTypeEstimator}
%

Let $0 < r,\delta < 1$. Building upon the definitions of $\nu_{\delta,N}$ \eqref{eqNuN}, $\Sigma_{r}(x, \nu_{\delta,N})$ (Definition~\ref{dfn:covMatrix}) and $\sigma_{r,\delta,N}$ \eqref{eqDfnsigmardeltaN}, we can now introduce two varifold-type estimators $W_{r,\delta,N}$ and $\widetilde{W}_{r,\delta,N}$ of $V_S$: we define the Radon measures in $\R^n \times {\xSym}_+(n)$
\begin{equation} \label{eq:defWrdeltaN}
 W_{r,\delta,N} = \nu_{\delta, N} \otimes \delta_{\sigma_{r,\delta,N}} \quad \text{and} \quad \widetilde{W}_{r,\delta,N} = \nu_{\delta, N} \otimes \delta_{\Sigma_r(x, \nu_{\delta,N})}
\end{equation}
together with their deterministic counterpart:
\begin{equation} \label{eq:defWrdelta} 
 W_{r,\delta} = \nu_{\delta} \otimes \delta_{\sigma_{r,\delta}} \quad \text{and} \quad \widetilde{W}_{r,\delta} = \nu_{\delta} \otimes \delta_{\Sigma_r(x, \nu_{\delta})} \: .
\end{equation}
We recall, for instance considering $W_{r,\delta,N}$, that with such a definition (see Section~\ref{secVarifolds}) we equivalently have that for all $f \in \xC_c (\R^n \times \xSym_+(n))$,
\begin{equation*}
\int_{\R^n \times {\xSym}_+(n)} f\: d W_{r,\delta,N} = \int_{\R^n} f \left( x , \sigma_{r,\delta,N} (x)\right) \: d \nu_{\delta,N} (x) = \int_{\R^n} f \left( x , \sigma_{r,\delta,N}(x)\right) \Phi (\theta_{\delta,N}(x)) \: d \mu_N(x) \: . 
\end{equation*}
Note that $W_{r,\delta,N}$ (respectively $\widetilde{W}_{r,\delta,N}$, $W_{r,\delta}$, $\widetilde{W}_{r,\delta}$) are Radon measures in $\R^n \times \xSym_+(n)$ but are not $d$--varifolds, indeed the covariance matrices $\sigma_{r,\delta,N}(x) = \Phi(\theta_{\delta,N}(x)) \Sigma_r(x,\mu_N)$ (resp.  $\Sigma_r(x, \nu_{\delta,N})$, $\sigma_{r,\delta}(x)$, $\Sigma_r(x, \nu_{\delta})$) are generally not orthogonal projectors and can not be identified with elements of $\Gr$.
Nevertheless, considering the Radon measure $W_S = \cH^d_{| S} \otimes \delta_{\Pi_{T_x S}}$ in $\R^n \times \xSym_+(n)$ rather than the $d$--varifold $V_S = \cH^d_{| S} \otimes \delta_{T_x S}$, it is possible to show that $W_{r,\delta,N}$ is a convergent estimator of $W_S$ as stated in Proposition~\ref{prop:CV_WrdeltaN} below. We refer to the last part of Section~\ref{secVarifolds} for more details concerning the identification between $W_S$ and $V_S$. We also recall that the localized Bounded Lipschitz distance is introduced in Definition~\ref{dfnBetaLoc}, we recall that $\beta_D$ stands for $\beta_{D \times \xSym_+(n)}$ hereafter.

\begin{proposition} \label{prop:CV_WrdeltaN}
Assume that $S$ and $\theta$ satisfy \hyperref[hypH1]{$(H_1)$} and \hyperref[hypH2]{$(H_2)$}. 
Then,
\begin{enumerate}[$(i)$]
\item case $d>2$ and $D$ arbitrary bounded open set: there exists a constant $M = M(d,C_0,\eta,\phi) \geq 0$ such that for any $D \subset \R^n$ bounded open set and 
for all $N \in \N^\ast$, $0 < \delta,r < 1$ satisfying $N^{-\frac{1}{d}} \leq \min(\delta, r)$,
\[ 
\left.
\begin{array}{r}
\E \left[ \beta_D (W_{r, \delta, N}, \: W_{r,\delta}) \right] \\
\E [ \beta_D (\widetilde{W}_{r, \delta, N}, \: \widetilde{W}_{r,\delta}) ] 
\end{array} \right\rbrace
\leq  M \frac{N^{-\frac{1}{d}} }{  \min(\delta, r)} \mu (D^{\gamma_N})  \quad \text{with } \gamma_N = N^{-\frac{d-2}{d^2}} \xrightarrow[N \to +\infty]{} 0 \: .
\]
\item case $B$ open ball: there exists a constant $M = M(d,C_0,\eta,\phi) \geq 0$ such that for any open ball $B \subset \R^n$ of radius $R_B < 1$, for all $N \in \N^\ast$, $0< r, \delta < 1$ satisfying $N^{-\frac{1}{d}} \leq \min(\delta, r, R_B)$, 
\begin{equation} \label{eqBetaWrdeltaN}
\left.
\begin{array}{r}
\E \left[ \beta_B (W_{r, \delta, N}, \: W_{r,\delta}) \right] \\
\E [ \beta_B (\widetilde{W}_{r, \delta, N}, \: \widetilde{W}_{r,\delta}) ] 
\end{array} \right\rbrace \leq \frac{M}{ \min(\delta, r)} \mu(B) \times
\left\lbrace
\begin{array}{lcl}
N^{-\frac{1}{d}}  & \text{if} & d > 2 \\
N^{-\frac{1}{2}} \ln N   & \text{if} & d = 2\\
N^{-\frac{1}{2}}       & \text{if} & d < 2
\end{array}
\right. .
\end{equation}
\end{enumerate}
\end{proposition}

\begin{proof}
We recall that $M \geq 0$ stands for a generic constant that may vary from one line to another (see Remark~\ref{remkCstMSecTgt}).
As a first step, we prove the result for $W_{r,\delta,N}$ and then we point out the main differences when considering $\widetilde{W}_{r,\delta,N}$ in order to complete the proof. Let $D$ be a bounded open set in $\R^n$. Let $f \in \xC_c \left( \R^n \times \xSym_+(n) \right)$ be such that $\xLip (f) \leq 1$ and $\| f \|_\infty \leq 1$ and assume that $\supp f \subset D \times \xSym_+(n)$. 

\smallskip
\noindent {\bf Step $1$:} we first prove the statement for $W_{r,\delta,N}$.
By triangular inequality and definition \eqref{eq:defWrdeltaN}, \eqref{eq:defWrdelta} of $W_{r,\delta,N}$ and $W_{r,\delta}$,
\begin{align} \label{eq:flagfoldA123}
\left| \int f \: d W_{r,\delta,N} \right. & \left. - \int f \: d W_{r,\delta} \right| \leq A_1 + A_2 + A_3 \quad \text{with} \quad \\
A_1  & = \left| \int_{\R^n} f \left( x , \sigma_{r,\delta,N}(x) \right) \Phi (\theta_{\delta,N}(x)) \: (d \mu_N(x) - d \mu(x) ) \: \right| \nonumber \\
A_2 & = \left| \int_{\R^n} f \left( x , \sigma_{r,\delta,N}(x) \right) \left( \Phi (\theta_{\delta,N}(x)) - \Phi (\theta_{\delta}(x)) \right)  \:  d \mu(x)  \: \right| \nonumber\\
A_3 & = \left| \int_{\R^n} \left( f \left( x , \sigma_{r,\delta,N}(x) \right) - f \left( x , \sigma_{r,\delta}(x) \right) \right) \Phi (\theta_{\delta}(x))   \:  d \mu(x)  \: \right| \: . \nonumber
\end{align}
We will use throughout the proof that $\displaystyle \| \Phi \|_\infty + \xLip(\Phi) \leq M$ (see Lemma~\ref{lemPropertiesPhiThetaN}(i)). Applying Proposition~\ref{propDensityPointwiseCv}, we first obtain
\begin{equation} \label{eq:flagfoldA2}
\E \left[ A_2 \right]  \leq \| f \|_\infty \xLip( \Phi) \int_D  \E [ |\theta_{\delta,N} -\theta_{\delta}| ] \: d\mu \leq  \frac{M}{\sqrt{N \delta^d}} \mu(D) \: .
\end{equation}
We can then apply Proposition~\ref{propCVSigmaN} to infer
\begin{equation} \label{eq:flagfoldA3}
\E \left[ A_3 \right]  \leq  \xLip(f)\: \| \Phi \|_\infty \int_D \E [ \left\| \sigma_{r,\delta,N} - \sigma_{r,\delta} \right\| ] \: d\mu \leq M \left( \frac{1}{\sqrt{N \delta^d}} + \frac{1}{\sqrt{N r^d}} \right) \mu(D)  \: .
\end{equation}
We are left with $A_1$, we introduce $g \in \xC_c(\R^n)$ such that for $x \in \R^n$, $g(x):=f \left(x,\sigma_{r,\delta,N}(x) \right) \Phi(\theta_{\delta,N}(x))$ and
it remains to check that $g$ satisfies assumption \eqref{eqHypg} of Proposition~\ref{propBetaLoc}.
Let $x$, $y \in \R^n$.
Recalling Definition~\ref{dfn:covMatrix} and \eqref{eqDfnPsiToPhiKernel}, \eqref{eq:lipPsiij}, and using $\xLip(\psi_r) \leq \frac{1}{r}\xLip(\psi)$, we infer,
\begin{align}
\left\|  \Sigma_r(x,\mu_N) -  \Sigma_r(y,\mu_N) \right\| & \leq \frac{1}{C_\phi r^d} \int_{\R^n} \left\| \psi_r(z-x)  - \psi_r (z-y) \right\|  \: d \mu_N(z) \leq M \frac{\xLip(\psi) }{r^{d+1}} \mu_{N} \left(B(x,r) \cup B(y,r) \right) |x-y| \nonumber \\
& \leq \frac{ M }{r} \Delta_{r,N}(x,y) |x-y| \: . \label{eq:lipSigma}
\end{align}
We then recall that by Lemma~\ref{lemPropertiesPhiThetaN} $(ii)$:
\begin{equation} 
\label{eqLipThetaDeltaNTemp}
 \left| \Phi \left( \theta_{\delta,N} (x) \right) - \Phi \left( \theta_{\delta,N}(y) \right) \right| \leq \xLip(\Phi) \frac{M}{\delta} \Delta_{\delta,N}(x,y)|x-y|
\end{equation}
and consequently, from \eqref{eq:lipSigma} and \eqref{eqLipThetaDeltaNTemp} we infer:
\begin{align}
\left\| \sigma_{r,\delta,N}(x) \right. & \left. - \sigma_{r,\delta,N}(y) \right\|  = \left\| \Phi(\theta_{\delta,N}(x)) \Sigma_r(x,\mu_N) - \Phi(\theta_{\delta,N}(y)) \Sigma_r(y,\mu_N) \right\| \nonumber \\
& \leq \| \Phi \|_\infty \frac{ M }{r} \Delta_{r,N}(x,y) |x-y| + \| \Sigma_r(y,\mu_N) \| \frac{M}{\delta} \Delta_{\delta,N}(x,y) |x-y| \nonumber  \\
& \leq M \left( \frac{1}{r} \Delta_{r,N}(x,y) + \frac{1}{\delta} \Delta_{\delta,N}(x,y) \Delta_{r,N}(x,y) \right) |x-y| \: , \label{eq:lipSigmaBis}
\end{align}
where we used $\displaystyle \| \Sigma_r(y,\mu_N) \| \leq \frac{1}{C_\phi r^d} \int_{B(y,r)} \| \psi_r (z-y) \| \: d\mu_N(z) \leq M \frac{\mu_N(B(y,r))}{r^d} \leq M \Delta_{r,N}(x,y)$. Therefore coming back to $g$, we have thanks \eqref{eqLipThetaDeltaNTemp} and \eqref{eq:lipSigmaBis},
\begin{align}
\label{eqBetaLocWrdeltaN_hypg}
|g(x) & - g(y)| \nonumber \\
\leq & | f(x,\sigma_{r,\delta,N}(x))| \left| \Phi \left( \theta_{\delta,N} (x) \right) -  \Phi \left( \theta_{\delta,N} (y) \right) \right|+ \left|  \Phi \left( \theta_{\delta,N} (y) \right) \right| | f(x,\sigma_{r,\delta,N}(x))- f(y,\sigma_{r,\delta,N}(y)) |\nonumber\\
\leq & \| f \|_\infty \xLip(\Phi) \frac{M}{\delta} \Delta_{\delta,N}(x,y)|x-y| + \| \Phi \|_{\infty} \xLip(f) \big(1+\xLip(\sigma_{r,\delta,N}) \big)  |x-y|\nonumber \\
 \leq &  M \left( \frac{1}{\delta} \Delta_{\delta,N}(x,y)  + \frac{1}{r} \Delta_{r,N}(x,y) + \frac{1}{\delta} \Delta_{r,N}(x,y) \Delta_{\delta,N}(x,y)+ 1\right)  |x-y| 
\end{align}
and $g$ therefore satisfies assumption \eqref{eqHypg} of Proposition~\ref{propBetaLoc} with $\displaystyle \kappa_1 = \kappa_2 = \kappa_3 = \kappa_0 (= M)$ and $\kappa = 1$. We can thus apply Proposition~\ref{propBetaLoc}$(i)$ in order to control $\E [|A_1|] = \E \left[ \left| \int g \: d \mu_N - \int g \: d \mu \right| \right]$ and we can then conclude the proof of $(i)$ thanks to \eqref{eq:flagfoldA123}, \eqref{eq:flagfoldA2} and \eqref{eq:flagfoldA3}, noting that by assumption $N \delta^d \geq 1$ and thus $\frac{1}{\sqrt{N \delta^d}}= (N \delta^d)^{-1/2} \leq (N \delta^d)^{-1/d} = \frac{N^{-1/d}}{\delta}$ if $d > 2$ and therefore
$\E [ A_2 ]$ are $\E [A_3]$ are controlled by $\displaystyle M \frac{N^{-\frac{1}{d}} }{  \min(\delta, r)} \mu (D^{\gamma_N})$ as well.

\smallskip
\noindent {\bf Step $2$:} Starting from the similar decomposition \eqref{eq:flagfoldA123} for $\widetilde{W}_{r,\delta,N}$, i.e. replacing $\sigma_{r,\delta, N}$, $\sigma_{r,\delta}$ with
$\Sigma_{r} ( \cdot, \nu_{\delta,N})$, $\Sigma_{r} ( \cdot, \nu_{\delta})$, we follow the same strategy as in Step $1$.
Applying Proposition~\ref{propDensityPointwiseCv} and Proposition~\ref{propCVSigmaN}, we similarly obtain
\begin{equation} \label{eq:flagfoldA23tilde}
\E \left[ A_2 \right] 
\leq \frac{M}{\sqrt{N \delta^d}} \mu(D) 
\quad \text{and} \quad \E \left[ A_3 \right]  \leq \frac{M}{\delta} \mu(D) \times
\left\lbrace \begin{array}{lcl}
N^{-\frac{1}{d}}         & \text{if} & d > 2\\
N^{-\frac{1}{2}} \ln N   & \text{if} & d = 2\\
N^{-\frac{1}{2}}         & \text{if} & d < 2
\end{array} \right.  
\: .
\end{equation}
We are left with $A_1$, we similarly introduce $g \in \xC_c(\R^n)$ s.t. for $x \in \R^n$, $g(x):=f(x,\Sigma_r(x,\nu_{\delta,N}))\Phi (\theta_{\delta,N}(x))$ and
it remains to check that $g$ satisfies assumption \eqref{eqHypg} of Proposition~\ref{propBetaLoc}.
Let $x$, $y \in \R^n$. Similarly to \eqref{eq:lipSigma} and additionally using $| \Phi (\theta_{\delta,N}(z))| \leq M$, we have
\begin{align} 
    \left\| \Sigma_r(x, \nu_{\delta,N})-\Sigma_r(y,\nu_{\delta,N}) \right\| & \leq \frac{1}{C_\phi r^d}
\int_{z \in \R^n} \left\| \psi_r(z-x)  - \psi_r (z-y) \right\| \left|\Phi\left( \theta_{\delta,N}(z)\right) \right| \: d \mu_N(z) \nonumber \\
& \leq \frac{ M }{r} \Delta_{r,N}(x,y) |y-x| \: . \label{eq:lipSigmaTilde}
\end{align}
Coming back to $g$, we have thanks to \eqref{eqLipThetaDeltaNTemp} and \eqref{eq:lipSigmaTilde},
\begin{align} \label{eqBetaLocWrdeltaN_hypgBis}
|g(x) & - g(y)| \nonumber \\
\leq & |f(x,\Sigma_r(x,\nu_{\delta,N}))| \left| \Phi \left( \theta_{\delta,N} (x) \right) -  \Phi \left( \theta_{\delta,N} (y) \right) \right|+ \left|  \Phi \left( \theta_{\delta,N} (y) \right) \right| |f(x,\Sigma_r(x,\nu_{\delta,N}))-f(y,\Sigma_r(y,\nu_{\delta,N})) | \nonumber \\
\leq & \| f \|_\infty \frac{M}{\delta} \Delta_{\delta,N}(x,y) |y-x| + \| \Phi \|_{\infty} \xLip(f) \big(1+\xLip(\Sigma_r(\cdot,\nu_{\delta,N})) \big)  |x-y| \nonumber \\
 \leq & M \left( \frac{1}{\delta} \Delta_{\delta,N}(x,y)  + \frac{1}{r} \Delta_{r,N}(x,y) + 1\right)  |x-y| 
\end{align}
and $g$ again satisfies assumption \eqref{eqHypg} of Proposition~\ref{propBetaLoc} with $\displaystyle \kappa_1 =  \kappa_2 =  \kappa_0$, $\kappa_3 = 0$ and $\kappa = 1$ and applying Proposition~\ref{propBetaLoc} allows to control $\E [|A_1|] = \E \left[ \left| \int g \: d \mu_N - \int g \: d \mu \right| \right]$ and conclude the proof of $(i)$ thanks to \eqref{eq:flagfoldA23tilde} and $\frac{1}{\sqrt{N \delta^d}} \leq \frac{N^{-\frac{1}{d}}}{\delta}$ for $d > 2$.

\smallskip
\noindent {\bf Step $3$:} Assuming moreover that $D = B$ is an open ball of radius $0 < R_B < 1$, we have $\| f \|_\infty \leq R_B$ since $\supp f \subset B \times \xSym_+(n)$ and $\xLip(f) \leq 1$ and thus $\| g \|_\infty  \leq M R_B = \kappa \kappa_0$ now with $\kappa = R_B$ while $\kappa_0 = M$ is unchanged. Therefore $(ii)$ follows similarly from \eqref{eqBetaLocWrdeltaN_hypg}, \eqref{eqBetaLocWrdeltaN_hypgBis} and Proposition~\ref{propBetaLoc}$(ii)$, also noting that for $d \leq 2$, $\frac{1}{\sqrt{N \delta^d}}= \frac{N^{-1/2}}{\delta^{1/2}} \leq \frac{N^{-1/2}}{\delta}$ and therefore
$\E [ A_2 ]$ and $\E [A_3]$ are controlled by the right hand side of \eqref{eqBetaWrdeltaN} in both cases, hence concluding the proof of Proposition~\ref{prop:CV_WrdeltaN}.
\end{proof}

In order to infer the convergence of the varifold-type estimators $W_{r,\delta,N}$ and $\widetilde{W}_{r,\delta,N}$ towards $W$, we are left with investigating the convergence of the deterministic part, which is exactly the purpose of Proposition~\ref{propBoundWrdeltaW}.

\begin{proposition} \label{propBoundWrdeltaW}
Assume that $S$ and $\theta$ satisfy \hyperref[hypH1]{$(H_1)$}, \hyperref[hypH2]{$(H_2)$} and \hyperref[hypH3]{$(H_3)$} and let $D \subset \R^n$ be an open set. Then
\begin{equation} \label{eqCvWrdeltaW}
 \beta_D \left( W_{r,\delta},W_S \right) \xrightarrow[\delta,r\to 0]{} 0 \quad \text{and} \quad \beta_D \left( \widetilde{W}_{r,\delta},W_S \right) \xrightarrow[\delta,r\to 0]{} 0 \: .
\end{equation}
\end{proposition}

\begin{remark}
Note that in the proof of Proposition~\ref{propBoundWrdeltaW}, we establish the following bounds that will be of use in the sequel:
\begin{equation} \label{eqBoundWrdeltaW}
 \beta_D \left( W_{r,\delta} , W_S \right) \leq M \int_{D} |\theta_\delta - \theta | \: d \mu  + M \int_{D}\| \sigma_{r,\delta} (x) - \Pi_{T_x S} \| \: d \mu(x) 
 \: ,
\end{equation}
and 
\begin{equation} \label{eqBoundWrdeltaWTilde}
\left.
\begin{array}{r}
 \beta_D \left(\widetilde{W}_{r,\delta}, W_S \right) \\
 \displaystyle \int_D \left\| \Sigma_r(x,\nu_{\delta}) - \Pi_{T_x S} \right\| \: d \nu(x) 
\end{array}
\right\rbrace
 \leq M \int_{D^r} | \theta_\delta - \theta | \: d \mu + M \int_D  \| \Sigma_r(x, \nu) - \Pi_{T_x S} \| \: d \mu(x)
 \: .
\end{equation}
\end{remark}

\begin{proof}
As in the proof of Proposition~\ref{prop:CV_WrdeltaN}, we first deal with the case of $W_{r,\delta}$ (Step $1$) and then we check that a similar strategy is valid when handling $\widetilde{W}_{r,\delta}$ (Step $2$).
Let $D$ be an open set in $\R^n$. Let $f \in \xC_c \left( \R^n \times \xSym_+(n) \right)$ satisfying $\xLip (f) \leq 1$ and $\| f \|_\infty \leq 1$ and assume that $\supp f \subset D \times \xSym_+(n)$. We recall (and we use hereafter) that $\cH^d_{| S}$ and $\mu$ are equivalent: $\mu = \theta \cH^d_{| S} \leq \theta_{max} \cH^d_{| S}$ and $\cH^d_{| S} = \theta^{-1} \mu \leq \theta_{min}^{-1} \cH^d_{| S}$.

\smallskip
\noindent {\bf Step $1$:} let us write
 \begin{multline} \label{eq:rectifold}
\left| \int f \: d W_{r,\delta}  - \int f \: d W_S \right| \leq A_1 + A_2  \quad \text{with} \quad \\
A_1   = \left| \int_{D} f \left( x , \sigma_{r,\delta}(x) \right) \left(  d \nu_\delta(x)- d \nu(x) \right)  \right| 
\quad \text{and} \quad
A_2  = \left| \int_{D} \left(f \left( x , \sigma_{r,\delta}(x) \right) - f \left( x , \Pi_{T_x S} \right) \right) \: d \nu (x) \right| \: .
\end{multline}
Starting with $A_1$, we have
\begin{equation} \label{eqBetaWrdelta_A1}
 A_1  = \left| \int_D f(x, \sigma_{r,\delta}(x))  \left( \Phi(\theta_\delta(x)) - \Phi(\theta(x)) \right)  d\mu(x) \right| \leq \| f \|_\infty \xLip(\Phi) \int_D \left| \theta_\delta - \theta \right| \: d\mu \: 
\end{equation}
while 
\begin{equation} \label{eqBetaWrdelta_A2}
 A_2 \leq \xLip(f) \int_D \| \sigma_{r,\delta} (x) - \Pi_{T_x S} \| \: d \nu(x) \: .
\end{equation}
Therefore, combining \eqref{eq:rectifold}, \eqref{eqBetaWrdelta_A1} and \eqref{eqBetaWrdelta_A2} we obtain \eqref{eqBoundWrdeltaW}, and then recalling \eqref{eqsigmardeltaCV} (in Proposition~\ref{propCVSigmardelta}) and Proposition~\ref{lemDensityPointwiseCv}, we conclude by dominated convergence that
\begin{equation*}
 \beta_D \left( W_{r,\delta} , W_S \right) \leq M \int_{D \cap S} |\theta_\delta(x) - \theta(x) | + \| \sigma_{r,\delta} (x) - \Pi_{T_x S} \| \: d \cH^d(x) \xrightarrow[\delta, r \to 0_+]{} 0 \: ,
\end{equation*}
where the domination holds with a constant: for $x \in S$, $|\theta_\delta(x)| \leq M$ thanks to \eqref{eqThetaDeltaUpperBound} and $\| \sigma_{r,\delta}(x) \|  \leq \| \Phi \|_\infty M$ thanks to \eqref{eqInfBoundSigma}.

\smallskip
\noindent {\bf Step $2$:} we similarly write
\begin{multline} \label{eq:rectifoldBis}
\left| \int f \: d \widetilde{W}_{r,\delta} - \int f \: d W_S \right| \leq A_1 + A_2 \quad \text{with} \quad \\
A_1  = \left| \int_{D} f \left( x , \Sigma_r(x, \nu_{\delta}) \right) \left(  d \nu_\delta(x)- d \nu(x) \right)  \right| \quad \text{and} \quad
A_2  = \left| \int_{D} \left(f \left( x , \Sigma_r(x, \nu_{\delta}) \right) - f \left( x , \Pi_{T_xS} \right) \right) \: d \nu (x) \right| \: .
\end{multline}
We obtain the same control as in \eqref{eqBetaWrdelta_A1} for $A_1$:
\begin{equation} \label{eqBetaWrdelta_A1Bis}
 A_1  = \left| \int_D f(x,\Sigma_r(x,\nu_{\delta}))  \left( \Phi(\theta_\delta(x)) - \Phi(\theta(x)) \right)  d\mu(x) \right| \leq M \int_D \left| \theta_\delta - \theta \right| \: d\mu \: .
\end{equation}
On the other hand, recalling that $\nu = \cH^d_{| S}$ is $d$--Ahlfors regular and since $B(x,r) \subset D^r$ for $x \in D$, Fubini theorem yields
\begin{align*}
 \frac{1}{r^d} \int_{x \in D} \left( \int_{B(x,r)} |\theta_\delta - \theta| \: d \mu \right) \: d \nu(x) & = \frac{1}{r^d} \int_{y \in D^r} |\theta_\delta(y) - \theta(y) | \int_{x \in B(y,r) \cap D} \: d \nu(x) \: d \mu(y) \\
 & \leq \frac{1}{r^d}\int_{y \in D^r} |\theta_\delta(y) - \theta(y) | \nu (B(y,r)) \: d \mu(y) \\
 & \leq M \int_{D^r} | \theta_\delta - \theta | \: d\mu \: ,
\end{align*}
and then recalling \eqref{eqSigmardeltaCV}, we have
\begin{align}
 A_2 & \leq \xLip(f) \int_D \left\| \Sigma_r(x,\nu_{\delta}) - \Pi_{T_x S} \right\| \: d \nu(x)  \leq M \int_{D^r} | \theta_\delta - \theta | \: d\mu + \int_D \| \Sigma_r (x, \nu) - \Pi_{T_x S} \| \: d \nu
 \: .  \label{eqBetaWrdelta_A2Bis}
\end{align}
We obtain from \eqref{eq:rectifoldBis}, \eqref{eqBetaWrdelta_A1Bis} and \eqref{eqBetaWrdelta_A2Bis}:
\begin{equation*}
 \beta_D \left(\widetilde{W}_{r,\delta}, W_S \right) \leq M \int_{D^r} | \theta_\delta - \theta | \: d \mu + M \int_D  \| \Sigma_r(x,\nu) - \Pi_{T_x S} \| \: d \mu(x) \: ,
\end{equation*}
whence together with \eqref{eqBetaWrdelta_A2Bis} we obtain \eqref{eqBoundWrdeltaWTilde} as well.
Moreover, for any $r \leq 1$, $D^r \subset D^1$
and we similarly conclude by dominated convergence applied together with Proposition~\ref{lemDensityPointwiseCv} and Proposition~\ref{propCVSigmar} that the convergence holds for $\delta, r \rightarrow 0$.
\end{proof}
Note that the convergence \eqref{eqCvWrdeltaW} holds as soon as $\delta$ and $r$ tend to $0$ for both $W_{r,\delta}$ and $\widetilde{W}_{r,\delta}$, while the pointwise convergence of $\Sigma_{r} (\cdot , \nu_\delta)$ additionally requires some relation between $\delta$ and $r$ (see Remark~\ref{remkNonUnifCVdeltar}). We combine Proposition~\ref{prop:CV_WrdeltaN} and Proposition~\ref{propBoundWrdeltaW} to state the following convergence result for $W_{r,\delta, N}$ and $\widetilde{W}_{r,\delta,N}$. Due to the lack of explicit rate in the regularity class $\cQ$ concerning the convergence of the deterministic part stated in Proposition~\ref{propBoundWrdeltaW}, the mean convergence stated in Corollary~\ref{coroCvWrdeltaNW} suffers the same lack of explicit rate. Such an issue is then addressed in Section~\ref{secUniformPWHolder} and~\ref{secSplit} assuming a stronger regularity model $\cP$ defined in \eqref{eqClassP}.

\begin{corollary}
 \label{coroCvWrdeltaNW}
 Assume that $S$ and $\theta$ satisfy \hyperref[hypH1]{$(H_1)$}, \hyperref[hypH2]{$(H_2)$} and \hyperref[hypH3]{$(H_3)$}.
 Let $B$ be an open ball of radius $0 < R_B < 1$ and let $(\delta_N)_N, \: (r_N)_N \subset (0,1)^{\N^\ast}$ be sequences tending to $0$ and satisfying
\[
 \frac{1}{\min(\delta_N, r_N)}  \left\lbrace 
\begin{array}{lcl}
N^{-\frac{1}{d}}  \xrightarrow[N \to +\infty]{} 0 & \text{if} & d > 2 \\
N^{-\frac{1}{2}} \ln N  \xrightarrow[N \to +\infty]{} 0 & \text{if} & d = 2 \\
N^{-\frac{1}{2}}   \xrightarrow[N \to +\infty]{} 0 & \text{if} & d < 2
\end{array}
\right. \: .
\]
Then,
$\displaystyle
\E \left[ \beta_B (W_{r_N,\delta_N, N}, W_S) \right] \xrightarrow[N \to +\infty]{} 0
$ and $\displaystyle
\E \left[ \beta_B (\widetilde{W}_{r_N,\delta_N, N}, W_S) \right] \xrightarrow[N \to +\infty]{} 0
$.
\end{corollary}

\section{Quantitative estimates in a piecewise Hölder regularity class}
\label{secUniformPWHolder}

In Proposition~\ref{lemDensityPointwiseCv} and Proposition~\ref{nudel}, as well as in Proposition~\ref{propCVSigmar} and \ref{propBoundWrdeltaW} the convergence of $\theta_\delta$ to $\theta$, $\Sigma_r$/$\sigma_{r,\delta}$ to $\Pi_{T_x S}$, $\nu_\delta$ to $\nu = \cH^d_{| S}$ and $W_{r,\delta}$/$\widetilde{W}_{r,\delta}$ to $W_S$ do not come with any convergence rate, which also impacts the convergence results of the associated estimators stated in Corollary~\ref{coroDensityPointwiseCv}, in Corollary~\ref{corNuDeltaN}, in Proposition~\ref{propCVSigmaN} and in Corollary~\ref{coroCvWrdeltaNW}. 
The current section is then dedicated to quantifying the aforementioned deterministic convergences (of $\theta_\delta$, $\nu_\delta$, $\Sigma_r$/$\sigma_{r,\delta}$, $W_{r,\delta}$/$\widetilde{W}_{r,\delta}$) under stronger regularity assumptions both for the set $S$ and the density $\theta$. 

In Section~\ref{secPwHolderDef}, we introduce the regularity framework that we implement in the sequel and is referred to as assumptions \hyperref[hypH1]{$(H_1)$} to \eqref{hypH7} in statements and called {\it piecewise Hölder regularity class} $\cP$, see \eqref{eqClassP}. 
Loosely speaking, we assume that the set $S$ is at least uniformly piecewise $\xC^{1,a}$ (see Definition~\ref{dfnPwHolderS}) and the density $\theta$ is similarly uniformly piecewise $\xC^{0,b}$ (see Definition~\ref{dfnPwHolderTheta}): such regularity assumptions allow for a singular set $\fS = S_{sg} \cup \Theta_{sg}$ that is a union of Ahlfors regular sets of lower dimensions.
With such regularity notions at hand, we first quantify in Section~\ref{secUniformPointwise} the pointwise convergence of $\theta_\delta$ to $\theta$ and $\Sigma_r (x, \nu)$/$\sigma_{r,\delta}(x)$ to $\Pi_{T_x S}$, see Proposition~\ref{coroUniformCVHolder}. In Section~\ref{secUniformBLdist}, we subsequently provide similar rates for the Bounded Lipschitz distances between the measures $\nu_\delta$, $\nu$ (in Proposition~\ref{cvnudel}) and $W_{r,\delta}$,$W_S$/ $\widetilde{W}_{r,\delta}$,$W_S$ (in Proposition~\ref{propCvHolderW}). In Section~\ref{secCVvarifold}, we propose to replace $W_{r,\delta,N}$, $\widetilde{W}_{r,\delta,N}$, that are measures in $\R^n \times \xSym_+(n)$ with varifolds $V_{r,\delta,N}$, $\widetilde{V}_{r,\delta,N}$: to this end, we substitute the tangent plane estimators $\sigma_{r,\delta,N}$ and $\Sigma_r( \cdot , \nu_{\delta,N})$ with an orthogonal projector of rank $d$ thanks to a classical (at least from a numerical perspective) truncation process in the eigen decomposition. We conclude Section~\ref{secUniformPWHolder} with Theorem~\ref{thmCvHolderV} that states the mean convergence of $W_{r,\delta,N}$, $\widetilde{W}_{r,\delta,N}$ and $V_{r,\delta,N}$, $\widetilde{V}_{r,\delta,N}$ to $W_S$ consistently with the rates established for the deterministic part in Proposition~\ref{propCvHolderW}.

\subsection{Piecewise Hölder regularity class}
\label{secPwHolderDef}

Let us recall that up to this point, we only assumed $\xL^1$--regularity (with respect to $\nu$) for the density $\theta$. It is well-known that such an hypothesis is sufficient to obtain the pointwise convergence of averages $\oint_{B(x,\delta)}\theta \: d \nu$ at $\nu$--a.e. point $x$, which is somehow connected to the pointwise convergence of $\theta_\delta$ to $\theta$. However, it is not possible to quantify such a convergence in terms of $\delta$ only assuming $\xL^1$--regularity of $\theta$, let us recall the following example:
%
\begin{example} \label{exQuantifyL1}
Let us illustrate that assuming $f \in \xL^1$ is not enough to provide a uniform bound on the pointwise convergence of the averages on small balls.
For $k\in \N$, $k \neq 0$, we consider the $\xL^1$--function $f_k : (0,1) \rightarrow \R$ such that 
\[
 f_k(x) = \left\lbrace \begin{array}{cl}
                      1 & \text{if } \lfloor kx \rfloor \text{ is even}\\
                      0 & \text{otherwise}
                     \end{array}
\right. \: .
\]
The function $f_k$ is $\frac{2}{k}$--periodic and ``equally'' oscillates between $0$ and $1$. Then, the average $\frac{1}{2 \delta} \int_{B(x,\delta)} f_k(y) \: dy$
is close to $f_k(x) \in \{0,1\}$ if $\delta \ll \frac{1}{k}$ is small with respect to the period, whereas it remains close to $\frac{1}{2}$ (whence far from $f_k(x)$) as long as $\delta \gg \frac{1}{k}$. It illustrates that quantifying (in terms of $\delta$ and uniformly with respect to $f_k$ and thus $k$) the almost everywhere pointwise convergence of $\frac{1}{2 \delta} \int_{B(\cdot ,\delta)} f_k(y) \: dy$ to $f_k$ requires to restrict to a class of functions with stronger regularity: the $\xL^1$--assumption does not provide enough control over the oscillation of the function.
\end{example}

We then consider a stronger regularity framework both for $S$ (see Definition~\ref{dfnPwHolderS}) and $\theta$ (see Definition~\ref{dfnPwHolderTheta}), concretely adding assumptions \eqref{hypH4} to \eqref{hypH7} to the current regularity framework consisting of assumptions \hyperref[hypH1]{$(H_1)$} to \hyperref[hypH1]{$(H_3)$}. As already introduced in Section~\ref{secIntroRecVarifold}, we then obtain the regularity class 
\begin{align} \label{eqClassP}
\cP & = \{ \mu = \theta \cH^d_{| S} \: : \: S, \: \theta \text{ satisfy  \hyperref[hypH1]{$(H_1)$} \text{ to } \eqref{hypH7}} \} \nonumber \\
& = \{ \mu = \theta \cH^d_{| S} \: : \: S, \: \theta \text{ are respectively uniformly piecewise } \xC^{1,a}, \; \xC^{0,b} \text{ according to Definitions~\ref{dfnPwHolderS},~\ref{dfnPwHolderTheta}} \}
\end{align}
that depends on $d, \widetilde{C_0}, \eta, \phi, \theta_{min/max}, C = \max(C_{\theta, sg}, C_{S,sg}), R = \min(R_{\theta, sg}, R_{S,sg})$. 
We mention that the following Definitions~\ref{dfnPwHolderS} and~\ref{dfnPwHolderTheta} are close though different (see in particular \eqref{hypH5}) from Definition~7.2 in \cite{BuetLeonardiMasnou}.
\begin{definition}
\label{dfnPwHolderS}
Let $0 < a \leq 1$ and let $S$ be a closed set satisfying $\cH^d(S) < \infty$. We say that $S$ is {\em uniformly piecewise $\xC^{1,a}$} if there exist 
a closed set $S_{sg} \subset S$, $C = C_{S,sg} \geq 1$ and $R = R_{S,sg} \in (0,1)$
such that the following properties hold: 
   \begin{enumerate}
     \item $S$ is $d$--Ahlfors regular with constant $\widetilde{C_0}$ and $S$ is $d$--rectifiable (see \hyperref[hypH1]{$(H_1)$} and \hyperref[hypH3]{$(H_3)$}) ;
    \item $\displaystyle S_{sg} = \bigcup_{l = 0}^{d-1} S_{sg,l}$ and for each $0 \leq l \leq d-1$, $S_{sg,l}$ is closed and $\cH^l_{| S_{sg,l}}$ is $l$--Ahlfors regular: 
       \begin{equation} \tag{$H_4$} \label{hypH4}
       \forall x \in S_{sg,l} \text{ and } 0< r  \leq R, \quad C^{-1} r^l \leq \cH^l(S_{sg,l}\cap B(x,r))\leq C r^{l} \:  ;
       \end{equation}
    \item for all $0<r \leq R$ and for all $x \in S \setminus (S_{sg})^{Cr}$, $S \cap B(x,r)$ is a $\xC^{1,a}$ graph: there exist an open neighborhood $\cU$ of $0$ in $\R^d$ and $u : \cU \rightarrow \R^{n-d}$ of class $\xC^{1,a}$ such that $u(0) = 0$ and
    \begin{equation} \tag{$H_5$} \label{hypH5}
     (id, u) : \cU \rightarrow q(S \cap B(x,r))\text{ is a diffeomorphism and,} \quad \forall y,z \in \cU, \, \| Du(z) - Du(y)\| \leq C |z-y|^a \: ,
    \end{equation}
    where $q$ is an affine isometry sending the affine tangent plane $x + T_x M$ to $\R^d \times \{0\}^{n-d}$.
   \end{enumerate}  
\end{definition}
\begin{remark}[Uniformly piecewise $\xC^{m,a}$ sets]
\label{remkUnifPwSmoothSet}
 Let $m \in \N$, $m \geq 1$ and $0 < a \leq 1$. Though we are not addressing higher regularity issue, note that it is also possible to define 
 \begin{enumerate}[$\bullet$]
  \item uniformly piecewise $\xC^{m}$ sets: replacing $\xC^{1,a}$ with $\xC^{m}$ in \eqref{hypH5} and dropping the Hölder condition on $Du$,
  \item uniformly piecewise $\xC^{m,a}$ sets: replacing $\xC^{1,a}$ with $\xC^{m,a}$ in \eqref{hypH5} and writing the Hölder condition on $D^m u$ instead of $Du$.
\end{enumerate}
\end{remark}
\begin{example}[Submanifold with boundary]
Note that with the above definition of uniformly piecewise $\xC^{1,a}$ set $S$, if $S \subset \R^n$ is a $d$--submanifold with boundary then such a boundary necessarily lies in the singular set $S_{sg}$ (this is more precisely due to \eqref{hypH5}) and the decomposition of the singular set $S_{sg}$ amounts to $S_{sg} = S_{sg,d-1}$, that is exactly the boundary of $S$, all the other $S_{sg,l}$ being empty. As we just mention, the sets $S_{sg,l}$ can be empty and are not required to be disjoint.
\end{example}
\begin{definition}
\label{dfnPwHolderTheta}
Let $0 < b \leq 1$, let $S$ be a closed set satisfying $\cH^d(S) < \infty$ and let $\theta : \R^n \rightarrow \R_+$ be a Borel function. We say that $\theta$ is {\em uniformly piecewise $\xC^{0,b}$} if there exist
a closed set $\Theta_{sg} \subset S$, $C = C_{\theta,sg} \geq 1$ and $R = R_{\theta,sg} \in (0,1)$
such that the following properties hold: 
   \begin{enumerate}
     \item $\theta \in \xL^1(\cH^d_{| S})$, $\int_{\R^n} \theta \: d \cH^d = 1$ and there exist $0 < \theta_{min} \leq \theta_{max} < +\infty$ such that for $\cH^d$--a. e. $x \in S$, $$\theta_{min} \leq \theta(x) \leq \theta_{max} \: .$$ (see \hyperref[hypH2]{$(H_2)$}) ;
     \item $\displaystyle \Theta_{sg} = \bigcup_{l = 0}^{d-1} \Theta_{sg,l}$ and for each $0 \leq l \leq d-1$, $\Theta_{sg,l}$ is closed and $\cH^l_{| \Theta_{sg,l}}$ is $l$--Ahlfors regular: 
       \begin{equation} \tag{$H_6$} \label{hypH6}
       \forall x \in \Theta_{sg,l} \text{ and } 0< r  \leq R, \quad C^{-1} r^l \leq \cH^l(\Theta_{sg,l}\cap B(x,r))\leq C r^{l} \:  ;
       \end{equation}

    \item for all $0<r \leq R$ and for all $x \in S \setminus (\Theta_{sg})^{Cr}$, $\theta$ is $\xC^{0,b}$ in $S \cap B(x,r)$: 
    \begin{equation} \tag{$H_7$} \label{hypH7}
    \forall y,z \in S \cap B(x,r), \, | \theta(z) - \theta(y)| \leq C |z-y|^b \: .
    \end{equation}
    \end{enumerate}
\end{definition}

\noindent Note that though we shortened the terminology, Definition~\ref{dfnPwHolderTheta} above does not tell anything about $\theta$ outside $S$ and it would be more accurate to say that $\theta$ is uniformly piecewise $\xC^{0,b}$ with respect to $S$.

\begin{remark}[Constant $M$] \label{remkCstMSecHolder}
 All along the current section (Section~\ref{secUniformPWHolder}), $M$ stands for a generic constant that may vary from one statement to another one, and only depends on $d, n, \widetilde{C_0}, \eta, \phi, \theta_{min/max}, C = \max(C_{\theta, sg}, C_{S,sg}) \geq 1, \: R = \min(R_{\theta, sg}, R_{S,sg})< 1$ and is consequently uniform in the regularity class $\cP$. If relevant, a more restrictive dependency can be indicated in some statements.
\end{remark}
We will use the following notations for the unions of singular sets:
\begin{equation} \label{notSigmaSing}
\fS = S_{sg} \cup \Theta_{sg} \quad \text{and} \quad \fS_l = S_{sg,l} \cup \Theta_{sg,l} \: .
\end{equation}
The following Lemma~\ref{lemSingSet} draws consequences of \eqref{hypH4} and \eqref{hypH6} concerning the structure of the singular set: each set $\fS_l$ is $l$--Ahlfors regular and the $\cH^d$ measure of a $\rho$--neighbourhood $\fS_l^\rho$ in $S$ of $\fS_l$ behaves like $\cH^l ( \fS_l )$ times the thickness $\rho^{d-l}$ where the exponent $d-l$ corresponds to the co-dimension of $\fS_l$ relatively to $S$.
\begin{lemma} \label{lemSingSet}
Let $0 < a,b \leq 1$.
We assume that $S \subset \R^n$ is uniformly piecewise $\xC^{1,a}$ in the sense of Definition~\ref{dfnPwHolderS} and $\theta : \R^n \rightarrow \R_+$ is uniformly piecewise $\xC^{0,b}$ in the sense of Definition~\ref{dfnPwHolderTheta}. In other words, $S$ and $\theta$ satisfy assumptions \hyperref[hypH1]{$(H_1)$} to \eqref{hypH7}. Then
\begin{enumerate}[$(i)$]
 \item Let $0 < \rho \leq R$ and $D \subset \R^n$ be a Borel set, then for $0 \leq l \leq d-1$,
 \begin{equation} \label{eqThickD}
  \cH^d (\fS_l^\rho \cap D \cap S ) \leq M \rho^{d-l} \cH^l ( \fS_l \cap D^{2\rho} ) \quad \text{and} \quad  \cH^d (\fS^\rho \cap D \cap S ) \leq M  \sum_{l=0}^{d-1} \rho^{d-l} \: \cH^l (\fS_l \cap D^{2\rho}) \: ,
 \end{equation}
 and similarly $\displaystyle \cH^d (S_{sg,l}^\rho \cap D \cap S ) \leq M \rho^{d-l} \cH^l ( S_{sg,l} \cap D^{2\rho} )$ and $\displaystyle \cH^d (\Theta_{sg,l}^\rho \cap D \cap S ) \leq M \rho^{d-l} \cH^l ( \Theta_{sg,l} \cap D^{2\rho} )$.
 \item Up to renaming the constants, assumption \eqref{hypH4}, respectively \eqref{hypH6}, could be equivalently replaced with:\\
 $\displaystyle S_{sg} = \bigcup_{l = 0}^{d-1} S_{sg,l}$ and for each $0 \leq l \leq d-1$, $S_{sg,l}$ is closed and satisfy: for all $x\in S_{sg,l}$ and $0< r \leq R_B \leq R$,
 \begin{equation} \tag{$H_4^\prime$} \label{hypH4prime}
        \cH^l(S_{sg,l}\cap B(x,R_B))\leq C R_B^{l} \quad \text{and} \quad \cH^{d}(S_{sg,l}^r \cap B(x, R_B)\cap S)\leq C r^{d-l}\cH^{l}(S_{sg,l}\cap B(x, R_B)) \: ,
       \end{equation}
respectively,\\
$\displaystyle \Theta_{sg} = \bigcup_{l = 0}^{d-1} \Theta_{sg,l}$ and for each $0 \leq l \leq d-1$, $\Theta_{sg,l}$ is closed and satisfy: for all $x\in \Theta_{sg,l}$ and $0< r \leq R_B \leq R$,
 \begin{equation} \tag{$H_6^\prime$} \label{hypH6prime}
        \cH^l(\Theta_{sg,l}\cap B(x,R_B))\leq C R_B^{l} \quad \text{and} \quad \cH^{d}(\Theta_{sg,l}^r \cap B(x, R_B)\cap S)\leq C r^{d-l}\cH^{l}(\Theta_{sg,l}\cap B(x, R_B)) 
\end{equation}
 \item Let $B \subset \R^n$ be an open ball of radius $R_B \leq \frac{R}{3C}$, and for $0 \leq l \leq d-1$, define $\xi_l = \xi_l(Cr, B)$ such that $\xi_l = 1$ if $B \cap (\fS_l)^{Cr} \neq \emptyset$ and $\xi_l = 0$ otherwise. Then for all $0 < r \leq R_B$,
 \begin{equation} \label{remkBallCenterSsing}
  \cH^d ((\fS_l)^{Cr} \cap B \cap S ) 
  \leq M \mu(B) \xi_l \left( \frac{r}{R_B} \right)^{d-l} \quad \text{and} \quad \cH^d ((\fS)^{Cr} \cap B \cap S ) \leq M \mu(B) \sum_{l=0}^{d-1} \xi_l \left( \frac{r}{R_B} \right)^{d-l}  \: .
 \end{equation}
\end{enumerate}
\end{lemma}
\begin{proof}
\begin{enumerate}[$(i)$]
\item Let $0 \leq l \leq d-1$, $0 < \rho \leq R$ and $D \subset \R^n$ be a bounded Borel set (we will consider the unbounded case afterwards). Take a family $\cF$ of two by two disjoint balls with same radius $2 \rho$ that are centered at $\fS_l^\rho \cap D$ and satisfy $\displaystyle \fS_l^\rho \cap D \subset \cup_{B \in \cF} 2B$ (recall that $D$ is bounded and one can add disjoint balls as long as possible to obtain such a family $\cF$, similarly to the proof of Proposition~\ref{propPackingPartition}). From the family $\cF$ we can construct a family $\cG$ of two by two disjoint balls of radius $\rho$, centered at $\fS_l \cap D^\rho$ and such that $\displaystyle \fS_l^\rho \cap D \subset \cup_{B \in \cG} 5B$. Indeed, let us write
\[
 \cF = \left\lbrace B(x_j, 2\rho) \right\rbrace_{j \in J} \quad \text{and} \quad \cG = \left\lbrace B(z_j, \rho) \right\rbrace_{j \in J} \: ,
\]
where $z_j$ is chosen so that $z_j \in \fS_l$ and $|z_j - x_j| < \rho$ (possible since $x_j \in \fS_l^\rho \cap D$). The ball in $\cG$ are two by two disjoint: indeed, as the balls in $\cF$ are disjoint, then for $i \neq  j \in J$, $|x_i - x_j| \geq 4 \rho$ so that $|z_i - z_j| > 2 \rho$. Furthermore, $B(x_j, 4 \rho) \subset B(z_j, 5 \rho)$ and $z_j \in \fS_l \subset S$ so that by $d$--Ahlfors regularity of $\nu = \cH^d_{| S}$, $\cH^d (B(z_j, 5 \rho) \cap S) \leq \widetilde{C}_0 5^d \rho^d$ and then
\begin{align}
\label{eqPackingNeighbourood1}
 \cH^d (\fS_l^\rho \cap D \cap S) \leq \cH^d \left( \cup_{j \in J} B(x_j ,4 \rho) \cap S \right) \leq \cH^d \left( \cup_{j \in J} B(z_j , 5 \rho) \cap S \right) \leq 5^d \widetilde{C}_0 (\# J) \rho^d \: .
\end{align}
On the other hand, $x_j \in D$ so that $z_j \in D^\rho$ and thus $B(z_j, \rho) \subset D^{2\rho}$. By $l$--Ahlfors regularity of $\cH^l_{| \fS_l}$,
\begin{equation}
\label{eqPackingNeighbourood2}
 (\# J) C^{-1} \rho^l \leq \sum_{j \in J} \cH^l ( \fS_l \cap B(z_j, \rho)) = \cH^l \left( \fS_l \cap \cup_{j \in J} B(z_j, \rho) \right) \leq  \cH^l \left( \fS_l \cap D^{2\rho} \right)
\end{equation}
and combining both estimates \eqref{eqPackingNeighbourood1} and \eqref{eqPackingNeighbourood2} we infer $\displaystyle \cH^d (\fS_l^\rho \cap D \cap S) \leq 5^d \widetilde{C}_0 C \rho^{d-l} \cH^l \left( \fS_l \cap D^{2\rho} \right)$. Summing the previous estimates for $l$ from $0$ to $d-1$, we obtain
\begin{equation*}
 \nu \left( \fS^\rho \cap D \right) = \nu \left( \cup_{l} \fS_l^\rho \cap D \right) \leq \sum_{l=0}^{d-1} \nu \left( \fS_l^\rho \cap D \right)  \leq M  \sum_{l=0}^{d-1} \rho^{d-l} \: \cH^l (\fS_l \cap D^{2\rho}) \: .
\end{equation*}
Finally, in the case where $D$ is not assumed to be bounded, usual non-decreasing measure property leads: 
\[
 \nu \left( \fS_l^\rho \cap D \right) = \lim_{\substack{k \to + \infty\\ k \in \N}} \nu \left( \fS_l^\rho \cap (D \cap B(0,k)) \right) \leq M \rho^{d-l} \lim_{\substack{k \to + \infty\\ k \in \N}} \cH^l \left(\fS_l \cap (D \cap B(0,k))^{2\rho} \right) \leq M \rho^{d-l} \cH^l (\fS_l \cap D^{2\rho}) \: .
\]
The same proof leads to the similar estimates stated for $S_{sg,l}$ and $\Theta_{sg,l}$ using the $l$--Ahlfors regularity of $\cH^l_{| S_{sg,l}}$ and $\cH^l_{| \Theta_{sg,l}}$ instead of $\cH^l_{| \fS_l}$.
\item Let us prove that \eqref{hypH4} could be equivalently replaced with \eqref{hypH4prime}. First assume that $S$ and $\theta$ satisfy assumptions \hyperref[hypH1]{$(H_1)$} to \eqref{hypH7}, let $0 \leq l \leq d-1$, $x\in S_{sg,l}$ and $0< r \leq R_B \leq R$ with $R \leq \xdiam{S_{sg,l}}$. As $\cH^l_{| S_{sg,l}}$ is $l$--Ahlfors regular, we directly have
\[
 \cH^l ( S_{sg,l} \cap B(x,R_B)) \leq C R_B^l \: .
\]
Applying $(i)$ with $D = B(x, R_B)$ and $\rho = r$, we obtain
\begin{equation}
 \cH^{d}(S_{sg,l}^r \cap B(x, R_B)\cap S)\leq M r^{d-l}\cH^{l}(S_{sg,l}\cap B(x, R_B + 2r)) \leq M r^{d-l}\cH^{l}(S_{sg,l}\cap B(x, R_B)) \: ,
\end{equation}
where the last inequality follows from $R_B + 2r \leq 3R_B$ and the $l$--Ahlfors regularity of $\cH^{l}_{| S_{sg,l}}$.\\
Conversely assume that $S$ and $\theta$ satisfy assumptions \hyperref[hypH1]{$(H_1)$} to \hyperref[hypH1]{$(H_3)$}, \eqref{hypH4prime}, \eqref{hypH5} to \eqref{hypH7}. Let $0 \leq l \leq d-1$, $x \in S_{sg,l}$ and $0 < R_B \leq R$, note that $S_{sg,l}$ being closed, $\supp \cH^l_{| S_{sg,l}} = S_{sg,l}$. By \eqref{hypH4prime} we already have $\displaystyle \cH^l \left(S_{sg,l} \cap B(x,R_B) \right) \leq C R_B^l$.
Then, taking $r = R_B$ in \eqref{hypH4}, we have $S_{sg,l}^r \cap B(x,R_B) = B(x,R_B)$ so that
\begin{multline*}
 R_B^d \leq \widetilde{C_0}  \cH^d \left( B(x,R_B) \cap S \right) \leq \widetilde{C_0}  C R_B^{d-l} \cH^l \left( S_{sg,l} \cap B(x,R_B) \right) \\
 \quad \Rightarrow \quad ( C \widetilde{C_0} )^{-1} R_B^l \leq \cH^l \left( S_{sg,l} \cap B(x,R_B) \right) \: .
\end{multline*}
We infer that for all $x \in S_{sg,l}$ and $0 < R_B \leq R$,
$\displaystyle
M^{-1} R_B^l \leq \cH^l \left( S_{sg,l} \cap B(x,R_B) \right) \leq M R_B^l \: .
$
We then recall that $S$ and thus $S_{sg,l} \subset S$ are bounded (see Remark~\ref{rkHyp}) so that it is possible to cover $S_{sg,l}$ with a finite number of balls of radius less than $R$ centered at $S_{sg,l}$ leading to $\cH^l(S_{sg,l}) < \infty$. We can then conclude that $\cH^l_{| S_{sg,l}}$ is $l$--Ahlfors regular thanks to Remark~\ref{remkAhlfors}.\\
The proof for \eqref{hypH6}--\eqref{hypH6prime} is identical.
\item Let $B \subset \R^n$ be an open ball of radius $R_B \leq \frac{R}{3C}$ and let $0 < r \leq R_B$. If $(\fS_l)^{Cr} \cap B = \emptyset$, there is nothing to check, and otherwise, let $x \in \fS_l$ such that $B \subset B(x, 2R_B + C r) \subset B(x, 3CR_B)$ (since $C \geq 1$ by Definitions~\ref{dfnPwHolderS} and \ref{dfnPwHolderTheta}). Then by $(ii)$\eqref{hypH4prime} and \eqref{hypH6prime}, since $C r \leq C R_B \leq 3 C R_B \leq R$,
\begin{align*}
 \cH^d ((\fS_l)^{Cr} \cap B \cap S ) & \leq \cH^d ((\fS_l)^{Cr} \cap B(x, 3CR_B) \cap S )\\
 &\leq \cH^d ((S_{sg,l})^{Cr} \cap B(x, 3CR_B) \cap S ) + \cH^d ((\Theta_{sg,l})^{Cr} \cap B(x, 3CR_B) \cap S )  \\
 & \leq C (Cr)^{d-l} \left( \cH^l (S_{sg,l} \cap B(x, 3CR_B) )  +  \cH^l ( \Theta_{sg,l} \cap B(x, 3CR_B) ) \right) \\
 & \leq 2 C^{2+d} 3^l r^{d-l} R_B^l \\ 
 & \leq 2 C^{2+d} 3^l C_0 \mu(B) \left( \frac{r}{R_B} \right)^{d-l} \: .
\end{align*}
\end{enumerate}
\end{proof}

For the sake of clarity, let us give examples that illustrate Definitions~\ref{dfnPwHolderS} and~\ref{dfnPwHolderTheta}, focusing on assumptions \eqref{hypH4} to \eqref{hypH7}.
\begin{example}
Besides smooth or $\xC^{1,a}$ compact $d$--submanifolds, natural examples of uniformly piecewise $\xC^{1,a}$ sets are polytopes like the cube for which $\fS_0 = \{ \rm corners \}$ and $\fS_1 = \{ \rm edges \}$. Another interesting example is the so--called stadium obtained by gluing two half--circles with two segments. Because of the four gluing points, the stadium is not $\xC^2$ but only $\xC^{1,1}$, and more important, it is actually uniformly piecewise smooth (in the sense of Definition~\ref{dfnPwHolderS} where $\xC^{1,a}$ is replaced with smooth in \eqref{hypH5}, see Remark~\ref{remkUnifPwSmoothSet}) with $\fS = \fS_0 = \{\text{four gluing points} \}$.
The stadium example raises the question of examples of $\xC^{1,a}$--sets that are not uniformly piecewise smooth (loosely speaking, examples of sets that are not smooth or $\xC^k$ up to a gluing set that can be included in the singular set $\fS$). One can for instance consider primitives of the following family of Weierstrass functions: given $0 < s < 1$ and $t > 1$ such that $st \geq 1$, define for $x \in [0,1]$,
\[
 f_{s,t} (x) = \sum_{j=0}^\infty s^j \cos (t^{j} x) \: .
\]
Let $a = - \frac{\ln s}{\ln t} \in (0,1)$. The function $f_{s,t}$ is $a$--Hölder though nowhere differentiable (see \cite{Hardy}), whence not Lipschitz according to Rademacher theorem. Then considering a primitive $F_{s,t}$ of $f_{s,t}$, the graph of $F_{s,t}$ is a $\xC^{1,a}$ curve that is not uniformly piecewise $\xC^2$, note that it is $1$--Ahlfors since $F_{s,t}$ is Lipschitz in $[0,1]$ (see Example~\ref{AhlParam}).
\setcounter{subfigure}{0}
\begin{figure}[!htp]
\subcaptionbox{A stadium.}{\includegraphics[width=0.31\textwidth]{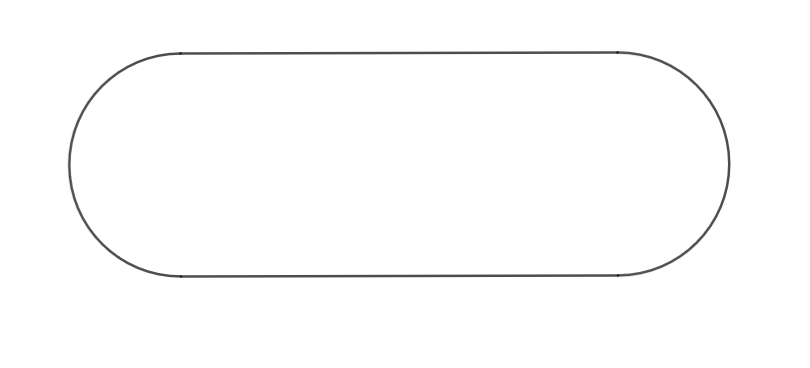}}
\subcaptionbox{$f_{s,t}$ for $s =0.3$ and $t = 4$, $a\approx 1.095 
$.}{\includegraphics[width=0.31\textwidth]{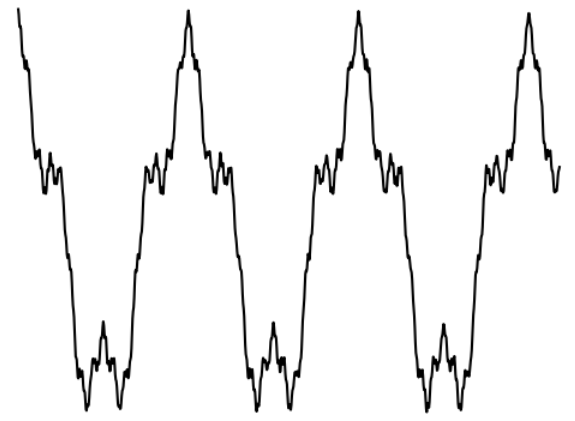}}
\subcaptionbox{$F_{s,t}$ for $s = 0.3$ and $t = 4$, $a\approx 1.095 $.}{\includegraphics[width=0.31\textwidth]{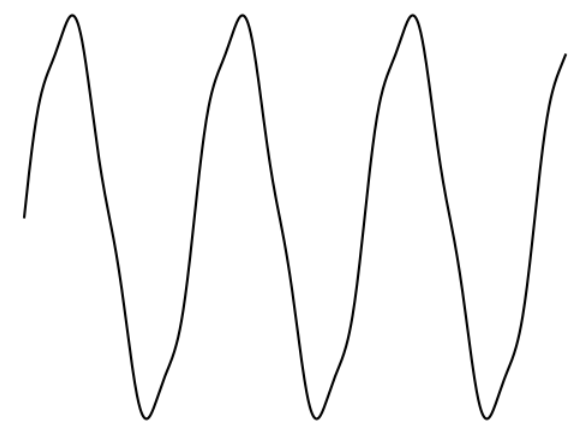}}
\end{figure}

Regarding the singular set $S_{sg}$, it can be constituted of several dimensional pieces as for simplices and polytopes. However, note that cusps or tangential contacts are not allowed by the local graph property \eqref{hypH5}. We believe that there is some room for improvement there since the tangent space does not vary badly near such singularities, such graph requirement might be relaxed. Note also that transverse crossings are allowed and, loosely speaking at least, the maximal angle at crossings is related to the constant $C \geq 1$ appearing in $S \setminus (S_{sg})^{Cr}$ in \eqref{hypH5}, we refer to Remark~\ref{remkContact} and Figure~\ref{figContact} below for additional comments regarding $S_{sg}$ and \eqref{hypH5}.
However, $S_{sg}$ itself can be fairly irregular: 
for instance considering a bounded open domain $D \subset \R^2$ such that $\partial D$ is $1$--Ahlfors regular and $S = \overline{D} \times \{ 0 \} \subset \R^3$, then $S$ is piecewise $\xC^{1,a}$/smooth regular with $S_{sg} = \partial D \times \{0\}$. A similar example regarding the singular set $\Theta_{sg}$ is given by 
$\theta = 1 + \one_{D}$ in $\R^2$ then $\theta$ is piecewise constant outside $\Theta_{sg} = \partial D$ and thus uniformly piecewise $\xC^{0,b}$/smooth with $\Theta_{sg} = \partial D$ (with $\partial D$ $1$--Ahlfors regular), in this example one can consider any compact set $S$ containing $D$.
\end{example}

\setcounter{subfigure}{0}
\begin{figure}[!htp]
\centering
\subcaptionbox{A transverse contact.}{\includegraphics[width=0.33\textwidth]{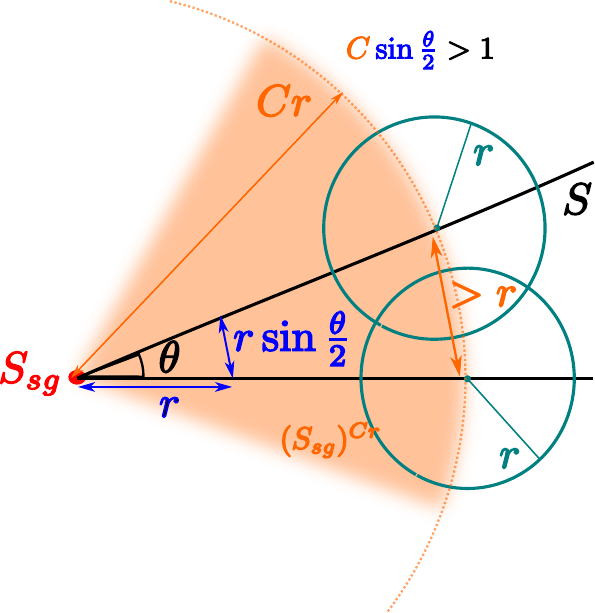}}
\qquad \qquad
\subcaptionbox{A tangential contact}{\includegraphics[width=0.33\textwidth]{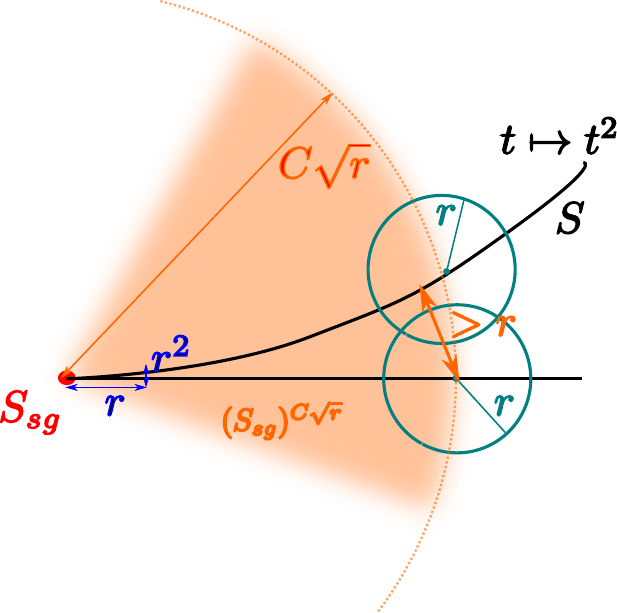}}
\caption{\label{figContact}}
\end{figure}

\begin{remark}[Possible extensions of such a framework] \label{remkContact}
It would be possible to further relax assumptions on the singular set, such changes would then affect the convergence rate of the $d$--dimensional measure of the "bad" set in which the pointwise estimates do not hold (i.e. the enlargements of $\fS$).
\begin{enumerate}[$\bullet$]
    \item First, we chose to allow a singular set $\fS$ constituted of a finite union of Ahlfors regular sets $\fS_l$ of integer dimension $l < d$, however, it would not be difficult to adapt Section~\ref{secUniformPWHolder} and~\ref{secSplit} to the case where $\fS$ is constituted of a finite union of Ahlfors regular sets $\fS_l$ of {\it real} dimension $l \leq d-1$, $l \in \{l_1, \ldots , l_J \} \subset [ 0, d-1]$. For instance, in Theorem~\ref{thmCvHolderSplit}, the statement would be unchanged up to summing over real numbers $l_j$ ($j = 1, \ldots, J$) instead of summing over integers $l = 0, \ldots, d-1$.
    However, the integrality of $l$ is not crucial but most natural examples already fit such a framework and we decided to keep the singular strata of integer dimensions.
    \item As already mentioned, \eqref{hypH5} in particular controls the maximal angle that is allowed around $S_{sg}$. In the particular case illustrated in Figure~\ref{figContact}$(a)$, we observe that the constant $C$ appearing in $S \setminus (S_{sg})^{Cr}$ \eqref{hypH5} and the contact angle $\theta$ at $S_{sg}$ then satisfy $C \sin \frac{\theta}{2} > 1$ thus preventing $\theta \in [0, \pi)$ to be smaller then some threshold. Similarly, \eqref{hypH5} excludes tangential contact for $S$. Indeed, considering the following simple example $S$ made of a tangential contact between an horizontal line and a parabola as illustrated in Figure~\ref{figContact}$(b)$, we observe that the issue stems from the small distance between both branches of $S$, which is no longer linear but quadratic with respect to the distance to the contact point $r$. It is then hopeless to require that $S$ is a graph in a ball of radius $r$ when only at distance $Cr$ from the contact point $S_{sg}$, nevertheless it can be required for point at distance larger than $C \sqrt{r}$ from $S_{sg}$ as illustrated in Figure~\ref{figContact} in the specific case at hand. More generally, asking the graph requirement to hold for $x \in S \setminus (S_{sg})^{Cr^{\gamma}}$ instead of $S \setminus (S_{sg})^{Cr}$ (for some $\gamma \in (0, 1]$) would allow tangential contact at various orders. Under such a modified assumption, Lemma~\ref{lemSingSet}$(iii)$ has then to be adapted to check that in a ball $B$ of radius $R_B \geq r$ the bad set $(S_{sg})^{Cr^{\gamma}} \cap S \cap B$ has $d$--dimensional measure controlled by 
    \[
    M \mu(B) \sum_{l=0}^{d-1} \xi_l (C r^\gamma, B) \left( \frac{ r^\gamma}{R_B} \right)^{d-l} \: .
    \]
    and the analogous change can be straightforwardly transferred to the conclusion of Theorem~\ref{thmCvHolderSplit}. We note that for small $\gamma$, the convergence rate can be degraded: one can compare $\delta_N^{\min(a,b)}$ with $\delta_N^{\gamma(d-l)}$ for $l$ such that $\fS_l \neq \emptyset$ in \eqref{eqThmVFfoldCV1SplitNuD}. Though such an adaptation of the framework is possible, we believe that the convergence rate of Theorem~\ref{thmCvHolderSplit} should not be generally degraded in presence of tangential contact. Indeed, on one hand, the tangent space estimation should not be degraded around such points since it has small variation. On the other hand, the pointwise estimation of density should give a multiple of the correct density, which is why the pointwise density estimation would fail. However, in our previous example, we emphasize that we would estimate a half--line with density $2$, which is indeed close to the set $S$ (near the contact point) when considering the bounded Lipschitz distance, despite the failure of the pointwise estimation of the density. It is not clear to us how to rigorously implement such ideas and we leave it for future thoughts.
\end{enumerate}
\end{remark}

\medskip
We henceforth assume that $S$ and $\theta$ satisfy \hyperref[hypH1]{$(H_1)$} to \eqref{hypH7} up to the end of the paper or equivalently that $\mu = \theta \cH^d_{| S} \in \cP$.

\subsection{Uniform convergence rates for the regularized density $\theta_\delta$ and covariance matrix $\Sigma_r$}
\label{secUniformPointwise}

We now prove that in the regularity class $\cP$ (see \eqref{eqClassP}) defined in the previous Section~\ref{secPwHolderDef} for $S$ and $\theta$, it is possible to obtain uniform convergence rates for the pointwise convergences concerning the density $\theta_\delta(x) \xrightarrow[\delta \to 0]{} \theta(x)$ and concerning the tangent space $\Sigma_r (x, \nu), \: \sigma_{r,\delta}(x) \xrightarrow[\delta,r \to 0]{} \Pi_{T_x S}$, provided that $x \in S$ is away from the singular set $S_{sg} \cup \Theta_{sg}$, as stated in Proposition~\ref{coroUniformCVHolder}. We first prove Lemma~\ref{lemmaConvspeedtgt} which allows one to deal with both the density and the tangent space approximations as they are both kernel (convolution) based. 

\begin{lemma}
\label{lemmaConvspeedtgt}
Let $0 < a \leq 1$. Let $S \subset \R^n$ be a closed set satisfying $\cH^d(S) < \infty$ and assume that $S$ is uniformly piecewise $\xC^{1,a}$. Using the notations of Definition~\ref{dfnPwHolderS}, let $0 < r \leq R = R_{S,sg}$ and $x \in S \setminus (S_{sg})^{Cr}$. Then,
for all $f \in \xC_c (\R^n,\R)$ bounded and Lipschitz with support in $B(0,1)$, 
\[
 \left| \int_S f\left(\frac{z-x}{r}\right) \: d\cH^d(z) - \int_{x + T_x S} f \left(\frac{z-x}{r}\right) \: d\cH^d(z) \right|\leq M (\| f \|_\infty + \xLip(f)) r^{d+a} \: ,
\]
where $M$ only depends on $d$ and $C_{S,sg}$.
\end{lemma}

\begin{proof}
Let $0 < r \leq R = R_{S,sg}$ and $x \in S \setminus (S_{sg})^{Cr}$. Let $q$ be an affine isometry sending $x + T_x M$ to $\R^d \times \{0\}$ and let $u : \cU \subset \R^d \rightarrow q(S \cap B(x,r))$ be $\xC^{1,a}$ as in \hyperref[hypH5]{$(H_5)$}.
Note that $\cU \subset B(0,r) \subset \R^d$ is necessarily bounded (for $y \in \cU$ and $z = (y, u(y))$, we have $|y|^2 \leq |z|^2 < r^2$).
Up to the affine isometry $q$, we suppose that $x=0$ and $T_x S$ is $\R^d\times\{0\}^{n-d}$ so that $u(0) = 0$ and $Du(0) = 0$. 
Let $\epsilon > 0$ such that $B(0,\epsilon) \subset \cU$, then for any $y \in B(0,\epsilon)$,
\begin{equation} \label{eqDuHolder}
 \| Du(y) \| = \| Du(y) - Du(0) \| \leq C |y|^a \: ,
\end{equation}
where $C = C_{S, sg}$, and
\begin{equation} \label{equHolder}
 | u(y) | = \left| u(0) + \int_0^1 Du(ty) y \: dt \right| = \left| \int_0^1 (Du(ty) - Du(0)) y \: dt \right| \leq C |y| \int_0^1 |ty|^a \: dt \leq C |y|^{a+1} \: .
\end{equation}
In particular, for $z = (y, u(y))$, $|z|^2 \leq |y|^2 \left( 1 + C^2 |y|^{2a} \right)$. 
Let $\epsilon_0 = \min \{ |y| \: : \: y \in \partial \cU \}$ and let $y_0 \in \partial \cU$ such that $|y_0| = \epsilon_0$. Let us check that 
\begin{equation} \label{eqInnerBallRadius}
 \epsilon_0 \geq r \frac{1}{\sqrt{1 + C^2 r^{2a}}} \: .
\end{equation}
Let $(t_k)_{k \in \N}$ satisfy $0 < t_k < 1$ and $t_k \rightarrow 1$, and consider $z_k = (t_k y_0, u(t_k y_0)) \in B(x,r) \cap S$: by compactness of $\overline{B(x,r) \cap S} \subset \overline{B(x,r)} \cap S$, we can assume that $z_k \rightarrow z = (y_0, v) \in \overline{B(x,r)} \cap S$. Note that $z \in \partial B(x,r)$ otherwise $(y_0, v) \in B(x,r) \cap S$ can be written as $z = (y_1, u(y_1))$ with $y_1 \in \cU$ incompatible with $y_1 = y_0 \in \partial \cU$.
By definition of $\epsilon_0$, we have
$B(0,\epsilon_0) \subset \cU$ and thus for all $k$,
\[
\underbrace{|t_k y_0|^2 + |u(t_k y_0)|^2}_{\rightarrow |z|^2 = r^2} \leq |t_k y_0|^2 (1 + C^2 |t_k y_0|^{2a}) \leq \epsilon_0^2 (1 + C^2 r^{2a}) \quad \Rightarrow \quad \eqref{eqInnerBallRadius} \: .
\]
Let $f \in \xC_c (\R^n,\R)$ be a Bounded Lipschitz function with support in $B(0,1)$. We recall that up to assuming that the affine isometry $q$ is the identity map, we have $x = 0$, $T_x S = \R^d \times \{ 0 \}^{n-d}$, $\cU \subset B(0,r)$ and we apply the area formula to the map $w = ({\rm id}, u) : \cU \rightarrow B(x,r) \cap S$ in the first integral below:
\begin{align}
 \left| \int_{z \in S} \right. & \left. f \left(\frac{z-x}{r}\right) \: d \cH^d(z) - \int_{z \in x+T_x S} f \left(\frac{z-x}{r}\right) \: d \cH^d(z) \right| \nonumber \\
 & = \left| \int_{y \in \cU} f\left(\frac{(y, u(y))}{r}\right) J_d w(y) \: dy - \int_{y \in B(0,r)} f \left(\frac{(y, 0)}{r}\right)  \: dy \right| \nonumber \\
 & \leq \underbrace{\int_{\cU} }_{ \int_{B(0,r)}} \left| f \left(\frac{(y, u(y))}{r}\right) J_d w (y) - f \left(\frac{(y, 0)}{r}\right) \right| \: dy + \int_{B(0,r) \setminus \cU} \left| f \left(\frac{(y, 0)}{r}\right) \right| \: dy \: , \label{eqAreaFormula1}
\end{align}
where $J_d w$ is the tangential Jacobian of $w = ({\rm id},u)$ and satisfies (see for instance Remark 2.72 in \cite{ambrosio2000fbv})
\[
 J_d w = \sqrt{1 + W_u} \quad \text{where} \quad W_u(y) =\sum_{\substack{B \text{ square minor} \\ \text{of } Du(y)}} \det(B)^2 \: .
\]
Note that for $1 \leq k \leq d$, $y \in B(0,r) \subset \R^d$ and $B$ a $k \times k$ minor of $Du(y)$ and reminding \eqref{eqDuHolder}, we have up to multiplicative constants only depending on $d$:
$
| \det B | \lesssim \| B \|^k \lesssim \| Du(y) \|^k$ and $\| Du(y) \|^k \leq C^k |y|^{ak} \leq C^d r^a
$ since $r^k \leq r$ and $C^k \leq C^d$ for $r \leq R < 1 \leq 1$, $C \geq 1$ and $1 \leq k \leq d$. Therefore, there exists a constant $c(d) \geq 1$ only depending of $d$ such that
\[
0 \leq W_u(y) \leq c(d) r^{2a} \: .
\]
Now, by concavity of $\sqrt{1 + \cdot}$, for any $t \geq 0$, $\sqrt{1 + t} \leq 1 + \frac{1}{2} t$ and consequently
\[
 0 \leq J_d w (y) - 1 \leq \frac{1}{2} W_u(y) \leq c(d) C^{2d} r^{2a}\quad \Rightarrow \quad | J_d w (y) - 1 | \leq M r^{2a} \: .
\]
Using $\left|(y, u(y)) - (y, 0) \right| = |u(y)| \leq C |y|^{a+1} \leq C r^{a+1}$ by \eqref{equHolder}, we can infer that
\begin{align}
 & \int_{B(0,r)}  \left| f \left(\frac{(y, u(y))}{r}\right) J_d w (y) - f \left(\frac{(y, 0)}{r}\right) \right| \: dy \nonumber \\
 \leq & \int_{B(0,r)} \left| f \left(\frac{(y, u(y))}{r}\right) \right| \left| J_d w(y) - 1 \right| \: dy + \int_{B(0,r)} \left| f \left(\frac{(y, u(y))}{r}\right) - f \left(\frac{(y, 0)}{r}\right) \right| \: dy \nonumber \\
 \leq & \| f \|_\infty M r^{2a} \omega_d r^d + \xLip(f) C r^a \omega_d r^d \nonumber \\
 \leq & (\| f \|_\infty + \xLip(f)) M  r^{d+a} \: . \label{eqAreaFormula2}
\end{align}
Eventually, we recall that $B(0,\epsilon_0) \subset \cU$ with $\epsilon_0$ satisfying \eqref{eqInnerBallRadius} and we use that for all $t \geq 0$, $(1 + t)^\frac{-d}{2} \geq 1 - \frac{d}{2} t$ (follows from convexity argument for instance) so that
\begin{align}
\int_{B(0,r) \setminus \cU} \left| f \left(\frac{(y, 0)}{r}\right) \right| \: dy & \leq \| f \|_\infty \left| B(0,r) \setminus B(0, \epsilon_0) \right| \leq \| f \|_\infty \omega_d r^d \left( 1 - (1 + C^2 r^{2a})^\frac{-d}{2} \right)  \nonumber \\
& \leq \| f \|_\infty \omega_d r^d \frac{d}{2} C^2 r^{2a} \: . \label{eqAreaFormula3}
\end{align}
We conclude the proof of Lemma~\ref{lemmaConvspeedtgt} thanks to \eqref{eqAreaFormula1}, \eqref{eqAreaFormula2} and \eqref{eqAreaFormula3}.
\end{proof}
\begin{proposition}
\label{coroUniformCVHolder}
Let $0 < a,b \leq 1$.
We assume that 
$S$ and $\theta$ satisfy assumptions \hyperref[hypH1]{$(H_1)$} to \eqref{hypH7} i.e $\mu = \theta \cH^d_{| S} \in \cP$. Using the notations of Definitions~\ref{dfnPwHolderS} and \ref{dfnPwHolderTheta}, let $0 < \delta , r \leq R = \min( R_{S,sg}, R_{\theta,sg}) < 1$, $C = \max (C_{S,sg}, C_{\theta, sg}) \geq 1$,  
then, with the notation $c = \min(a,b)$,
\begin{equation} \label{eqUniformCVThetaHolder}
\forall x \in S \setminus (\fS)^{C\delta} , \quad \left| \theta_\delta(x)-\theta(x) \right| \leq M (\delta^a + \delta^b) \leq M \delta^c
\end{equation}
and
\begin{equation} \label{eqUniformCVCovarianceHolder}
\left\lbrace
\begin{array}{rl}
 \forall x \in S \setminus (S_{sg})^{Cr} , & \quad \left\|\Sigma_r(x,\nu)-\Pi_{T_xS} \right\| \leq  M r^a
 \\
 \forall x \in S \setminus \fS^{Cr} , & \quad \left\| \Sigma_r(x, \mu) - \theta(x) \Pi_{T_x S} \right\| \leq  M ( r^a + r^b ) \leq M r^c
 \\
 \forall x \in S \setminus \fS^{C \max(\delta, r)} , & \quad \left\| \sigma_{r,\delta}(x) - \Pi_{T_x S} \right\| \leq  M (\delta^a + \delta^b + r^a + r^b ) \leq M \left(\delta^c + r^c \right) \: .
 \end{array}
 \right.
\end{equation}
\end{proposition}

\begin{proof}
We recall that $\fS = S_{sg} \cup \Theta_{sg}$ and let $0 < \delta, r \leq R$.
We start with the proof of \eqref{eqUniformCVThetaHolder}: from \eqref{eqCetaRn} and then dilation and translation, we have for $x \in S \setminus \fS^{C\delta}$,
\begin{equation*}
C_\eta = \int_{T_xS}\eta(|y|)d\mathcal{H}^d(y) = \frac{1}{\delta^d} \int_{x + T_x S} \eta \left(\frac{|y-x|}{\delta}\right) \:  d\cH^d(y)
\end{equation*}
so that 
\begin{align*}
    \theta_\delta(x)-\theta(x)=&
    \frac{1}{C_\eta \delta^d} \int_{B(x,\delta) \cap S} \eta \left(\frac{|y-x|}{\delta}\right) (\theta(y)-\theta(x)) \: d\cH^d(y)
     \\
     &+\theta(x) \frac{1}{C_\eta \delta^d} \left(\int_{S} \eta \left(\frac{|y-x|}{\delta}\right) \: d\cH^d(y)-\int_{x+T_x S} \eta \left(\frac{|y-x|}{\delta}\right) \:  d\cH^d(y)\right) \: .
\end{align*}
Applying Lemma~\ref{lemmaConvspeedtgt} with $f = \eta(|\cdot|)$ (that has same Lipschitz and infinity bounds as $\eta$) and thanks to \hyperref[hypH1]{$(H_1)$} and \eqref{hypH7}, we can infer
\begin{align*}
    \left| \theta_\delta(x)-\theta(x) \right| & \leq 
    \frac{C \delta^b}{C_\eta \delta^d} \| \eta \|_\infty \underbrace{\cH^d(B(x,\delta) \cap S)}_{\leq M \delta^d}
     +\theta_{max} \frac{1}{C_\eta \delta^d} M (\| \eta \|_\infty + \xLip(\eta)) \delta^{d+a} \\
     & \leq M (\delta^a + \delta^b) \: .
\end{align*}
We similarly prove \eqref{eqUniformCVCovarianceHolder}: 
by Definition~\ref{dfn:covMatrix} and Proposition~\ref{propCVSigmar}, for $x \in S \setminus (S_{sg})^{Cr}$,
\begin{equation*}
    \left\| \Sigma_r(x, \nu) - \Pi_{T_x S} \right\| = 
    \frac{1}{C_\phi r^d} \left\| \int_{S} \psi \left(\frac{y-x}{r}\right) \: d\cH^d(y)-\int_{x+T_x S} \psi \left(\frac{y-x}{r}\right) \:  d\cH^d(y)\right\| \: .
\end{equation*}
and we infer from Lemma~\ref{lemmaConvspeedtgt} with $f = \psi$ that
\begin{align}
\label{eqUniformCVCovarianceHolderNu}
    \left\| \Sigma_r(x, \nu) - \Pi_{T_x S} \right\|  
     & \leq \frac{M}{r^d} (\| \psi \|_\infty + \xLip(\psi)) r^{d+a} \leq M r^a \: .
\end{align}
Then, again thanks to \hyperref[hypH1]{$(H_1)$} and \eqref{hypH7} and \eqref{eqUniformCVCovarianceHolderNu}, we obtain if moreover $x \in S \setminus \fS^{Cr}$,
\begin{align}
\label{eqUniformCVCovarianceHolderMu}
    \left\| \Sigma_r(x, \mu) - \theta(x) \Pi_{T_x S} \right\| & \leq 
     \frac{1}{C_\phi r^d} \left\| \int_{S} \psi \left(\frac{y-x}{r}\right) (\theta(y) - \theta(x)) \: d\cH^d(y) \right\| + \theta(x) \left\| \Sigma_r(x, \nu) - \Pi_{T_x S} \right\| \nonumber \\
     & \leq \frac{C r^b}{C_\phi r^d} \| \phi \|_\infty \cH^d(B(x, r) \cap S) + \theta_{max} M r^a \nonumber \\
     & \leq M (r^b + r^a) \: ,
\end{align}
so that finally, from \eqref{eqUniformCVCovarianceHolder} and \eqref{eqUniformCVCovarianceHolderMu} we infer: for $x \in S \setminus \fS^{C\max(\delta,r)}$,
\begin{align*}
 \left\| \sigma_{r,\delta}(x) -  \Pi_{T_x S} \right\| & \leq \left\| \Phi \left( \theta_\delta(x) \right) \left( \Sigma_r (x, \mu) - \theta(x) \Pi_{T_x S} \right) \right\| 
 + \left\| \left( \Phi \left(\theta_\delta(x) \right) - \Phi(\theta(x)) \right) \theta(x) \Pi_{T_x S} \right\| \\
 & \leq \| \Phi \|_\infty M (r^b + r^a) + \theta_{max} \xLip(\Phi) M (\delta^a + \delta^b) \\
     & \leq M (\delta^a + \delta^b + r^a + r^b) \: .
\end{align*}
\end{proof}
\subsection{Uniform convergence rates of $\nu_\delta$ and $W_{r,\delta}$ in terms of Bounded Lipschitz distance}
\label{secUniformBLdist}
Still in the same regularity class $\cP$ (see \eqref{eqClassP}) for $\mu = \theta \cH^d_{| S}$ (i.e. $S$ and $\theta$ satisfy assumptions \hyperref[hypH1]{$(H_1)$} to \eqref{hypH7} as defined in Section~\ref{secPwHolderDef}), we can draw on the pointwise rates established in Proposition~\ref{coroUniformCVHolder} to obtain similar rates for $\nu_\delta$ (as stated in Proposition~\ref{cvnudel}) and $W_{r,\delta}$, $\widetilde{W}_{r,\delta}$ (as stated in Proposition~\ref{propCvHolderW}. The proof simply consists in controlling the measure of offsets of the singular set $S_{sg} \cup \Theta_{sg}$ on the one hand and uses the aforementioned pointwise rates on the other hand.

Let us start with uniform bounds for $\| \theta_\delta - \theta \|_{\xL^1_{loc}(\mu)}$ and $\beta_D (\nu_\delta, \nu)$. We use the notation $\| \cdot \|_{\xL^1(D, \mu)} = \int_D | \cdot | \: d \mu$ for the $\xL^1$--norm in the open set $D$ with respect to $\mu$.
\begin{proposition}
\label{cvnudel}
Let $0 < a,b \leq 1$.
We assume that $S$ and $\theta$ satisfy assumptions \hyperref[hypH1]{$(H_1)$} to \eqref{hypH7} i.e. $\mu = \theta \cH^d_{| S} \in \cP$. Let $0 < \delta \leq \frac{R}{C} < 1$ (with $R = \min( R_{S,sg}, R_{\theta,sg}) < 1$ and $C = \max (C_{S,sg}, C_{\theta, sg}) \geq 1$) and $D \subset \R^n$ be an open set, then
\begin{equation} \label{eqThetaDeltaThetaL1D}
\beta_D (\nu_\delta, \nu) \leq \| \theta_\delta - \theta \|_{\xL^1(D, \mu)}
\quad \text{and} \quad
 \| \theta_\delta - \theta \|_{\xL^1(D, \mu)} 
 \leq M (\delta^a + \delta^b) \mu (D) + M  \sum_{l=0}^{d-1} \delta^{d-l} \: \cH^l \left(\fS_l \cap D^{2C\delta}\right) \: .
\end{equation}
In the particular case where $D = B \subset \R^n$ is an open ball of radius $0< \delta \leq R_B \leq \frac{R}{3C}$.
Then, recalling that $\xi_l = \xi_l(C\delta, B)$ satisfies $\xi_l = 1$ if $B \cap (\fS_l)^{C\delta} \neq \emptyset$ and $\xi_l = 0$ otherwise,
\begin{equation} \label{eqThetaDeltaThetaL1}
 \beta_B(\nu_{\delta},\nu) \leq \int_B \left|\theta_\delta - \theta \right| \: d \mu \leq \left\lbrace \begin{array}{ll}
   M \left( \delta^a +\delta^b\right)\mu(B) & \text{if } B \cap \fS^{C\delta}  = \emptyset  \\
   M \left( \displaystyle  \delta^a +\delta^b +  \sum_{l=0}^{d-1} \xi_l \left( \frac{\delta}{R_B} \right)^{d-l} \right) \mu(B)  & \text{in general,}
   \end{array}
   \right.
\end{equation} 
and for $0 < r \leq R_B$ and $x \in S \setminus \fS^{C(\delta + r)}$, we have
$\displaystyle
 \left\| \Sigma_r (x, \nu_\delta) - \Sigma_r (x,\nu) \right\| 
 \leq M ( \delta^a +\delta^b)
$.
\end{proposition}
\begin{proof}
Let $0 < \delta \leq R$ and $D \subset \R^n$ be an open set. We first note that by definition of $\beta_D$,
\begin{equation*}
 \beta_D (\nu_\delta, \nu) \leq \left|\nu_{\delta} - \nu \right|(D) = \int_{D} \left| \Phi (\theta_\delta(x)) - \Phi(\theta(x))  \right| \: d\mu (x) \leq M \int_{D} \left| \theta_\delta - \theta  \right|  \: d\mu \: .
\end{equation*}
Furthermore, we can consider separately $\fS^{C\delta}$ and $S \setminus \fS^{C\delta}$.
On one hand, applying \eqref{eqUniformCVThetaHolder} in Proposition~\ref{coroUniformCVHolder}, we obtain
\begin{equation*} 
 \int_{D \setminus \fS^{C\delta}} \left| \theta_\delta - \theta  \right| \:  d \mu \leq M (\delta^a + \delta^b) \mu (D) \: ,
\end{equation*}
On the other hand, recalling \eqref{eqThetaDeltaUpperBound}: for $x \in  S$, $| \theta_\delta(x)| \leq M$ and using \eqref{eqThickD} with $C \delta \leq R$
we obtain
\begin{align*} 
 \int_{D \cap \fS^{C\delta}} \left| \theta_\delta - \theta  \right| \: \underbrace{  \: d \mu }_{\theta d \cH^d_{| S}} & \leq M
 \cH^d \left(\fS^{C\delta} \cap D \cap S\right) \leq M  \sum_{l=0}C^{d-l} \delta^{d-l} \: \cH^l \left(\fS_l \cap D^{2C\delta}\right) \: ,
\end{align*}
hence concluding the proof of \eqref{eqThetaDeltaThetaL1D}. The proof of \eqref{eqThetaDeltaThetaL1} is exactly the same, but applying \eqref{remkBallCenterSsing} instead of \eqref{eqThickD} in the last estimate above.
Finally, for $x \in S \setminus \fS^{C(\delta + r)}$, we have $B(x,r) \cap \fS^{C\delta} = \emptyset$ and
since $\| \psi_r \|_\infty \leq \| \phi \|_\infty$,
\begin{align*}
 \left\| \Sigma_r (x, \nu_\delta) - \Sigma_r (x,\nu) \right\| & = \frac{1}{C_\phi r^d} \left\| \int_{B(x,r)} \psi_r \left( y-x \right) \: d \nu_\delta - \int_{B(x,r)} \psi_r \left(y-x \right) \: d \nu \right\| \leq \| \phi \|_\infty \frac{1}{C_\phi r^d} \left|\nu_{\delta} - \nu \right|(B(x,r)) \\
 & \leq  \frac{M}{r^d} \int_{B(x,r)} | \theta_\delta - \theta | \: d \mu \leq M (\delta^a + \delta^b) \frac{\mu(B(x,r))}{r^d} \quad \text{thanks to \eqref{eqThetaDeltaThetaL1}} \\
 & \leq M C_0 \left(\delta^a + \delta^b\right)  \: ,
\end{align*}
hence concluding the proof of Proposition~\ref{cvnudel}.
\end{proof}

We then similarly give uniform bounds for  $\|\Sigma_{r} (\cdot, \nu_\delta) - \Pi_{T_\cdot S} \|_{\xL^1_{loc}(\mu)}$ and $\| \sigma_{r,\delta} -  \Pi_{T_\cdot S} \|_{\xL^1_{loc}(\mu)}$ as well as for $\beta_B (W_{r, \delta},  W_S)$ and $\beta_B (\widetilde{W}_{r, \delta},  W_S)$.
\begin{proposition} \label{propCvHolderW}
Let $0 < a,b \leq 1$.
We assume that $S$ and $\theta$ satisfy assumptions \hyperref[hypH1]{$(H_1)$} to \eqref{hypH7} i.e. $\mu = \theta \cH^d_{| S} \in \cP$. Let $B \subset \R^n$ be an open ball of radius $0< R_B \leq \frac{R}{6C}$ with $R = \min (R_{S,sg}, R_{\theta, sg}) < 1$, $C = \max (C_{S,sg}, C_{\theta, sg}) \geq 1$ and let $0 < \delta,  r < 1$ such that $\delta + r \leq R_B$.
Then,
\begin{equation} \label{eqCvHolderWrdelta}
\begin{array}{rcl}
 \displaystyle \int_B \left\| \Sigma_r(x,\nu_{\delta}) - \Pi_{T_x S} \right\| \: d \mu(x) , \quad \beta_B (\widetilde{W}_{r, \delta},  W_S)  & \leq & M \mu(B)  \left(\displaystyle \delta^{\min(a,b)} + r^a + \sum_{l=0}^{d-1} \xi_l \left(\frac{\delta + r}{R_B} \right)^{d-l} \right)\\ 
 \\
 \displaystyle \int_B \left\| \sigma_{r,\delta}(x) - \Pi_{T_x S} \right\| \: d \mu(x), \quad \beta_B (W_{r, \delta},  W_S)  & \leq & M \mu(B) \left(\displaystyle \delta^{\min(a,b)} + r^{\min(a,b)} + \sum_{l=0}^{d-1} \xi_l \left(\frac{\delta + r}{R_B} \right)^{d-l} \right)
 \end{array}
 \: ,
 \end{equation}
where $\xi_l = \xi_l(C(\delta + r), B)$ satisfies $\xi_l = 1$ if $B \cap \fS^{C(\delta + r)} \neq \emptyset$ and $\xi_l = 0$ otherwise.
We similarly have the following global estimates: let $0 < \delta,r \leq R$ with $\delta + r \leq \frac{R}{C}$ and let $D \subset \R^n$ be an open set, then
\begin{equation} \label{eqCvHolderWrdeltaD}
\begin{array}{l}
 \displaystyle \int_D \left\| \Sigma_r(x,\nu_{\delta}) - \Pi_{T_x S} \right\| \: d \mu(x) , \: \beta_D (\widetilde{W}_{r, \delta},  W_S)   \leq  M \mu(D^r) \left(\displaystyle \delta^{\min(a,b)} + r^a \right) + M \displaystyle \sum_{l=0}^{d-1} (\delta + r)^{d-l} \cH^l \left(\fS_l \cap D^{2C(\delta+r)}\right) \\ 
 \\
 \displaystyle \int_D \left\| \sigma_{r,\delta} - \Pi_{T_x S} \right\| \: d \mu(x), \: \beta_D (W_{r, \delta},  W_S)  \leq M \mu(D) \left(\displaystyle \delta^{\min(a,b)} + r^{\min(a,b)} \right) + \displaystyle M \sum_{l=0}^{d-1}\left(\delta + r \right)^{d-l} \cH^l \left(\fS_l \cap D^{2C(\delta + r)}\right)
 \end{array}
 \end{equation}
\end{proposition}
We recall that $W_{r,\delta}$ and $\widetilde{W}_{r,\delta}$ are defined in \eqref{eq:defWrdelta}. 

\begin{remark}[Choice of tangent space estimator]
We observe that when considering only the deterministic part, 
\begin{enumerate}[$(i)$]
 \item on one hand, first estimating the density to define $\nu_\delta = \frac{\theta}{\theta_\delta} \cH^d_{| S} \sim \cH^d_{| S}$ and then estimating the tangent plane thanks to $\Sigma_r (\cdot, \nu_\delta)$ (that is computing the covariance matrix with respect to the measure $\nu_\delta$ whose density has been corrected) allows to obtain $r^a$ in the r.h.s. of \eqref{eqCvHolderWrdelta} ;
 \item on the other hand, estimating the tangent space directly from $\mu = \theta \cH^d_{| S}$ (that is without correcting the density at this stage) through $\sigma_{r,\delta}$, we obtain instead $r^{\min(a,b)}$ which involves the regularity of the density through $b$. In the case where the density is less regular than the tangent plane (in the sense $b < a$), the control is thus less accurate. However, this only impacts the global bound \eqref{eqCvHolderWrdelta} if $r > \delta$.
\end{enumerate}
\end{remark}
\begin{proof}
The proof builds upon estimates \eqref{eqBoundWrdeltaW} and \eqref{eqBoundWrdeltaWTilde}. We recall that $R \in (0,1)$ and $C \geq 1$ as in Definitions~\ref{dfnPwHolderS} and \ref{dfnPwHolderTheta}.
First of all, as $r \leq R_B$, $\mu (B^r) \leq \mu(2B) \leq M \mu(B)$ by Ahlfors regularity of $\mu$.
Then, note that $\xi_l = \xi_l (C(\delta + r), B)$ here in \eqref{eqCvHolderWrdelta} and $\xi_l (C\delta, B^r) \leq \xi_l (C\delta + r, B) \leq \xi_l (C(\delta + r), B)$. Indeed, if $x \in B^r \cap (\fS_l)^{C\delta} \neq \emptyset$, then there exists $y \in B$ and $z \in \fS_l$ such that $|x - y| < r$ and $|x - z| <C\delta$ and then $|y - z| < r + C\delta$ so that $y \in B \cap \fS_l^{C\delta + r} \subset B \cap \fS_l^{C(\delta + r)} \neq \emptyset$.
We can therefore apply \eqref{eqThetaDeltaThetaL1} in Proposition~\ref{cvnudel} to the open ball $B^r$ of radius $R_B +r$ with $R_B \leq R_B +r \leq 2R_B \leq \frac{R}{3C}$ and we obtain
\begin{align} \label{eqUniformIntThetaDelta}
 \int_{B^r} |\theta_\delta - \theta | \: d \mu 
 & \leq M \left( \delta^{\min(a,b)} + \sum_{l=0}^{d-1} \xi_l(C\delta, B^r) \frac{\delta^{d-l}}{(R_B+r)^{d-l}} \right) \mu(B^r) \nonumber \\
 & \leq M \left( \delta^{\min(a,b)} + \sum_{l=0}^{d-1} \underbrace{\xi_l(C(\delta + r),B)}_{\xi_l} \frac{\delta^{d-l}}{R_B^{d-l}} \right) \mu(B) \: .
\end{align}
We then partition $B = B_1 \sqcup B_2$ into $B_1 = B \setminus \fS^{C(\delta + r)}$ and $B_2 = B \cap \fS^{C(\delta + r)}$, so that we can use uniform bounds in $B_1$ and control the measure of $B_2$ (similarly to what is done in the proof of \eqref{eqThetaDeltaThetaL1}). More precisely,
thanks to 
Proposition~\ref{coroUniformCVHolder}, using $\mu(B_1) \leq \mu(B)$, we first have
\begin{equation*}
 \int_{B_1} \| \Sigma_r (x, \nu) - \Pi_{T_x S} \| \: d\mu(x) \leq M r^a \mu (B) \: . 
\end{equation*}
Then, thanks to \eqref{eqInfBoundSigma}, for all $x \in S$, $\| \Sigma_r (x,\nu) - \Pi_{T_x S} \| \leq M + \| \Pi_{T_x S} \| \leq M+1$ so that using \eqref{remkBallCenterSsing} (with $\delta + r \leq R_B$), we obtain
\begin{align*}
 \int_{B_2} \| \Sigma_r (x, \nu) - \Pi_{T_x S} \| \: d\mu(x) & \leq M \mu (B_2) \leq M \theta_{max} \cH^d \left( B \cap S \cap \fS^{C(\delta + r)}  \right)
\leq M \sum_{l=0}^{d-1} \xi_l \frac{(\delta + r)^{d-l}}{R_B^{d-l}} \mu(B) \: .
\end{align*}
Consequently,
\begin{equation} \label{eqUniformIntSigmar}
 \int_{B} \| \Sigma_r (x, \nu) - \Pi_{T_x S} \| \: d\mu(x) \leq M \left( r^a + \sum_{l=0}^{d-1} \xi_l \frac{(\delta + r)^{d-l}}{R_B^{d-l}} \right) \mu(B) \: ,
\end{equation}
and in view of \eqref{eqBoundWrdeltaWTilde}, \eqref{eqUniformIntThetaDelta} and \eqref{eqUniformIntSigmar}, we can conclude that
\begin{align*}
\left.
\begin{array}{r}
 \beta_B \left(\widetilde{W}_{r,\delta}, W \right) \\
 \displaystyle \int_B \left\| \Sigma_r(x,\nu_{\delta}) - \Pi_{T_x S} \right\| \: d \mu(x) 
\end{array}
\right\rbrace
 & \leq M \int_{B^r} | \theta_\delta - \theta | \: d \mu + M \int_B  \| \Sigma_r(x,\nu) - \Pi_{T_x S} \| \: d \mu(x) \\
 & \leq M \left( \delta^{\min(a,b)} + r^a + \sum_{l=0}^{d-1} \xi_l \frac{(\delta + r)^{d-l}}{R_B^{d-l}} \right) \mu(B) \: .
\end{align*}
As $\max(\delta,r) \leq r + \delta$, we can use the same decomposition for $B$ and thanks to \eqref{eqUniformCVCovarianceHolder} in Proposition~\ref{coroUniformCVHolder}, we similarly have
\begin{equation*}
 \int_{B_1} \| \sigma_{r,\delta} (x) - \Pi_{T_x S} \| \: d\nu(x) \leq M (\delta^a + \delta^b + r^a + r^b) \mu (B) \: ,
\end{equation*}
while for $x \in  S$, $\| \sigma_{r,\delta}(x) \| \leq \| \Phi \|_\infty \| \Sigma_r (x, \mu) \| \leq M$ and 
\begin{align*}
 \int_{B_2} \| \sigma_{r,\delta} (x) - \Pi_{T_x S} \| \: d\mu(x) & \leq (M + 1) \mu (B_2) \leq M \sum_{l=0}^{d-1} \xi_l \frac{(\delta + r)^{d-l}}{R_B^{d-l}} \mu(B) \: ,
\end{align*}
hence leading, thanks to \eqref{eqBoundWrdeltaW} and \eqref{eqUniformIntThetaDelta}, to almost the same bound for $\displaystyle \int_B \left\| \sigma_{r,\delta}(x) - \Pi_{T_x S} \right\| \: d \mu(x) $ and $\beta_B \left(W_{r,\delta}, W_S \right) $ as for $\beta_B \left(\widetilde{W}_{r,\delta}, W_S \right)$ (up to replacing $r^a$ with $r^{\min(a,b)}$), which conclude the proof of \eqref{eqCvHolderWrdelta}.\\
The proof of \eqref{eqCvHolderWrdeltaD} is then a straightforward adaptation, replacing the use of the local estimates \eqref{eqThetaDeltaThetaL1} in Proposition~\ref{cvnudel} and \eqref{remkBallCenterSsing} with the global ones \eqref{eqThetaDeltaThetaL1D} in Proposition~\ref{cvnudel} and \eqref{eqThickD}.
\end{proof}
\subsection{Convergence rate of the varifold estimator in a piecewise Hölder regularity class}
\label{secCVvarifold}
In Section~\ref{secCVvarifold}, we can eventually address the inference of the varifold structure. Up to this point, we proposed convergent estimators $W_{r,\delta,N}$, $\widetilde{W}_{r,\delta,N}$ of $W_{r,\delta}$, $\widetilde{W}_{r,\delta}$ (see Proposition~\ref{prop:CV_WrdeltaN}) that can be coupled to the accuracy of the regularization established in Proposition~\ref{propCvHolderW} to give estimators of $W_S = \cH^d_{| S} \otimes \delta_{\Pi_{T_x S}}$ (see \eqref{eqThmVFfoldCV2}). However, such estimators $W_{r,\delta, N}$, $\widetilde{W}_{r,\delta,N}$ (with $\delta = \delta_N$, $r = r_N$) are Radon measures in $\R^n \times \xSym_+(n)$ that are not $d$--varifolds since $\sigma_{r,\delta,N}(x)$ or $\Sigma_r (\cdot, \nu_{\delta,N})$ are not orthogonal projectors in general. The last step to obtain an estimator $V_{r,\delta,N}$ (resp. $\widetilde{V}_{r,\delta,N}$) in the space of varifolds is to replace the symmetric matrix $\sigma_{r,\delta,N}(x)$ (resp. $\Sigma_r (\cdot, \nu_{\delta,N})$) with an orthogonal projection of rank $d$ that will be denoted by $\pi_{r,\delta,N}(x)$ (resp. $\Pi_{r} (\cdot, \nu_{\delta,N})$). 
We then consider $V_{r,\delta,N} = \nu_{\delta,N} \otimes \delta_{\pi_{r,\delta,N}}$ and $\widetilde{V}_{r,\delta,N} = \nu_{\delta,N} \otimes \delta_{\Pi_r(\cdot, \nu_{\delta,N})}$ that are Radon measures in $\R^n \times \xSym_+(n)$ with support in $\R^n \times {\rm P}_{d,n}$ and can be identified with $d$--varifolds through the bi-Lipschitz correspondence $\mathcal{I}$ of Proposition~\ref{propBiLipI}. We recall the notation ${\rm P}_{d,n} \subset \xSym_+(n)$ for the set of orthogonal projectors of rank $d$ in $\R^n$.

\noindent Beyond the formal difference between Radon measures in $\R^n \times \xSym_+(n)$ $W_{r,\delta,N}$, $\widetilde{W}_{r,\delta,N}$ and $d$--varifolds $V_{r,\delta,N}$, $\widetilde{V}_{r,\delta,N}$, it is important to have such a varifold estimator if we keep in mind the idea of estimating curvature quantities relying on the deterministic approximations proposed in \cite{BuetLeonardiMasnou} and \cite{BuetLeonardiMasnou2} in the future. We now explain how we pass from a symmetric matrix 
to a rank $d$ orthogonal projector via the eigen decomposition, note that this construction is consistent with what is often performed numerically, including in both aforementioned works \cite{BuetLeonardiMasnou} and \cite{BuetLeonardiMasnou2}.

Given $\Sigma \in \xSym_+(n)$, we associate a rank--$d$ orthogonal projector $\Pi \in {\rm P}_{d,n}$ as follows: let $\lambda_1 (\Sigma) \geq \lambda_2 (\Sigma) \geq \ldots \geq \lambda_n(\Sigma)$ be the ordered eigenvalues of $\Sigma$, associated with an orthonormal basis of eigenvectors $(u_1, \ldots, u_n)$, we then define $\Pi = \sum_{k = 1}^d u_k \otimes u_k$. 
Note that such a matrix $\Pi$ depends on the choice of the orthonormal basis $(u_1, \ldots, u_n)$ in the case where $\lambda_d (\Sigma) = \lambda_{d+1}(\Sigma)$. In what follows, we do not exclude such a possibility and we consider that $\Pi$ can be taken to be any of the admissible choice. Then, for any projector $\tilde{\Pi} \in {\rm P}_{d,n}$, we have
\begin{equation} \label{eqWeylStab}
 \|  \Sigma - \Pi  \| \leq \|  \Sigma - \tilde{\Pi}  \| \: .
\end{equation}
Indeed, as $\Sigma = \sum_{k=1}^n \lambda_k( \Sigma) u_k \otimes u_k$ and by definition of $\Pi$, the symmetric matrix $\Sigma - \Pi = \sum_{k=1}^n \left(\lambda_k( \Sigma) - \epsilon_k \right) u_k \otimes u_k$ has eigenvalues $\lambda_k(\Sigma) - \epsilon_k$ with the notation $\epsilon_k \in \{ 0, 1\}$, $\epsilon_k = 1$ if $k \leq d$ and $\epsilon_k = 0$ if $k \geq d+1$. We therefore obtain on one hand
\begin{equation} \label{eqFrobPiSigma}
 \| \Sigma - \Pi \| = \max_{k=1 \ldots n} \left|  \lambda_k (\Sigma) - \epsilon_k \right| \: .
\end{equation}
On the other hand, given $\tilde{\Pi} \in {\rm P}_{d,n}$, we observe that $\Pi$ and $\tilde{\Pi}$ have same eigenvalues $(\epsilon_k)_{k = 1 \ldots n}$ and therefore, applying Weyl inequality (more precisely, the fact that the ordered eigenvalues are $1$--Lipschitz among real symmetric or hermitian matrices endowed with the operator norm), we obtain for $k = 1, \ldots, n$,
\begin{align*}
 \left| \lambda_k \left( \Sigma \right) - \epsilon_k \right| & = | \lambda_k \left( \Sigma \right) - \lambda_k ( \tilde{\Pi} ) | \leq \|  \Sigma - \tilde{\Pi} \|
\end{align*}
and from \eqref{eqFrobPiSigma} we thus infer \eqref{eqWeylStab}.
In the particular case where $\Sigma$ is of the form $\Sigma_r (x, \lambda)$ (see Definition~\ref{dfn:covMatrix}) for $r > 0$, $x \in \R^n$ and $\lambda$ a finite Radon measure in $\R^n$, we use the notation $\Pi_r (x, \lambda)$ for the associated orthogonal projector of rank $d$ and we obtain with $\tilde{\Pi} = \Pi_{T_x S}$ in \eqref{eqWeylStab}:
\begin{equation} \label{eqFrobPiNSigmaN}
 \left\| \Sigma_r (x, \nu_{\delta,N}) - \Pi_r (x, \nu_{\delta,N}) \right\| \leq 
  \left\|  \Sigma_r (x, \nu_{\delta,N}) - \Pi_{T_x S} \right\|  \: ,
\end{equation}
while in the case where $\Sigma$ is of the form $\sigma_{r,\delta,N} (x)$, we use the notation $\pi_{r,\delta,N} (x)$ for the associated orthogonal projector of rank $d$ and we similarly obtain
\begin{equation} \label{eqFrobPiNsigmaN}
 \left\| \sigma_{r,\delta,N} (x) - \pi_{r,\delta,N} (x) \right\| \leq 
  \left\|  \sigma_{r,\delta,N} (x) - \Pi_{T_x S} \right\| \: ,
\end{equation}
Building upon such controls \eqref{eqFrobPiNSigmaN} and \eqref{eqFrobPiNsigmaN}, we introduce the $d$--varifolds 
\begin{equation} \label{eqVrdeltaN}
 V_{r,\delta,N} = \nu_{\delta,N} \otimes \delta_{\pi_{r,\delta,N}(x)}
 \quad \text{and} \quad
 \widetilde{V}_{r,\delta,N} = \nu_{\delta,N} \otimes \delta_{\Pi_r(x,\nu_{\delta,N})}
\end{equation}
and we investigate the convergence of $V_{r,\delta,N}$, (respectively $\widetilde{V}_{r,\delta,N}$) towards $W_S = \nu \otimes \delta_{\Pi_{T_x S}}$. Combining the mean convergence rate obtained in Proposition~\ref{prop:CV_WrdeltaN} for $\E [ \beta_B (W_{r, \delta, N}, \: W_{r,\delta}) ] $ (respectively  $\E [ \beta_B (\widetilde{W}_{r, \delta, N}, \: \widetilde{W}_{r,\delta}) ] $) with the uniform bound obtained for $ \beta_B (W_{r, \delta}, \: W_S)$ (respectively $ \beta_B (\widetilde{W}_{r, \delta}, \: W_S)$) in Proposition~\ref{propCvHolderW} we thus obtain a convergent varifold estimator of $W_S$ with explicit mean convergence rate in the regularity class $\cP$ at hand, as stated in Theorem~\ref{thmCvHolderV}. Note that we only state a local version in balls, for the sake of clarity, it is however possible to directly write a global estimate,  which will be done in Theorem~\ref{thmCvHolderSplit} with refined convergence rate. 

\begin{theorem} \label{thmCvHolderV}
Let $0 < a,b \leq 1$.
We assume that $S$ and $\theta$ satisfy assumptions \hyperref[hypH1]{$(H_1)$} to \eqref{hypH7} i.e. $\mu = \theta \cH^d_{| S} \in \cP$. Let $B \subset \R^n$ be an open ball of radius $0< R_B \leq \frac{R}{6C}$ with $R =\min (R_{S,sg}, R_{\theta, sg}) < 1$, $C =\max (C_{S,sg}, C_{\theta, sg}) \geq 1$ and let $0 < \delta,  r < 1$ with $\delta + r \leq R_B$.
Then, for $N \in \N^\ast$ large enough so that $N^{-\frac{1}{d}} \leq \min(\delta, r)$, we have
\begin{equation} \label{eqThmVFfoldCV1}
\left.
\begin{array}{r}
\E \left[ \beta_B (\widetilde{W}_{r,\delta,N}, \widetilde{V}_{r,\delta,N}) \right] \\
\E \left[ \beta_B (\widetilde{W}_{r,\delta,N}, W_S) \right] \\
\E \left[ \beta_B (\widetilde{V}_{r,\delta,N}, W_S) \right] 
\end{array} \right|
\leq  M \mu(B) \left[ \delta^{\min(a,b)} + r^a + \sum_{l=0}^{d-1} \xi_l \left(\frac{\delta+ r}{R_B} \right)^{d-l} \hspace{-3pt} + \frac{1}{\min(\delta, r)}
\left\lbrace
\begin{array}{lcl}
N^{-1/d}    & \text{if} & d > 2\\
N^{-1/2} \ln N  & \text{if} & d = 2\\
N^{-1/2}    & \text{if} & d = 1
\end{array}
\right. \right]
\end{equation}
and
\begin{equation} \label{eqThmVFfoldCV2}
\left.
\begin{array}{r}\E \left[ \beta_B (W_{r,\delta,N}, V_{r,\delta,N}) \right] \\
\E \left[ \beta_B (W_{r,\delta,N}, W_S) \right] \\
\E \left[ \beta_B (V_{r,\delta,N}, W_S) \right] 
\end{array} \right| \leq M \mu(B) \left[ \delta^{\min(a,b)} + r^{\min(a,b)} + \sum_{l=0}^{d-1} \xi_l \left(\frac{\delta+r}{R_B} \right)^{d-l} \hspace{-15pt} + \frac{1}{\min(\delta,r)}
\left\lbrace
\begin{array}{lcl}
N^{-1/d}    & \text{if} & d > 2\\
N^{-1/2} \ln N  & \text{if} & d = 2\\
N^{-1/2}    & \text{if} & d = 1
\end{array}
\right. \right]
\end{equation}
where $\xi_l = \xi_l(C(\delta + r), B)$ satisfies $\xi_l = 1$ if $B \cap \fS^{C(\delta + r)} \neq \emptyset$ and $\xi_l = 0$ otherwise.

\noindent If the ball $B$ is allowed to have larger radius $0 < R_B < 1$ (instead of $R_B \leq R$), the same statements hold for $\delta + r\leq \frac{R}{C}$ and replacing in both right hand sides in \eqref{eqThmVFfoldCV1} and \eqref{eqThmVFfoldCV2}:
\begin{equation} \label{eqThmVFfoldCV2BD}
 \mu(B) \sum_{l=0}^{d-1} \xi_l \left(\frac{\delta+r}{R_B} \right)^{d-l} \quad \text{with} \quad  \sum_{l=0}^{d-1} (\delta + r)^{d-l} \cH^l (\fS_l \cap B^{2C(\delta+ r)}) \: .
\end{equation}
\end{theorem}
\begin{proof}
We recall that $M$ is a generic constant (see Remark~\ref{remkCstMSecHolder}). We start with proving the theorem concerning $\widetilde{W}_{r,\delta,N}$ and $\widetilde{V}_{r,\delta,N}$.
Let $f \in \xC_c(\R^n \times \xSym_+(n))$ with $\supp f \subset B \times \xSym_+(n)$, $ \|f\|_\infty \leq 1$ and $\xLip(f)\leq 1$. Note that $\supp f \subset B \times \xSym_+(n)$ so that $\| f \|_\infty \leq R_B$. Then, thanks to \eqref{eqFrobPiNSigmaN}, we have for $x \in S$,
\begin{align*}
 \left| f \left(x,\Pi_r(x,\nu_{\delta,N}) \right) - f \left(x,\Sigma_r(x,\nu_{\delta,N}) \right) \right| & \leq \left\| \Pi_r(x,\nu_{\delta,N}) - \Sigma_r(x,\nu_{\delta,N}) \right\|  \leq  \left\| \Pi_{T_x S} - \Sigma_r(x,\nu_{\delta,N}) \right\| \\
 & \leq \left\| \Pi_{T_x S} - \Sigma_r(x,\nu_{\delta}) \right\| + \left\| \Sigma_r(x,\nu_{\delta}) - \Sigma_r(x,\nu_{\delta,N}) \right\|
\end{align*}
and consequently, as  $\| \Phi \|_\infty \leq M$, we obtain
\begin{align} \label{eqThmVFfoldCV3}
\left| \int f \: d \widetilde{V}_{r,\delta,N} - \int f \: d \widetilde{W}_{r,\delta,N} \right| & = \left| \int \left[ f \left(x,\Pi_r(x,\nu_{\delta,N}) \right) - f \left(x,\Sigma_r(x,\nu_{\delta,N}) \right) \right] \Phi \left( \theta_{\delta,N} (x) \right) \: d \mu_{N} \right| \nonumber \\
& \leq M \int_{B} g \: d \mu_N + M \int_B \|  \Pi_{T_x S} - \Sigma_r(x,\nu_{\delta}) \| \: d\mu_N
\end{align}
with $\displaystyle g(x) = \min \left(2 R_B,  \| \Sigma_r(x,\nu_{\delta}) - \Sigma_r(x,\nu_{\delta,N}) \| \right)$.
On one hand, we know from Proposition~\ref{propCvHolderW} that
\begin{align} \label{eqThmVFfoldCV4}
\mathbb{E} \left[ \int_B \| \Pi_{T_x S} - \Sigma_r(x,\nu_{\delta})  \| \: d\mu_N \right] & = \int_B \|  \Pi_{T_x S}  - \Sigma_r(x,\nu_{\delta}) \| \: d\mu  \leq M  \left( \delta^{\min(a,b)} + r^a + \sum_{l=0}^{d-1} \xi_l \left(\frac{r + \delta}{R_B} \right)^{d-l} \right)\mu(B) \: .
\end{align}
On the other hand, let us check that $g$ satisfies the assumption \eqref{eqHypg} of Proposition~\ref{propBetaLoc} with $\displaystyle \kappa_1 = 0$, $\displaystyle \kappa_2 = M$, $\displaystyle \kappa_0 = \kappa_0(r) = \frac{M}{r}$ and $\kappa = R_B$. 
Indeed, for $x, y \in \R^n$, using that $\| \Phi \|_\infty \leq M$, $\xLip(\psi_r) \leq \frac{M}{r}$ (see \eqref{eq:lipPsiij}) and the Ahlfors regularity of $\mu$, we have
\begin{align*} 
    \left\| \Sigma_r(x, \nu_{\delta})-\Sigma_r(y,\nu_{\delta}) \right\| & \leq \frac{1}{C_\phi r^d}
\int_{z \in \R^n} \left\| \psi_r(z-x)  - \psi_r (z-y) \right\| \left|\Phi\left( \theta_{\delta}(z)\right) \right| \: d \mu(z) \nonumber \\
& \leq \frac{ M }{r^{d+1}} \mu \left(B(x,r) \cup B(y,r) \right) |x-y|  \leq \frac{M}{r} |x-y| \: , 
\end{align*}
and similarly, see also \eqref{eq:lipSigmaTilde},
\begin{equation*}
  \left\| \Sigma_r(x, \nu_{\delta,N})-\Sigma_r(y,\nu_{\delta,N}) \right\| 
\leq \frac{M}{r^{d+1}} \mu_N \left(B(x,r) \cup B(y,r) \right) |x-y| = \frac{M}{r} \Delta_{r,N} (x,y)|x -y | \: ,
\end{equation*}
so that by triangular inequality, and noting that $t \in \R_+ \mapsto \min (2R_B , t)$ is $1$--Lipschitz, we obtain
\begin{align*}
 \left| g(x) - g(y) \right| & \leq \| \Sigma_r(x,\nu_{\delta,N}) - \Sigma_r(y,\nu_{\delta,N})  \| + \| \Sigma_r(x,\nu_{\delta}) - \Sigma_r(y,\nu_{\delta})  \| \\
 & \leq \frac{M}{r} \Delta_{r,N} (x,y) |x-y| + \frac{M}{r} |x-y| \:.
\end{align*}
Moreover $\| g \|_\infty  \leq 2 R_B$ so that $\| g \|_\infty \leq \kappa\kappa_0$ with $\kappa \leq R_B$.
We can thus apply Proposition~\ref{propBetaLoc}: 
\begin{equation*}
 \E \left[ \left| \int_B g \: d \mu_N - \int_B g \: d \mu \right| \right] \leq \frac{M}{r} \mu(B) \times \left\lbrace
\begin{array}{lcl}
N^{-1/d}    & \text{if} & d > 2\\
N^{-1/2} \ln N  & \text{if} & d = 2\\
N^{-1/2}    & \text{if} & d = 1
\end{array}
\right. 
\end{equation*}
and we recall that thanks to \eqref{Sigmaestim} in Proposition~\ref{propCVSigmaN},
\begin{equation*}
 \E \left[ \left| \int_B g \: d \mu \right| \right] \leq \int_B \E \left[ \| \Sigma_r(x,\nu_{\delta,N}) - \Sigma_r(x,\nu_{\delta}) \| \right] \: d \mu(x) \leq \frac{M}{\delta} \mu(B) \times \left\lbrace
\begin{array}{lcl}
N^{-1/d}    & \text{if} & d > 2\\
N^{-1/2} \ln N  & \text{if} & d = 2\\
N^{-1/2}    & \text{if} & d = 1
\end{array}
\right. 
\end{equation*}
so that we infer
\begin{align} \label{eqThmVFfoldCV5}
\E \left[ \left| \int_B g \: d \mu_N \right| \right] & \leq \E \left[ \left| \int_B g \: d \mu_N - \int_B g \: d \mu \right| \right] + \E \left[ \left| \int_B g \: d \mu \right| \right] \nonumber \\
& \leq M \left(  \frac{1}{\delta} + \frac{1}{r} \right) \mu(B) \times \left\lbrace
\begin{array}{lcl}
N^{-1/d}    & \text{if} & d > 2\\
N^{-1/2} \ln N  & \text{if} & d = 2\\
N^{-1/2}    & \text{if} & d = 1
\end{array}
\right.  .
\end{align}
We can conclude the proof of \eqref{eqThmVFfoldCV1} (first line) thanks to \eqref{eqThmVFfoldCV3}, \eqref{eqThmVFfoldCV4} and \eqref{eqThmVFfoldCV5}. We then obtain \eqref{eqThmVFfoldCV1} as a direct consequence of 
Proposition~\ref{prop:CV_WrdeltaN}$(ii)$ and Proposition~\ref{propCvHolderW}.

Let us check that the same strategy allows to prove Theorem~\ref{thmCvHolderV} for $W_{r,\delta,N}$ and $W_{r,\delta,N}$. Similarly to \eqref{eqThmVFfoldCV3}, we have
\begin{align} \label{eqThmVFfoldCV3Bis}
\left| \int f \: d V_{r,\delta,N} - \int f \: d W_{r,\delta,N} \right| 
\leq M \int_{B} g \: d \mu_N + M \int_B \| \sigma_{r,\delta}(x) - \Pi_{T_x S} \| \: d\mu_N
\end{align}
with $\displaystyle g(x) = \min \left( 2 R_B ,  \| \sigma_{r,\delta,N}(x) - \sigma_{r,\delta}(x)\| \right)$. 
Again from Proposition~\ref{propCvHolderW}, we know that
\begin{align} \label{eqThmVFfoldCV4Bis}
\mathbb{E} \left[ \int_B \| \sigma_{r,\delta}(x) - \Pi_{T_x S} \| \: d\mu_N \right] & = \int_B \| \sigma_{r,\delta}(x) - \Pi_{T_x S} \| \: d\mu  \leq M  \left[ \delta^{\min(a,b)} + r^{\min(a,b)} + \sum_{l=0}^{d-1} \xi_l \left(\frac{r + \delta}{R_B} \right)^{d-l} \right] \mu(B) \: .
\end{align}
We similarly check that $g$ satisfies the assumption \eqref{eqHypg} of Proposition~\ref{propBetaLoc} with $\displaystyle \kappa_1 = \kappa_2 = \kappa_3 = M$, $\displaystyle \kappa_0 = \kappa_0(\delta, r) = \frac{M}{\delta} + \frac{M}{r}$ and $\kappa = R_B$. Indeed, thanks to \eqref{eq:lipPsiij}, \eqref{eqInfBoundSigma}
for $x$, $y \in \R^n$,
we have
\begin{align*} 
\left\| \sigma_{r,\delta}(x) \right. & \left. -\sigma_{r,\delta}(y) \right\|  = \left\| \Phi(\theta_{\delta}(x)) \Sigma_r(x,\mu) - \Phi(\theta_{\delta}(y)) \Sigma_r(y,\mu) \right\| \nonumber \\
& \leq \| \Phi \|_\infty \| \Sigma_r (x , \mu) - \Sigma_r (y , \mu) \| + \| \Sigma_r (y , \mu ) \| \xLip(\Phi) | \theta_\delta(x) - \theta_\delta(y) | \\
& \leq \frac{M}{C_\phi r^d}
\int_{z \in \R^n} \left\| \psi_r(z-x)  - \psi_r (z-y) \right\|  \: d \mu(z) + \frac{M}{C_\eta \delta^d}
\int_{z \in \R^n} \left| \eta_\delta(z-x)  - \eta_\delta (z-y) \right|  \: d \mu(z) \\
& \leq \frac{M}{r^{d+1}}  \mu \left(B(x,r) \cup B(y,r) \right) |x-y| + \frac{M}{\delta^{d+1}}  \mu \left(B(x,\delta) \cup B(y,\delta) \right) |x-y| \nonumber  \\
& \leq M \left( \frac{1}{\delta} + \frac{1}{r}\right) |x-y| \text{ by Ahlfors regularity} \: .
\end{align*}
And similarly, see also \eqref{eq:lipSigmaBis},
\begin{equation*}
  \left\| \sigma_{r,\delta,N}(x)-\sigma_{r,\delta,N}(y) \right\| 
\leq M \left( \frac{1}{r} \Delta_{r,N}(x,y) + \frac{1}{\delta} \Delta_{\delta,N}(x,y) \Delta_{r,N}(x,y)  \right) |x-y| \: \: ,
\end{equation*}
so that by triangular inequality,
\begin{align*}
 \left| g(x) - g(y) \right| 
 \leq  M \left(\frac{1}{\delta} + \frac{1}{r} + \frac{1}{\delta}  \Delta_{\delta,N}(x,y) + \frac{1}{r} \Delta_{r,N} (x,y)  + \frac{1}{\delta} \Delta_{\delta,N}(x,y) \Delta_{r,N}(x,y) \right) |x-y|\:.
\end{align*}
We can thus apply Proposition~\ref{propBetaLoc} in order to control $\E \left[ \left| \int_B g \: d \mu_N - \int_B g \: d \mu \right| \right]$ and we recall that thanks to \eqref{eqConcentrationSigmaN0} in Proposition~\ref{propCVSigmaN},
\begin{equation*}
 \E \left[ \left| \int_B g \: d \mu \right| \right] \leq \int_B \E \left[ \| \Sigma_r(x,\nu_{\delta,N}) - \Sigma_r(x,\nu_{\delta}) \| \right] \: d \mu(x) \leq M \left( \frac{1}{\sqrt{N \delta^d}} + \frac{1}{\sqrt{N r^d}} \right) \mu(B) 
\end{equation*}
so that we infer
\begin{align} \label{eqThmVFfoldCV5Bis}
\E \left[ \left| \int_B g \: d \mu_N \right| \right] & \leq \E \left[ \left| \int_B g \: d \mu_N - \int_B g \: d \mu \right| \right] + \E \left[ \left| \int_B g \: d \mu \right| \right] \nonumber \\
& \leq M \left(  \frac{1}{\delta} + \frac{1}{r} \right) \mu(B) \times \left\lbrace
\begin{array}{lcl}
N^{-1/d}    & \text{if} & d > 2\\
N^{-1/2} \ln N  & \text{if} & d = 2\\
N^{-1/2}    & \text{if} & d = 1
\end{array}
\right. + M \left( \frac{1}{\sqrt{N \delta^d}} + \frac{1}{\sqrt{N r^d}} \right) \mu(B) \nonumber \\
& \leq   \frac{M}{\min(\delta,r)} \mu(B) \times \left\lbrace
\begin{array}{lcl}
N^{-1/d}    & \text{if} & d > 2\\
N^{-1/2} \ln N  & \text{if} & d = 2\\
N^{-1/2}    & \text{if} & d = 1
\end{array}
\right. \: .
\end{align}
We can conclude the proof of \eqref{eqThmVFfoldCV2} (first line) thanks to \eqref{eqThmVFfoldCV3Bis}, \eqref{eqThmVFfoldCV4Bis} and \eqref{eqThmVFfoldCV5Bis}. We then obtain \eqref{eqThmVFfoldCV2} as a direct consequence of 
Proposition~\ref{prop:CV_WrdeltaN} and Proposition~\ref{propCvHolderW}.

The proof of \eqref{eqThmVFfoldCV2BD} is exactly the same, replacing the local estimate \eqref{eqCvHolderWrdelta} with the global one \eqref{eqCvHolderWrdeltaD} in each application of Proposition~\ref{propCvHolderW} (also noting that $B^r \subset 2B$ and thus $\mu(B^r) \leq M \mu(B)$) when establishing \eqref{eqThmVFfoldCV4} and \eqref{eqThmVFfoldCV4Bis}.
\end{proof}
\section{Inferring the varifold structure in a piecewise Hölder regularity class}
\label{secSplit}
In Section~\ref{secCVvarifold} and more precisely in Theorem~\ref{thmCvHolderV}, we quantify the speed rate convergence of the deterministic term, that is the Bounded Lipschitz distance $\displaystyle \E \left[ \beta_B ( W_{r,\delta} , W_S )\right]$ thanks to the piecewise Hölder regularity assumptions \hyperref[hypH1]{$(H_1)$} to \eqref{hypH7} i.e. for $\mu = \theta \cH^d_{| S} \in \cP$.  
However, we did not take full advantage of such regularity: we directly apply Proposition~\ref{prop:CV_WrdeltaN} that gives the following bound, only relies on \hyperref[hypH1]{$(H_1)$} and \hyperref[hypH1]{$(H_2)$}
\[
 \E \left[ \beta_B ( W_{r,\delta,N} , W_{r,\delta}) \right] \leq \frac{M}{\min(\delta,r)} \mu(B) \times \left\lbrace
\begin{array}{lcl}
N^{-1/d}    & \text{if} & d > 2\\
N^{-1/2} \ln N  & \text{if} & d = 2\\
N^{-1/2}    & \text{if} & d = 1
\end{array}
\right.  \: .
\]
The purpose of the current section is precisely to improve the control of $\E \left[ \beta_B ( W_{r,\delta,N} , W_{r,\delta}) \right]$ in the piecewise Hölder regularity class $\cP$ defined in Section~\ref{secPwHolderDef}. Note that we will rely on the fact that when considering $\sigma_{r,\delta,N}$ the concentration result \eqref{eqConcentrationSigmaN0} gives a better rate than \eqref{Sigmaestim} obtained for $\Sigma_{r}(\cdot,\nu_\delta)$: see Remark~\ref{remkFasterRatesigmaVsSigma}. We therefore carry on our study considering only the tangent space estimator $\sigma_{r,\delta,N}$.

To this end, we need to estimate independently the measure $\mu$, the density $\theta$ and the tangent plane $T_x S$: we can divide our sample into $4$ independent parts, $(X_j)_{4 j + k}$ for $k = 0, 1, 2, 3$, for instance. For simplicity, we assume that the sample has size $4 N$ and we use the notations ${\bf X} = (X_1, \ldots, X_N)$, ${\bf Y} = (Y_1, \ldots, Y_N)$, 
${\tilde{\bf Y}} = (\tilde{Y}_1, \ldots, \tilde{Y}_N)$
and ${\bf Z} = (Z_1, \ldots, Z_N)$ for the four independent subsamples. We consequently have four associated empirical measures $\mu_N^{\bf X} = \frac{1}{N} \sum_{i=1}^N \delta_{X_i}$ and similarly $\mu_N^{\bf Y}$, $\mu_N^{\tilde{\bf Y}}$ and $\mu_N^{\bf Z}$. We can then independently estimate the measure $\mu$, the density $\theta$ and the tangent plane $T_x S$ by considering for $\delta, r > 0$:
\begin{equation*} 
\mu_N^{\bf X} \:, \quad \theta_{\delta, N}^{\bf Y} = \Theta_\delta  (\cdot, \mu_N^{\bf Y}) = \frac{\mu_N^{\bf Y} \ast \eta_\delta}{C_\eta \delta^d} \quad \text{and} \quad \sigma_{r,\delta,N} = \sigma_{r,\delta,N}^{\tilde{\bf Y}, {\bf Z}} = \Phi \left( \theta_{ \delta, N}^{\tilde{\bf Y}} \right) \Sigma_r (\cdot, \mu_N^{\bf Z})  \: .
\end{equation*}
In Lemma~\ref{lemMeanLipEstimators}, we evidence that in the piecewise Hölder regularity class $\cP$ at hand, $\theta_{\delta,N}$ and $\sigma_{r,\delta,N}$ admit mean Lipschitz estimates \eqref{eqMeanLipThetaN} and \eqref{eqMeanLipSigmaN} finer than the generic rough estimates previously used: for instance concerning $\theta_{\delta,N}$, we had in Lemma~\ref{lemPropertiesPhiThetaN} that for $x$, $y \in S$,
$\displaystyle
 \left| \theta_{\delta, N}(x) - \theta_{\delta, N}(y) \right| \leq  \frac{M}{\delta} \Delta_{\delta,N}(x,y) |x-y|$ and $\displaystyle \E \left[ \Delta_{\delta,N}(x,y) \right] \leq M$
so that we were using the mean Lipschitz estimate $\displaystyle \E \left[ \left| \theta_{\delta, N}(x) - \theta_{\delta, N}(y) \right| \right] \leq \frac{M}{\delta} |x-y|$ in the proof of Theorem~\ref{thmBetaLocNuN} and Proposition~\ref{prop:CV_WrdeltaN}. Such mean Lipschitz constants of order $\delta^{-1}$ for $\theta_{\delta,N}$ and $\min(\delta,r)^{-1}$ for $\sigma_{r,\delta,N}$ are responsible for the same factor appearing in the rates established in those results (Theorem~\ref{thmBetaLocNuN} and Proposition~\ref{prop:CV_WrdeltaN}). Therefore, a careful adaptation of the proof of Proposition~\ref{propBetaLoc} allows to leverage the aforementioned improved Lipschitz estimates \eqref{eqMeanLipThetaN} and \eqref{eqMeanLipSigmaN} in the piecewise Hölder regularity class $\cP$ to infer Theorem~\ref{thmCvHolderSplit}: we eventually obtain a varifold estimator $\widehat{V}_N = \widehat{V}_{\delta_N,N}$ (defined in \eqref{eqDfnVN}) with a mean convergence rate of order 
\[
 \E \left[ \beta ( \widehat{V}_N , W_S ) \right] \lesssim  N^{\displaystyle -\frac{c}{d + 2 c}} \quad \text{with} \quad c = \min(a,b) \quad \text{choosing  } \delta_N = N^{\displaystyle -\frac{1}{d + 2c}} \: .
\]

\begin{remark}
Note that $S$ being compact, it is sufficient to state Theorem~\ref{thmCvHolderSplit} in bounded open sets $D$. Indeed, we know that $S \subset B(0,k_0)$ for some $k_0 \in \N$ and then it is not difficult to check that $\beta (\cdot, \cdot) \leq \left( 1 + \frac{1}{k} \right) \beta_{B(0,k_0+ k)}$ for any $k \in \N^\ast$, thanks to a $\frac{1}{k}$--Lipschitz radial cutoff equal to $1$ in $B(0,k_0)$ and $0$ outside $B(0,k_0+k)$.
\end{remark}

\begin{remark}[Constant $M$] \label{remkCstMSecSplitt}
 All along the current section, $M$ stands for a generic constant that may vary from one statement to another one, and only depends on $d, n, \widetilde{C_0}, \eta, \phi, \theta_{min/max}, C = \max(C_{\theta, sg}, C_{S,sg}), R = \min(R_{\theta, sg}, R_{S,sg})$ and $\tau$.
\end{remark}
\begin{lemma} \label{lemMeanLipEstimators}
Let $d \in \N^\ast$
and $0 < a,b \leq 1$.
We assume that 
$S$ and $\theta$ satisfy assumptions \hyperref[hypH1]{$(H_1)$} to \eqref{hypH7} i.e. $\mu = \theta \cH^d_{| S} \in \cP$. Using the notations of Definitions~\ref{dfnPwHolderS} and \ref{dfnPwHolderTheta}, let $0 < \delta, r \leq R = \min( R_{S,sg}, R_{\theta,sg}) < 1$ and $\rho = \max(\delta,r)$,
then,
for all $x, y \in S \setminus \left( \fS^{C\rho} \right)$ with $|x - y | \leq \rho$, we have
\begin{equation} \label{eqConcentrationThetadeltaN}
 \E\left[ \left| \theta_{\delta, N}^{\bf Y}(x) -\theta(x) \right| \right] \leq M \left(  \delta^{\min(a,b)} + \frac{1}{\sqrt{N \delta^d}} \right)
\end{equation}
and
\begin{equation} \label{eqMeanLipThetaN}
 \E \left[ | \theta_{\delta, N}^{\bf Y}(x) - \theta_{\delta, N}^{\bf Y}(y)| \right] \leq M \left( |x-y|^b + \delta^{\min(a,b)} + \frac{1}{\sqrt{N \delta^d}}\right) \: ,
\end{equation}
as well as\begin{equation} \label{eqConcentrationSigmaN}
 \E  \left[ \left\| \sigma_{r,\delta,N}^{\tilde{\bf Y}, {\bf Z}}(x) - \Pi_{T_x S} \right\| \right] 
 \leq M \left( r^{\min(a,b)} + \delta^{\min(a,b)} + \frac{1}{\sqrt{N r^d}} + \frac{1}{\sqrt{N \delta^d}} \right)
\end{equation}
and
\begin{equation} \label{eqMeanLipSigmaN}
 \E  \left[ \left\| \sigma_{r,\delta,N}^{\tilde{\bf Y}, {\bf Z}}(x) - \sigma_{r,\delta,N}^{\tilde{\bf Y}, {\bf Z}}(y) \right\| \right] 
 \leq M \left( |x-y|^{a} + r^{\min(a,b)} + \delta^{\min(a,b)} + \frac{1}{\sqrt{N r^d}} + \frac{1}{\sqrt{N \delta^d}} \right) \: .
\end{equation}
\end{lemma}

\begin{proof}
The proof is a straightforward combination of previously stated results, we recall that $\fS = S_{sg} \cup \Theta_{sg}$.
Let $x, y \in S \setminus \fS^{C\rho} \subset S \setminus \fS^{C\delta}$ with $|x - y | \leq \rho$.
Applying Proposition~\ref{propDensityPointwiseCv} and Proposition~\ref{coroUniformCVHolder}, we have
\begin{equation*}
\E\left[ \left| \theta_{\delta, N}^{\bf Y}(x) -\theta(x) \right| \right] \leq \E\left[\left|\theta_{\delta, N}^{\bf Y}(x) -\theta_\delta(x)\right|\right] + \left| \theta_\delta(x)-\theta(x) \right|  \leq M \left( \frac{1}{\sqrt{N \delta^d}} + \delta^{\min(a,b)} \right) \: .
\end{equation*}
The same estimate holds at point $y$ and thus,
triangular inequality and \eqref{hypH7} conclude the proof of \eqref{eqMeanLipThetaN}:
\begin{align*}
\E \left[ | \theta_{\delta, N}^{\bf Y}(x) - \theta_{\delta, N}^{\bf Y}(y)| \right] & \leq \E\left[\left|\theta_{\delta, N}^{\bf Y}(x) -\theta(x)\right|\right] + \E\left[\left|\theta_{\delta, N}^{\bf Y}(y) -\theta(y)\right|\right] + \left| \theta(x) - \theta (y) \right| \\
&\leq 2M \left(\delta^{\min(a,b)} + \frac{1}{\sqrt{N \delta^d}}\right) +  C |x-y|^b \: . 
\end{align*}

\noindent Thanks to Proposition~\ref{propCVSigmaN} and Proposition~\ref{coroUniformCVHolder}, we similarly have
\begin{equation*}
\E\left[\left\| \Sigma_{r} (x , \mu_N^{\bf Z}) -\theta(x) \Pi_{T_x S} \right\| \right] \leq  M \left( \frac{1}{\sqrt{N r^d}} + r^{\min(a,b)} \right) \: .
\end{equation*}
Then, recalling that $\Pi_{T_x S} = \Phi(\theta(x)) \theta(x) \Pi_{T_x S}$, we have
\begin{align*}
& \E \left[ \left\| \sigma_{r,\delta,N}^{\tilde{\bf Y}, {\bf Z}} (x)  - \Pi_{T_x S} \right\| \right]  \\
\leq & \E  \left[ \left\| \Phi \left( \theta_{ \delta, N}^{\tilde{\bf Y}} (x) \right)  \Sigma_{r} (x , \mu_N^{\bf Z}) - \Phi \left( \theta_{ \delta, N}^{\tilde{\bf Y}} (x) \right) \theta(x) \Pi_{T_x S} \right\| \right] + \E \left[ \left\| \Phi \left( \theta_{ \delta, N}^{\tilde{\bf Y}} (x)\right) \theta(x) \Pi_{T_x S} - \Phi(\theta(x)) \theta(x) \Pi_{T_x S} \right\| \right] \\
 \leq & \| \Phi \|_\infty M \left( \frac{1}{\sqrt{N r^d}} + r^{\min(a,b)} \right) + \theta_{max} \xLip(\Phi) \E\left[\left|\theta_{\delta, N}^{\bf \tilde{Y}}(x) -\theta(x)\right|\right]  \\
\leq & M \left( \frac{1}{\sqrt{N r^d}} + r^{\min(a,b)} +  \frac{1}{\sqrt{N \delta^d}} + \delta^{\min(a,b)} \right) \quad \text{thanks to \eqref{eqConcentrationThetadeltaN}, } {\bf \tilde{Y}}, {\bf Y} \text{ having same law,} 
\end{align*}
hence proving \eqref{eqConcentrationSigmaN}.
It remains to combine \eqref{eqConcentrationSigmaN} at $x$ and $y$ with
\eqref{hypH5}: $\displaystyle \left\| \Pi_{T_x S} - \Pi_{T_y S} \right\| \leq C |x - y |^a$, to obtain \eqref{eqMeanLipSigmaN} and conclude the proof of Lemma~\ref{lemMeanLipEstimators}.
\end{proof}

We observe that when adding up bounds \eqref{eqMeanLipThetaN} and \eqref{eqMeanLipSigmaN}, $\delta$ and $r$ play symmetric roles, that is also true concerning the control we obtained in Proposition~\ref{propCvHolderW} (see the second estimate) for the deterministic term $\beta_B (W_{r,\delta}, W_S)$ and we consequently fix $\delta = r$ hereafter. We thus eventually consider
\begin{equation} \label{eqDfnNuHatWHat}
\widehat{\nu}_{\delta, N} = 
\Phi \left( \theta_{ \delta, N}^{\bf Y} \right) \mu_N^{\bf X} \quad \text{and} \quad \widehat{W}_{\delta, N} = \widehat{\nu}_{\delta, N} \otimes \delta_{\sigma_{\delta,N}} \quad \text{with} \quad  \sigma_{\delta,N} = \sigma_{\delta,\delta,N}^{\tilde{\bf Y}, {\bf Z}} = \Phi \left( \theta_{ \delta, N}^{\tilde{\bf Y}} \right) \Sigma_\delta (\cdot, \mu_N^{\bf Z}) \: .
\end{equation}
We also define the varifold estimator 
\begin{equation} \label{eqDfnVN}
\widehat{V}_{\delta, N} = \widehat{\nu}_{\delta, N} \otimes \delta_{\pi_{\delta,N}} \: , 
\end{equation}
where we recall that $\pi_{\delta,N} = \pi_{\delta, r ,N}$ for $\delta = r$ is the truncation of the symmetric matrix $\sigma_{\delta,N}$ according to its $d$ highest eigenvalues as explained in Section~\ref{secCVvarifold} and we recall \eqref{eqFrobPiNsigmaN}:
for $x \in S$,
\begin{equation} \label{eqStabSigmardeltaN}
 \| \sigma_{\delta,N}(x) - \pi_{\delta,N}(x)  \| \leq  \|  \sigma_{\delta,N}(x) - \Pi_{T_x S}  \|  \: .
\end{equation}

\begin{theorem} \label{thmCvHolderSplit}
Let $0 < a,b \leq 1$.
We assume that 
$S$ and $\theta$ satisfy assumptions \hyperref[hypH1]{$(H_1)$} to \eqref{hypH7} i.e. $\mu = \theta \cH^d_{| S} \in \cP $. Let $D \subset \R^n$ be a bounded open set. We recall that $R = \min (R_{S,sg}, R_{\theta, sg}) < 1$ and
let $N \in \N^\ast$ be large enough so that $\delta_N := N^{\displaystyle -\frac{1}{d + 2 \min (a,b)}} \leq \frac{R}{20C}$.
Then,
\begin{align} \label{eqThmVFfoldCV1SplitNuD}
\left.
\begin{array}{l}
\E \left[ \beta_D \left(\widehat{\nu}_{\delta_N,N} , \nu \right) \right] \\
 \E \left[ \beta_D (\widehat{W}_{\delta_N,N}, W_S) \right] \\
\E \left[ \beta_D (\widehat{V}_{\delta_N,N}, W_S) \right] 
\end{array} \right\rbrace \leq
M \mu(D^{\delta_N})  \delta_N^{\min(a,b)}  + M \sum_{l=0}^{d-1} \delta_N^{d-l} \cH^{l} \left( \fS_l \cap D^{20 C \delta_N}  \right) \leq \widetilde{M} \:  \delta_N^{\min(a,b)}  \: ,
\end{align}
with $\widetilde{M} = M \left( \mu (D^1) + \sum_{l=0}^{d-1} \cH^{l} \left( \fS_l \cap D^1  \right) \right) \leq M \left( 1 + \sum_{l=0}^{d-1} \cH^{l} ( \fS_l) \right)$.\\
In the particular case where $D = B$ is a ball of radius $R_B \leq \frac{R}{3}$, we have the local estimate
\begin{align} \label{eqThmVFfoldCV1SplitNu}
\left.
\begin{array}{l}
\E \left[ \beta_B \left(\widehat{\nu}_{\delta_N,N} , \nu \right) \right] \\
 \E \left[ \beta_B (\widehat{W}_{\delta_N,N}, W_S) \right] \\
\E \left[ \beta_B (\widehat{V}_{\delta_N,N}, W_S) \right] 
\end{array} \right\rbrace \leq
M \mu(B) \left(  \delta_N^{\min(a,b)} + \sum_{l=0}^{d-1} \xi_l \frac{ \delta_N^{d-l} }{R_B^{d-l}}\right)  \: ,
\end{align}
where $\xi_l = \xi_l (20C \delta_N , B)$ satisfies $\xi_l = 1$ if $B \cap \fS^{20C\delta_N} \neq \emptyset$ and $\xi_l = 0$ otherwise.
\end{theorem}

\begin{proof}[Proof of Theorem~\ref{thmCvHolderSplit}]
Let $D \subset \R^n$ be a bounded open set.
We fix $N \in \N^\ast$, $0 < \delta  \leq \frac{R}{20C}$ satisfying $N^{-\frac{1}{d}} \leq \delta$ (and we recall that we fixed $r = \delta$ in \eqref{eqDfnNuHatWHat}). 
In view of proving \eqref{eqThmVFfoldCV1SplitNu}, we will consider two types of test functions:
\begin{enumerate}[$(i)$]
 \item either $f \in \xC_c (\R^n, \R)$ is a $1$--Lipschitz function such that $\| f \|_\infty \leq 1$ and 
 $\supp f \subset D$, and in this case we introduce $g_{\delta,N}(x) = g_{\delta,N}^{\bf Y}(x) = f(x) \Phi \left( \theta_{\delta, N}^{\bf Y}(x) \right)$ and $g_{\delta}(x) = f(x) \Phi \left( \theta_{\delta}(x) \right)$;
 \item or $f \in \xC_c (\R^n \times \xSym_+(n), \R)$ is a $1$--Lipschitz function such that $\| f \|_\infty \leq 1$ and
 $\supp f \subset D \times \xSym_+(n)$, and in this case we introduce $g_{\delta,N}(x) = g_{\delta,N}^{{\bf Y}, \tilde{\bf Y}, {\bf Z}}(x) = f\left( x , \sigma_{\delta,N}^{\tilde{\bf Y}, {\bf Z}}(x) \right) \Phi \left( \theta_{\delta, N}^{\bf Y}(x) \right)$ and $g_{\delta}(x) = f\left( x , \sigma_{\delta,\delta}(x) \right) \Phi \left( \theta_{\delta}(x) \right)$.
\end{enumerate}
In both cases, 
$g_{\delta,N}$ and $\mu_N^{\bf X}$ are independent, we will shorten $(i)$ or $(ii)$ with $f \in \xBL$ so that
\begin{align} \label{eqInitSplit}
\left.
\begin{array}{r}
 \displaystyle \E \left[ \beta_D \left(\widehat{\nu}_{\delta,N} , \nu_\delta \right) \right] \text{ in case } (i) \\
 \displaystyle  \E \left[ \beta_D \left(\widehat{W}_{\delta,N} , W_{\delta, \delta} \right) \right] \text{ in case } (ii) 
 \end{array}
 \right\rbrace & = \E \left[ \sup_{f \in \xBL} \left| \int_D  g_{\delta,N}  \: d \mu_{N}^{\bf X} - \int_D g_\delta \:  d \mu  \right| \right] \nonumber \\
 & \leq 
 \E \left[ \sup_{f \in \xBL} \left| \int_D  g_{\delta,N}  \: \left( d \mu_{N}^{\bf X} -  d \mu \right) \right| \right] + \E \left[ \sup_{f \in \xBL} \left| \int_D  \left( g_{\delta,N} - g_\delta \right)  \:  d \mu  \right| \right] \nonumber \\
 & \leq \E \left[ \sup_{f \in \xBL} \left| \int_D  g_{\delta,N}  \: \left( d \mu_{N}^{\bf X} -  d \mu \right) \right| \right] + \frac{M}{\sqrt{N \delta^d}} \mu(D) \: ,
\end{align}
where the last estimate follows from Proposition~\ref{propDensityPointwiseCv} and Proposition~\ref{propCVSigmaN} since 
\begin{equation*}
\E \left[ \sup_{f \in \xBL} \int_D \left| g_{\delta,N} - g_\delta \right| \: d \mu \right] \leq \E \left[ \| \Phi \|_\infty  \int_D \| \sigma_{\delta,N} - \sigma_{\delta,\delta} \| \: d \mu + \xLip(\Phi) \int_D \left| \theta_{\delta,N} - \theta_\delta \right| \: d \mu \right] \leq \frac{M}{\sqrt{N \delta^d}} \mu(D) \: .
\end{equation*}
We are left with estimating $\displaystyle \E \left[ \sup_{f \in \xBL} \left| \int_D  g_{\delta,N}  \: \left( d \mu_{N}^{\bf X} -  d \mu \right) \right| \right]$, which is the purpose of Steps $1$ and $2$. As we explained in the introduction of Section~\ref{secSplit}, we want to leverage the mean Lipschitz estimates \eqref{eqMeanLipThetaN} and \eqref{eqMeanLipSigmaN} that are finer than the worst case Lipschitz estimate in the regularity class $\cP$, which prevents from applying Proposition~\ref{propBetaLoc}. Instead, we follow and adapt the steps of the proof of Proposition~\ref{propBetaLoc} hereafter.

We use the notations $\epsilon = N^{-\frac{1}{d}} \leq \delta$, $\varepsilon_u := 3^{-u}$ for $u \in \N$, and $T := D \cap S \setminus \fS^{10C\delta}$.

\noindent We define integers $0 \leq s \leq t$ such that
\[
 3^{-(t+1)}<  \epsilon \leq 3^{-t} = \varepsilon_t \quad \text{and}\quad 
3^{-(s+1)}< \epsilon^{\alpha} \leq 3^{-s} = \varepsilon_s \quad \text{and}\quad \epsilon_s \leq \delta \: ,
\]
where $0 < \alpha \leq 1$ is defined hereafter depending on the case $(i)$ or $(ii)$.
We recall that $T$ can be partitioned with $m_u$ pieces of diameter $\leq \varepsilon_u$ and $m_u$ furthermore satisfies the estimate \eqref{eqmr}
\begin{equation*} 
m_u \leq 4^d C_0 \varepsilon_u^{-d} \mu \left(D^\frac{\varepsilon_u}{4} \right) \: .
\end{equation*}
We can apply Lemma~\ref{lemNestedPartitions} to such partitions and define nested partitions $\left\lbrace A_j^u \: : \: u = s, \ldots, t, \, j= 1, \ldots, m_u \right\rbrace$ satisfying \eqref{eqPartitionProperty} and \eqref{eqNestedProperty}.
For each $u=s,\dots , t$ and $j=1,\dots,m_u$ we choose $x^u_j\in A_j^u$ and 
we introduce 
\[
M_u := \sum_{j=1}^{m_u} | \mu_N^{\bf X}(A^u_j)-\mu(A^u_j) | 
\quad \text{and} \quad
I_u:=\left|\sum_{j=1}^{m_u}g_{\delta,N}(x_j^u)(\mu_N^{\bf X} (A^u_j)-\mu(A^u_j))\right| 
 \: .
\]

\smallskip
\noindent {\bf Step $1$:} we can prove the following control:
\begin{multline} \label{eqBetagBis}
 \E \left[ \sup_{f \in \xBL} \left| \int_D  g_{\delta,N}  \: \left( d \mu_{N}^{\bf X} -  d \mu \right) \right| \right] \\
 \leq \E \left[ \sup_{f \in \xBL} I_t \right] + M \left( \epsilon^{\min(a,b)} + \delta^{\min (a,b)} + \frac{1}{\sqrt{N \delta^d}}   \right) \mu(D) +  M \sum_{l=0}^{d-1} \delta^{d-l} \cH^l (\fS_l \cap D^{20C\delta}) \: .
\end{multline}
Indeed, we remind (see \eqref{eqPartitionProperty}) that $(A_j^t)_{j=1}^{m_t}$ is a partition of $T =D \cap S \setminus \fS^{10C\delta}$ and for all $j$, $\xdiam A_j^t \leq 3 \varepsilon_t \leq 9 \epsilon$. As $\| g_{\delta,N} \|_\infty \leq \|f\|_\infty \| \Phi \|_\infty \leq M$, we first have
\begin{equation} \label{eqRemoveSingSet}
\Bigg| \int_D  g_{\delta,N}   \: \left( d \mu_{N}^{\bf X} -  d \mu \right) \Bigg|  
\leq M \left( \mu_N^{\bf X} \ \left( D \cap \fS^{10C\delta} \right) 
+ \mu  \left( D \cap \fS^{10C\delta}  \right) \right)  + 
\left| \int_T  g_{\delta,N}   \: \left( d \mu_{N}^{\bf X} -  d \mu \right) \right| \: .
\end{equation}
The first term in the right hand side of \eqref{eqRemoveSingSet} is now independent of $f$ and when taking the mean value, we have by \eqref{eqThickD} (with $\rho = 10 C\delta \leq R$):
\begin{align} \label{eqBetagBis1}
\E \left[ \mu_N^{\bf X} \left( D \cap \fS^{10C\delta}  \right) \right]  & = \mu  \left( D \cap \fS^{10C\delta}  \right) \leq \theta_{max} \cH^ d \left( D \cap \fS^{10C\delta} \cap S \right) \leq M \sum_{l=0}^{d-1} \delta^{d-l} \cH^l (\fS_l \cap D^{20C\delta})  \: .
\end{align}
Moreover, regarding the last term in the right hand side of \eqref{eqRemoveSingSet}, we have for $x \in A_j^t$, $| x - x_j^t| \leq \xdiam A_j^t \leq 9 \epsilon$ and thus, in case $(i)$ or $(ii)$:
\begin{align} \label{eqLipgBis}
 \left| g_{\delta,N}(x) - g_{\delta,N}(x_j^t) \right|
 & \leq  \| f \|_\infty \xLip(\Phi) \left| \theta_{\delta, N}^{\bf Y}(x) -\theta_{\delta,N}^{\bf Y}(x_j^t) \right| + \| \Phi \|_\infty \left\lbrace
 \begin{array}{lr}
 \left| f(x) - f(x_j^t) \right| & (i)\\ 
 \left| f\left( x , \sigma_{\delta,N}(x) \right) - f\left( x_j^t , \sigma_{\delta,N}(x_j^t) \right) \right| & (ii)
 \end{array}
 \right. \nonumber \\
 & \leq M \left| \theta_{\delta, N}^{\bf Y}(x) -\theta_{\delta,N}^{\bf Y}(x_j^t) \right| + M \left\lbrace
 \begin{array}{lr}
 9 \epsilon & (i) \\ 
 9 \epsilon + \left\| \sigma_{\delta,N}(x) -  \sigma_{\delta,N}(x_j^t)  \right\| & (ii)
 \end{array}
 \right. \: .
\end{align}
Observing that for $\delta = r$, the right hand side of \eqref{eqMeanLipThetaN} and \eqref{eqMeanLipSigmaN} are similar and bounded by 
\[
M \left( |x-x_j^t|^{\min(a,b)} +  \delta^{\min(a,b)} + \frac{1}{\sqrt{N \delta^d}} \right) \leq M \left( (9\epsilon)^{\min(a,b)} +  \delta^{\min(a,b)} + \frac{1}{\sqrt{N \delta^d}} \right) \: ,
\]
we can handle $(i)$ and $(ii)$ in the same way and we obtain thanks to \eqref{eqLipgBis}: 
\begin{multline} \label{eqBetagBis2}
\left| \int_T  g_{\delta,N}  \left( d \mu_{N}^{\bf X} -  d \mu \right) \right|  = \left| \sum^{m_t}_{j=1} \int_{A_j^t} g_{\delta,N}(x) \: d(\mu_N^{\bf X} -\mu)(x)\right|
\leq I_t + \left| \sum^{m_t}_{j=1} \int_{A_j^t} \left(g_{\delta,N}(x) - g_{\delta,N}(x_j^t) \right) \: d(\mu_N^{\bf X} -\mu)(x)\right| \\
\leq  I_t+ M \epsilon \left( \mu_N^{\bf X} (D) + \mu(D) \right) + M \sum^{m_t}_{j=1}\int \limits_{A_j^t} \left| \theta_{\delta, N}^{\bf Y}(x) -\theta_{\delta,N}^{\bf Y}(x_j^t) \right| + \left\| \sigma_{\delta,N}(x) -  \sigma_{\delta,N}(x_j^t)  \right\| \: d (\mu_N^{\bf X} +\mu)(x) \: . 
\end{multline}
By independence of ${\bf X}$ and ${\bf Y}$ and Lemma~\ref{lemMeanLipEstimators} (noting that $x, x_j^t \in S \setminus \fS^{C10\delta}$ and $|x-x_j^t| \leq 9 \epsilon \leq 10 \delta$), we have
\begin{align} \label{eqBetagBis3}
\E \left[ \sum^{m_t}_{j=1}\int \limits_{A_j^t} \left| \theta_{\delta, N}^{\bf Y}(x) -\theta_{\delta,N}^{\bf Y}(x_j^t) \right| \: d (\mu_N^{\bf X} +\mu)(x) \right] & = \E_{\bf X} \left[ \sum^{m_t}_{j=1}\int \limits_{A_j^t} \E_{\bf Y} \left[ \left| \theta_{\delta, N}^{\bf Y}(x) -\theta_{\delta,N}^{\bf Y}(x_j^t) \right| \right] \: d (\mu_N^{\bf X} +\mu)(x) \right] \nonumber \\
& \leq 2 M \left(  \epsilon^{\min(a,b)} + \delta^{\min(a,b)} + \frac{1}{\sqrt{N \delta^d}}\right) \mu (D) \: ,
\end{align}
and similarly
\begin{align} \label{eqBetagBis4}
\E \left[ \sum^{m_t}_{j=1}\int \limits_{A_j^t} \left\| \sigma_{\delta,N}(x) -  \sigma_{\delta,N}(x_j^t)  \right\| \: d (\mu_N^{\bf X} +\mu)(x) \right]  \leq  M \left( \epsilon^{\min(a,b)} + \delta^{\min(a,b)} + \frac{1}{\sqrt{N \delta^d}}\right) \mu (D) \: .
\end{align}
Combining \eqref{eqRemoveSingSet}, \eqref{eqBetagBis1}, \eqref{eqBetagBis2}, \eqref{eqBetagBis3}, \eqref{eqBetagBis4}  with $\epsilon \leq \epsilon^{\min(a,b)}$ we finally infer \eqref{eqBetagBis} and conclude Step $1$.

\smallskip
\noindent {\bf Step $2$:} We now estimate $I_t$ by induction on $u$ from $t$ to $s$:
\begin{equation} \label{eqItgBis}
\E \left[ \sup_{f \in \xBL} I_t \right] \leq M \mu \left( D^\delta \right) \left(  \epsilon^{(1- \alpha)\frac{d}{2}} +  \delta^{\min (a,b)} + \frac{\epsilon^{\frac{d}{2}}}{\delta^{\frac{d}{2}}}   +
\epsilon^{\frac{d}{2}} \sum_{u=s+1}^t  \epsilon_u^{\min(a,b) -\frac{d}{2}} \right) \: .
\end{equation}
Indeed, first recall that $\| g_{\delta,N} \|_\infty \leq \| \Phi \|_\infty \leq M$ so that $I_s \leq M M_s$.
Let now $s < u \leq t$, using \eqref{eqNestedProperty} and the Lipschitz estimate \eqref{eqLipgBis} for $g_{\delta,N}$ (between $x_j^u$ and $x_q^{u-1}$ satisfying $\left| x_j^u - x_q^{u-1} \right| \leq 3 \epsilon_{u-1} \leq 9 \epsilon_u$ instead of $x$ and $x_j^t$) we obtain the following recurrence relation:
\begin{align*}
I_u & \leq\Bigg| \sum^{m_{u-1}}_{q=1} \sum_{j \in I_{q,u}}  \left( g_{\delta,N}(x_j^u) -g_{\delta,N}(x_q^{u-1}) \right)  \: \left(\mu_N^{\bf X} (A_j^u) -\mu(A_j^u) \right) \Bigg| + \Bigg|\sum^{m_{u-1}}_{q=1} g_{\delta,N}(x_q^{u-1}) \: \underbrace{ \sum_{j \in I_{q,u}}  \left(\mu_N^{\bf X} (A_j^u) -\mu(A_j^u) \right) }_{= \mu_N^{\bf X}(A_q^{u-1}) - \mu (A_q^{u-1}) } \Bigg| \nonumber \\
& \leq  \sum^{m_{u-1}}_{q=1} \sum_{j \in I_{q,u}} M \left( \epsilon_{u} + \left| \theta_{\delta, N}^{\bf Y}(x_j^u) -\theta_{\delta,N}^{\bf Y}(x_q^{u-1}) \right| + \left\| \sigma_{\delta,N}(x_j^u) -  \sigma_{\delta,N}(x_q^{u-1})  \right\| \right)   \left| \mu_N^{\bf X} (A_j^u) -\mu(A_j^u) \right|  + I_{u-1} \: . \nonumber \\
\end{align*}
We can now take the supremum with respect to $f \in  \xBL$ and use the independence of ${\bf X}$ and $\left( {\bf Y}, \tilde{\bf Y}, {\bf Z} \right)$ to infer the following relation:
\begin{align}
\E  \left[ \sup_{f \in \xBL} I_u \right] 
& \leq  \E \left[ \sup_{f \in \xBL} I_{u-1} \right] 
 + M  \sum^{m_{u-1}}_{q=1} \sum_{j \in I_{q,u}}  \E_{\bf X} \left[  \left| \mu_N^{\bf X} (A_j^u) -\mu(A_j^u) \right| \right] \times \nonumber \\
 & \qquad \qquad \qquad \qquad \left( \epsilon_{u} +  \E_{\bf Y} \left[ \left| \theta_{\delta, N}^{\bf Y}(x_j^u) -\theta_{\delta,N}^{\bf Y}(x_q^{u-1}) \right| \right] + \E_{\tilde{\bf Y}, {\bf Z}} \left[ \left\| \sigma_{\delta,N}(x_j^u) -  \sigma_{\delta,N}(x_q^{u-1})  \right\| \right] \right)    \nonumber \\
& \leq  \E \left[ \sup_{f \in \xBL} I_{u-1} \right] + M \left(  \epsilon_u^{\min(a,b)} + \delta^{\min (a,b)} + \frac{1}{\sqrt{N \delta^d}}  \right) \E \left[ M_u \right] \: . 
\label{eqItRecBis} 
\end{align}
By induction from $u=t$ down to $s$ on \eqref{eqItRecBis} and recalling that $I_s \leq M M_s$, we have the following control:
\begin{equation*} 
\E \left[ \sup_{f \in \xBL} I_t \right] \leq M \left(  \E \left[ M_s \right] +  \left( \delta^{\min (a,b)} + \frac{1}{\sqrt{N \delta^d}} \right)  \sum_{u=s+1}^t \E \left[ M_u \right] +
\sum_{u=s+1}^t  \epsilon_u^{\min(a,b)}
\E \left[ M_u \right]  \right) \: .
\end{equation*}
We recall \eqref{eqmrMuN}:
$\displaystyle
\E\left[ M_u \right] \leq  \frac{M}{N^\frac{1}{2}} \varepsilon_u^{-\frac{d}{2}} \mu \left( D^\frac{\varepsilon_s}{4} \right)$ (and $N^\frac{1}{2} = \epsilon^{-\frac{d}{2}}$)
and moreover $\frac{\epsilon_s}{4} \leq \delta$ so that 
$\mu \left( D^\frac{\varepsilon_s}{4} \right) \leq M \mu (D^\delta)$.
Consequently
\begin{align} \label{eqSupIt1}
\E \left[ \sup_{f \in \xBL} I_t \right] & \leq \frac{M}{\sqrt{N}} \mu \left( D^\delta \right) \left(  \epsilon_s^{-\frac{d}{2}} +  \left( \delta^{\min (a,b)} + \frac{1}{\sqrt{N \delta^d}} \right)  \sum_{u=s+1}^t \epsilon_u^{-\frac{d}{2}} +
\sum_{u=s+1}^t  \epsilon_u^{\min(a,b) -\frac{d}{2}} \right)  \nonumber \\
& \leq M \mu \left( D^\delta \right) \left(  \epsilon^{(1-\alpha)\frac{d}{2}} +  \epsilon^{\frac{d}{2}}\left( \delta^{\min (a,b)} + \frac{\epsilon^{\frac{d}{2}}}{\delta^{
\frac{d}{2}}} \right)  \sum_{u=s+1}^t \epsilon_u^{-\frac{d}{2}} + \epsilon^{\frac{d}{2}}
\sum_{u=s+1}^t  \epsilon_u^{\min(a,b) -\frac{d}{2}} \right)  \: .
\end{align}
Note that for $\varepsilon_u^{-d/2}=(3^{-u})^{-d/2}=(3^{d/2})^{u}$ and $3^{d/2} \geq \sqrt{3} > 1$ so that 
\begin{equation*} 
\sum_{u=s+1}^t\varepsilon_u^{-d/2} \leq \sum_{u=0}^t (3^{d/2})^{u} \leq \frac{(3^{d/2})^{t+1}}{3^{d/2} - 1} \leq \frac{3^{d/2}}{3^{d/2} - 1} \varepsilon_t^{-d/2} \leq 3 \epsilon^{-d/2} \leq 3 N^{1/2} \quad \text{and} \quad N^{1/2} = \epsilon^{-\frac{d}{2}} \: ,
\end{equation*}
and together with \eqref{eqSupIt1} we can infer \eqref{eqItgBis}.

\smallskip
\noindent {\bf Conclusion of Step $1$ and $2$:} 
We can now draw an intermediate conclusion in the proof of \eqref{eqThmVFfoldCV1SplitNu}. Indeed, thanks to \eqref{eqInitSplit}, \eqref{eqBetagBis} and \eqref{eqItgBis},
\begin{multline*} 
 \left.
\begin{array}{c}
 \displaystyle \E \left[ \beta_D \left(\widehat{\nu}_{\delta,N} , \nu_\delta \right) \right] \\
 \displaystyle  \E \left[ \beta_D (\widehat{W}_{\delta,N} , W_{\delta, \delta} ) \right] 
 \end{array}
 \right\rbrace \leq  M \sum_{l=0}^{d-1} \delta^{d-l} \cH^l (\fS_l \cap D^{20C\delta})\\
 + M  \mu(D^\delta)  \left( \epsilon^{(1 - \alpha)\frac{d}{2}} + \epsilon^{\min(a,b)} + \delta^{\min (a,b)} + \frac{\epsilon^{\frac{d}{2}}}{\delta^{\frac{d}{2}}} 
 + \epsilon^{\frac{d}{2}} \sum_{u=s+1}^t  \epsilon_u^{\min(a,b) -\frac{d}{2}} \right)  \: , 
\end{multline*}
and thanks to Proposition~\ref{cvnudel} and~\ref{propCvHolderW} (with $r = \delta$) and triangular inequality, we infer the same control for $\nu$, $W_S$ instead of $\nu_\delta$, $W_{\delta,\delta}$:
\begin{multline} \label{eqIntermediateConclBis}
 \left.
\begin{array}{c}
 \displaystyle \E \left[ \beta_D \left(\widehat{\nu}_{\delta,N} , \nu \right) \right] \\
 \displaystyle  \E \left[ \beta_D (\widehat{W}_{\delta,N} , W_S ) \right] 
 \end{array}
 \right\rbrace \leq  M \sum_{l=0}^{d-1} \delta^{d-l} \cH^l (\fS_l \cap D^{20C\delta})\\
 + M  \mu(D^\delta)  \left( \epsilon^{(1 - \alpha)\frac{d}{2}} + \epsilon^{\min(a,b)} + \delta^{\min (a,b)} + \frac{\epsilon^{\frac{d}{2}}}{\delta^{\frac{d}{2}}} 
 + \epsilon^{\frac{d}{2}} \sum_{u=s+1}^t  \epsilon_u^{\min(a,b) -\frac{d}{2}} \right)  \: .
\end{multline}
It remains to simplify the last sum above depending on $\min(a,b) -\frac{d}{2}$.

\smallskip
\noindent {\bf Step $3$:} we now prove that
\begin{equation}\label{eqCvHolderSplitStep3}
\left.
\begin{array}{c}
 \displaystyle \E \left[ \beta_D \left(\widehat{\nu}_{\delta,N} , \nu \right) \right] \\
 \displaystyle  \E \left[ \beta_D (\widehat{W}_{\delta,N} , W_S ) \right] 
 \end{array}
 \right\rbrace \leq M \sum_{l=0}^{d-1} \delta^{d-l} \cH^l (\fS_l \cap D^{20C\delta}) + M \delta^{ \min(a,b)}\: \mu(D^\delta) \quad \text{with} \quad \delta = \delta_N = N^{\displaystyle -\frac{1}{d + 2 \min (a,b)}} \: .
\end{equation}

{\bf -- Case $\min(a,b) -\frac{d}{2} < 0$:} note that $\min(a,b) \leq 1$ whence the current case covers $d > 2$, $d = 2$ and $\min(a,b) \neq 1$, $d = 1$ and $\min(a,b) < \frac{1}{2}$.\\
Let us temporarily use the notation $c = \min(a,b)$.
Similarly to \eqref{eqBetaT_2}, we have $\varepsilon_u^{-d/2+ c}=(3^{-u})^{-d/2+ c}=(3^{d/2-c})^{u}$ and $\frac{d}{2} - c > 0$ implies $3^{d/2-c} > 1$ so that
\begin{equation*} 
\sum_{u=s+1}^t\varepsilon_u^{-d/2+c} \leq \sum_{u=0}^t (3^{d/2-c})^{u} \leq \frac{(3^{d/2-c})^{t+1}}{3^{d/2-c} - 1} \leq \frac{3^{d/2-c}}{3^{d/2-c} - 1} \varepsilon_t^{-d/2+c} \leq M \epsilon^{-d/2+c} \ \: .
\end{equation*}
We recall that we require $\epsilon^\alpha \leq \delta$ and writing $\delta = \epsilon^\beta$, it amounts to require $0 <  \beta \leq \alpha < 1$.
We can thus rewrites \eqref{eqIntermediateConclBis} as follows:
\begin{equation*} 
 \left.
\begin{array}{c}
 \displaystyle \E \left[ \beta_D \left(\widehat{\nu}_{\delta,N} , \nu \right) \right] \\
 \displaystyle  \E \left[ \beta_D (\widehat{W}_{\delta,N} , W_S ) \right] 
 \end{array}
 \right\rbrace \leq M \mu(D^\delta)  \left( \epsilon^{(1 - \alpha)\frac{d}{2}} + \epsilon^{c} + \delta^{c} + \frac{\epsilon^{\frac{d}{2}}}{\delta^{\frac{d}{2}}}  \right) + M \sum_{l=0}^{d-1} \delta^{d-l} \cH^l (\fS_l \cap D^{20C\delta}) \: ,
\end{equation*}
and we are left with carefully choosing $\alpha$ and $\delta = \epsilon^\beta$ satisfying $0 < \beta \leq \alpha < 1$. Note that $\alpha$ is a parameter of the proof while $\delta$ is a parameter of the estimators $\widehat{\nu}_{\delta, N}$, $\widehat{W}_{\delta, N}$.
Let us first focus on the $3$ terms involving $\delta$: $\displaystyle \delta^c = \epsilon^{\beta c}$, $\displaystyle \frac{\epsilon^{\frac{d}{2}}}{\delta^{\frac{d}{2}}} = \epsilon^{\frac{d}{2} - \frac{d}{2} \beta}$ and $\displaystyle \delta^{d-l} = \epsilon^{ (d-l) \beta}$, since $(d-l) \geq 1 \geq c$, we can optimize $\displaystyle \epsilon^{\beta c} + \epsilon^{\frac{d}{2} - \frac{d}{2} \beta}$ which gives
\begin{equation*}
 \beta c = \frac{d}{2} - \frac{d}{2} \beta \quad \Longleftrightarrow \quad \beta = \frac{d}{d + 2 c} 
\end{equation*}
and since $\displaystyle \frac{d c}{d + 2 c} \leq c \leq 1$ we then have
$\displaystyle \max \left\lbrace \delta^{c} , \: \left(\frac{\epsilon}{\delta} \right)^\frac{d}{2} , \: \delta^{d-l} , \epsilon^{c} \right\rbrace \leq \epsilon^{\displaystyle\frac{d c}{d + 2 c}}$.
Furthermore minimizing $\displaystyle \epsilon^\frac{d c}{d + 2 c} + \epsilon^{(1 - \alpha) \frac{d}{2}}$ we obtain
\[
 (1 - \alpha) \frac{d}{2} = \frac{d c}{d + 2 c } \quad \Longleftrightarrow \quad  \alpha =  \frac{d}{d + 2 \min(a,b)} = \beta
\]
so that we have $\epsilon^\alpha \leq \delta$ is satisfied (actually with equality $\delta = \epsilon^\alpha$) and \eqref{eqCvHolderSplitStep3} is proved for $\min(a,b) - \frac{d}{2} < 0$.

{\bf -- Case $\min(a,b) = \frac{d}{2}$:} which corresponds to $d = 2$ and $a = b = 1$ or $d = 1$ and $\min(a,b) = \frac{1}{2}$. 
In this case, $\epsilon^{\frac{d}{2}} = \epsilon^c$, $\epsilon_u^{-d/2+c} = 1$ and then
\[
\epsilon^{\frac{d}{2}} \sum_{u=s+1}^t \epsilon_u^{-d/2+c} \leq \epsilon^c t \leq M \epsilon^c |\ln \epsilon|  \: .
\]
The $3$ terms involving $\delta$ in the right hand side of \eqref{eqIntermediateConclBis} are unchanged and the optimization again yields $\delta = \epsilon^{\frac{d}{d+2 c}} = \epsilon^\frac{1}{2}$ (since $2c = d$) and we note that for $\epsilon \in (0,1]$, $\epsilon^{\frac{1}{2}} |\ln \epsilon|  \leq 1$, $\epsilon^{\frac{1}{4}} |\ln \epsilon|  \leq 2$, and thus
\begin{equation*}
\epsilon^c |\ln \epsilon| = \left\lbrace
\begin{array}{lcl}
 \epsilon |\ln \epsilon| \leq \epsilon^{\frac{1}{2}} = \delta = \delta^c & \text{if} & d = 2, \: c= 1 \\
\epsilon^{\frac{1}{2}} |\ln \epsilon| \leq 2 \epsilon^{\frac{1}{4}} = 2 \delta^c & \text{if} & d = 1, \: c= \frac{1}{2}
\end{array}
\right\rbrace \leq 2 \delta^c = 2 \delta^{\min(a,b)}
\end{equation*}
and we obtain the same control \eqref{eqCvHolderSplitStep3} for $\min(a,b) = \frac{d}{2}$.

{\bf -- Case $\min(a,b) - \frac{d}{2} > 0$:} 
in this case $d = 1$ and we have $\varepsilon_u^{-d/2+c} = \left( 3^{\frac{1}{2} - c} \right)^u$, and $3^{\frac{1}{2} - c} < 1$ thus
\[
 \epsilon^{\frac{d}{2}} \sum_{u=s+1}^t \epsilon_u^{-d/2+c} \leq M \epsilon^{\frac{1}{2}} | \ln \epsilon |
 = M \epsilon^{\frac{1}{2(1 + 2 c)}} | \ln \epsilon | \epsilon^{\frac{c}{1 + 2 c}} 
 \leq M \epsilon^{\frac{c}{1 + 2 c}} = M \delta^c
\]
so that the control by $\delta^c$ still holds in this last case and we conclude Step $3$ and the proof of \eqref{eqCvHolderSplitStep3}.

\smallskip
\noindent {\bf Step $4$:} we now compare $\widehat{V}_{\delta, N}$ and $\widehat{W}_{\delta, N}$ and show that
\begin{equation} \label{eqBLVrdeltaNWrdeltaN}
\E \left[ \beta_D \left( \widehat{V}_{\delta, N} , \widehat{W}_{\delta, N} \right) \right]  \leq M \sum_{l=0}^{d-1} \delta^{d-l} \cH^l (\fS_l \cap D^{20C\delta}) + M \delta^{ \min(a,b)}\: \mu(D) \quad \text{with} \quad \delta = \delta_N = N^{\displaystyle -\frac{1}{d + 2 \min (a,b)}} \: .
\end{equation}
At this step, we let $f \in \xC_c (\R^n \times \xSym_+(n))$ be a $1$--Lipschitz function satisfying $\supp f \subset D \times \xSym_+(n)$ as in case $(ii)$ and such that $\| f \|_\infty \leq 1$. As in Step $1$, we write $D = (D \cap S \cap \fS^{10C\delta}) \sqcup T$ and then
\begin{align*}
 \left| \int_{D \times \xSym_+(n)} \right. & \left. \hspace{-1pt} f \: d \widehat{V}_{\delta,N} - \int_{D \times \xSym_+(n)} \hspace{-1pt} f \: d \widehat{W}_{\delta,N} \right| = \int_{D} \left| f(x, \pi_{\delta,N}(x)) - f(x, \sigma_{\delta,N}(x)) \right| \Phi \left( \theta_{\delta,N} (x) \right) \: d \mu_N (x)\\
 & \leq \| \Phi \|_\infty \int_{T} \left\|  \pi_{\delta,N}(x) - \sigma_{\delta,N}(x) \right\| \: d \mu_N (x) + 2 \| \Phi_\infty \| \| f \|_\infty \mu_N \left( \fS^{10C\delta} \cap D \cap S  \right) \\
 & \leq M \left( \int_{T} \left\|  \Pi_{T_x S} - \sigma_{\delta,N}(x) \right\| \: d \mu_N (x) + \mu_N \left(  \fS^{10C\delta} \cap D \right)  \right) \quad \text{thanks to } \eqref{eqStabSigmardeltaN} \: .
\end{align*}
The right hand side above is independent of $f$ and therefore, taking the supremum with respect to $f$ and then the mean value, we obtain thanks to \eqref{eqBetagBis1} and \eqref{eqConcentrationSigmaN}:
\begin{align*}
\E \left[ \beta_B \left( \widehat{V}_{\delta, N} , \widehat{W}_{\delta, N} \right) \right] & \leq 
M \left(  \E_{\bf X} \left[ \int_{D} \E_{\tilde{\bf Y},{\bf Z}} \left[ \left\|  \Pi_{T_x S} - \sigma_{\delta,N}^{\tilde{\bf Y},{\bf Z}} \right\| \right] \: d \mu_N^{\bf X} \right] + \mu \left( D \cap \fS^{10C\delta} \right) 
 \right) \\
& \leq M \mu(D) \left( \delta^{\min(a,b)} + \frac{\epsilon^{\frac{d}{2}}}{\delta^{\frac{d}{2}}}  \right) + M \sum_{l=0}^{d-1} \delta^{d-l} \cH^l (\fS_l \cap D^{20C\delta})
\end{align*}
which conclude the proof of \eqref{eqBLVrdeltaNWrdeltaN} with $\displaystyle \delta = \delta_N = N^{\displaystyle -\frac{1}{d + 2 \min (a,b)}}$.

\smallskip
\noindent 
We conclude the proof of \eqref{eqThmVFfoldCV1SplitNuD} by triangular inequality on the Bounded Lipschitz distance, combining \eqref{eqCvHolderSplitStep3} and \eqref{eqBLVrdeltaNWrdeltaN}.

\smallskip
\noindent {\bf Step $5$:} In the particular case where $D = B$ is a ball of radius $R_B \leq \frac{R}{3}$ and $N$ is large enough so that $\delta_N \leq R_B$, we have the local estimate \eqref{eqThmVFfoldCV1SplitNu}. Indeed, for $l \in \{ 0, 1, \ldots, d-1\}$, if $B^{20C\delta} \cap \fS_l = \emptyset$ then $\cH^l ( B^{20C\delta} \cap \fS_l) = 0$ and otherwise, there exists $x \in \fS_l$ such that $B \subset B(x, 2R_B + 20C\delta) \subset B(x, (2 + 20C)R_B)$ and
recalling \eqref{hypH4} and \eqref{hypH6}, we have: $\cH^l (\fS_l \cap B(x, (2 + 20C)R_B)) \leq M R_B^l$ so that by Ahlfors regularity of $\mu$,
\[
 \sum_{l=0}^{d-1} \delta^{d-l} \cH^l (\fS_l \cap B^{20C\delta}) \leq M \sum_{l=0}^{d-1} \xi_l \delta^{d-l} R_B^l \leq M \sum_{l=0}^{d-1} \xi_l \frac{\delta^{d-l}}{ R_B^{d-l}} \mu(B) \: .
\]
Since we also have $\mu (B^\delta) \leq \mu (2B) \leq M \mu(B)$, we conclude the proof of \eqref{eqThmVFfoldCV1SplitNu} thanks to \eqref{eqThmVFfoldCV1SplitNuD}.
\end{proof}

We finally mention the simpler following $\| \cdot \|_{\xL^1(\mu)}$ convergences that hold under the same assumptions as those of Theorem~\ref{thmCvHolderSplit}.

\begin{proposition} \label{propFinalCV}
Let $0 < a,b \leq 1$.
We assume that $S$ and $\theta$ satisfy assumptions \hyperref[hypH1]{$(H_1)$} to \eqref{hypH7} i.e. $\mu = \theta \cH^d_{| S} \in \cP$. Let $D \subset \R^n$ be an open set. 
Let $N \in \N^\ast$ be large enough so that $\delta_N := N^{\displaystyle -\frac{1}{d + 2 \min (a,b)}} \leq \frac{R}{2C}$. Then
\begin{align}
\label{eqThetaNL1D}
\left.
\begin{array}{c}
\displaystyle \E \left[ \int_D | \theta_{\delta_N,N} - \theta | \: d \mu \right] \\
\displaystyle \E \left[ \int_D \| \sigma_{\delta_N,N}(x) - \Pi_{T_x S} \| \right] \: d \mu 
\end{array} \right\rbrace \leq
M \mu(D)  \delta_N^{\min(a,b)}  + M \sum_{l=0}^{d-1} \delta_N^{d-l} \cH^{l} \left( \fS_l \cap D^{ 4C\delta_N}  \right) \leq \widetilde{M} \:  \delta_N^{\min(a,b)}  \: ,
\end{align}
with $\widetilde{M} = M \left( 1 + \sum_{l=0}^{d-1} \cH^{l} \left( \fS_l   \right) \right)$.\\
In the particular case where $\fS = \emptyset$, we also have the uniform estimates: for all $x \in S$,
\begin{equation} \label{eqThetaNHolder}
 \E \left[ | \theta_{\delta_N,N}(x) - \theta(x) | \right] \leq M  \delta_N^{\min(a,b)} \quad \text{and} \quad \E \left[  \| \sigma_{\delta_N,N}(x) - \Pi_{T_x S} \| \right] \leq M  \delta_N^{\min(a,b)} \: .
\end{equation}

\end{proposition}

\begin{proof}
 We apply Proposition~\ref{propDensityPointwiseCv} and Proposition~\ref{cvnudel} so that
 \begin{align*}
  \E \left[ \int_D | \theta_{\delta_N,N} - \theta | \: d \mu \right] & \leq  \int_D \E \left[ | \theta_{\delta_N,N} - \theta_{\delta_N} | \right] \: d \mu  + \int_D | \theta_{\delta_N} - \theta | \: d \mu \\
  & \leq M \mu(D) \left( \frac{1}{\sqrt{N \delta_N^d}} + \delta_N^{\min(a,b)} \right) + M  \sum_{l=0}^{d-1} \delta_N^{d-l} \: \cH^l (\fS_l \cap D^{2C\delta_N})
 \end{align*}
 and we infer \eqref{eqThetaNL1D} (first line) since we recall that with $\delta_N = N^{\displaystyle -\frac{1}{d + 2 \min (a,b)}}$, we have $\displaystyle \frac{1}{\sqrt{N \delta_N^d}} = \delta_N^{\min(a,b)}$.\\
Similarly applying Proposition~\ref{propCVSigmaN} and Proposition~\ref{propCvHolderW}, we obtain
\begin{align*}
  \E \left[ \int_D \| \sigma_{\delta_N,N}(x) - \Pi_{T_x S} \| \: d \mu \right] & \leq  \int_D \E \left[ \| \sigma_{\delta_N,N} - \sigma_{\delta_N,\delta_N} \| \right] \: d \mu  + \int_D \| \sigma_{\delta_N,\delta_N}(x) - \Pi_{T_x S} \| \: d \mu \\
  & \leq M \mu(D) \left( \frac{2}{\sqrt{N \delta_N^d}} + 2 \delta_N^{\min(a,b)} \right) + M  \sum_{l=0}^{d-1} \delta_N^{d-l} \: \cH^l (\fS_l \cap D^{4C \delta_N})
 \end{align*}
 and we conclude the proof of \eqref{eqThetaNL1D}. The proof of \eqref{eqThetaNHolder} is a straightforward application of the aforementioned Propositions~\ref{propDensityPointwiseCv},~\ref{cvnudel}~\ref{propCVSigmaN} and~\ref{propCvHolderW}
\end{proof}

\bibliographystyle{alpha}
\bibliography{bibli.bib}

\end{document}